\documentclass[12pt, reqno]{amsart}
\usepackage{amssymb,amsthm,amsmath}
\usepackage[numbers,sort&compress]{natbib}
\usepackage{amssymb,amsmath}
\usepackage{amsfonts}
\usepackage{mathrsfs}
\usepackage{latexsym}
\usepackage{amssymb}
\usepackage{amsthm}
\usepackage{indentfirst}
\date{March 7, 2017}

\let\oldsection\section
\renewcommand\section{\setcounter{equation}{0}\oldsection}

\newtheorem{corollary}{Corollary}[section]
\newtheorem{theorem}{Theorem}[section]
\newtheorem{lemma}{Lemma}[section]
\newtheorem{proposition}{Proposition}[section]
\newtheorem{definition}{Definition}[section]

\newtheorem{remark}{Remark}[section]

\begin{document}

\title[Primitive equations horizontal viscosity vertical diffusivity]{Global well-posedness of the 3D primitive equations with horizontal viscosity and vertical diffusivity}

\author{Chongsheng~Cao}
\address[Chongsheng~Cao]{Department of Mathematics, Florida International University, University Park, Miami, FL 33199, USA}
\email{caoc@fiu.edu}

\author{Jinkai~Li}
\address[Jinkai~Li]{Department of Mathematics, The Chinese University of Hong Kong, Shatin, N.T., Hong Kong}
\email{jklimath@gmail.com}

\author{Edriss~S.~Titi}
\address[Edriss~S.~Titi]{
Department of Mathematics, Texas A\&M University, 3368 TAMU, College Station,
TX 77843-3368, USA. ALSO, Department of Computer Science and Applied Mathematics, Weizmann Institute of Science, Rehovot 76100, Israel.}
\email{titi@math.tamu.edu and edriss.titi@weizmann.ac.il}



\begin{abstract}
In this paper, we consider the 3D primitive equations of oceanic and atmospheric
dynamics with only horizontal eddy viscosities in the horizontal
momentum equations and only
vertical diffusivity in the temperature equation. Global well-posedness of strong solutions is established for any initial data such that the initial horizontal velocity $v_0\in H^2(\Omega)$ and the initial temperature $T_0\in H^1(\Omega)\cap L^\infty(\Omega)$ with $\nabla_HT_0\in L^q(\Omega)$,
for some $q\in(2,\infty)$. Moreover, the strong solutions
enjoy correspondingly more regularities if the initial temperature belongs to $H^2(\Omega)$. The main difficulties are the absence of the
vertical viscosity and the lack of the horizontal diffusivity, which,
interact with each other, thus causing
the ``\,mismatching\," of regularities between the horizontal momentum and temperature equations. To handle this ``mismatching" of regularities, we introduce several auxiliary
functions, i.e., $\eta, \theta, \varphi,$ and $\psi$ in the paper,
which are the horizontal curls or some appropriate combinations of the temperature with the horizontal divergences of the horizontal velocity $v$ or its vertical derivative $\partial_zv$. To overcome
the difficulties caused by the absence of the horizontal diffusivity,
which leads to the requirement of some $L^1_t(W^{1,\infty}_\textbf{x})$-type a priori estimates on $v$, we decompose the velocity into the ``temperature-independent" and temperature-dependent parts
and deal with them in different ways, by using the logarithmic Sobolev inequalities of the Br\'ezis-Gallouet-Wainger and Beale-Kato-Majda types, respectively. Specifically, a logarithmic Sobolev inequality of the limiting type, introduced in our previous work \cite{CAOLITITI3}, is used, and a new logarithmic type Gronwall inequality is exploited.
\end{abstract}

\maketitle

\tableofcontents
\allowdisplaybreaks

\section{Introduction}

\label{sec1}
The incompressible primitive equations form a fundamental block in models of oceanic and atmospheric
dynamics, see, e.g., the books Lewandowski \cite{LEWAN}, Majda \cite{MAJDA}, Pedlosky \cite{PED}, Vallis \cite{VALLIS}, Washington--Parkinson \cite{WP}, and Zeng \cite{ZENG}.
The primitive equations are derived from the Navier-Stokes equations by applying the Boussinesq and
hydrostatic approximations. The hydrostatic approximation is based on the fact that the
vertical scale of the the ocean and atmosphere is much smaller than the horizontal ones, and its mathematical justification, by taking
small aspect ratio limit, was carried out by Az\'erad--Guill\'en \cite{AZGU} in the framework of weak solutions and recently by Li--Titi \cite{LITITIHYDRO} in the framework of strong solutions; moreover, the strong convergence rates were also obtained in \cite{LITITIHYDRO}.
In the oceanic and atmospheric dynamics, due to the strong horizontal turbulent mixing, the horizontal viscosity is much stronger than the vertical viscosity and the vertical viscosity is very weak and often neglected.

In this paper, we consider the following incompressible
primitive equations, which have only horizontal viscosities and vertical diffusivity
\begin{eqnarray}
&\partial_tv+(v\cdot\nabla_H)v+w\partial_zv+\nabla_Hp-\Delta_H v+f_0\overrightarrow{k}\times
v=0,\label{1.1}\\
&\partial_zp+T=0,\label{1.2}\\
&\nabla_H\cdot v+\partial_zw=0,\label{1.3}\\
&\partial_tT+v\cdot\nabla_H T+w\partial_zT-\partial_z^2T=0,\label{1.4}
\end{eqnarray}
where the horizontal velocity $v=(v^1,v^2)$, the vertical velocity $w$, the temperature
$T$ and the pressure $p$ are the unknowns, and $f_0$ is the Coriolis parameter. In this paper, we use the notations $\nabla_H=(\partial_x,\partial_y)$
and $\Delta_H=\partial_x^2+\partial_y^2$ to denote the horizontal gradient and the
horizontal Laplacian, respectively. Here, the term $\overrightarrow{k}\times v$ is understood as the first two components of the vector product of $\overrightarrow{k}=(0,0,1)$ with $(v^1,v^2,0)$, i.e., $\overrightarrow{k}\times v=(-v^2,v^1)$.

The first systematically mathematical studies of the primitive equations were carried out in 1990s by Lions--Temam--Wang \cite{LTW92A,LTW92B,LTW95}, where they considered the systems with both
full viscosities and full diffusivity, and established the global
existence of weak solutions; however, the uniqueness of
weak solutions is still an open question, even
for the two-dimensional case. Note that this is different from the
incompressible Navier-Stokes equations, as it is well-known that
the weak solutions to the two-dimensional incompressible
Navier-Stokes equations are unique (see, e.g., Constantin--Foias \cite{CONFOINSBOOK}, Ladyzhenskaya \cite{LADYZHENSKAYA} and Temam \cite{TEMNSBOOK}, and even in the framework of the three-dimensional
Navier-Stokes equations, see Bardos et al.\,\cite{BLNNT}). However, we would like to point out that, though the general uniqueness of weak solutions to
the primitive equations is still unknown, some particular cases
have been solved, see \cite{BGMR03,KPRZ,PTZ09,TACHIM,LITITIUNIQ},
and in particular, it is proved in \cite{LITITIUNIQ} that
weak solutions, with bounded initial data, to the primitive equations are unique, as long as the discontinuity of the initial data is
sufficiently small. Remarkably, different from the three-dimensional
Navier-Stokes equations, global existence and uniqueness of strong
solutions to the three-dimensional primitive equations has already
been known since the breakthrough work by Cao--Titi \cite{CAOTITI2}.
This global existence of strong solutions to the primitive equations
were also proved later by Kobelkov \cite{KOB06} and
Kukavica--Ziane\cite{KZ07B}, by using some different approaches,
see also Hieber--Kashiwabara \cite{HIEKAS} and
Hieber--Hussien--Kashiwabara \cite{HIEHUSKAS} for some generalizations
in the $L^p$ settings.

Note that in all the papers mentioned in the previous paragraph, the systems in question are assumed to have full viscosities
in the horizontal momentum equations and full diffusivity in the
temperature equation. As stated in the previous paragraph, the
primitive equations with both full viscosities and full diffusivity
have a unique global strong solution, which is smooth away from
the initial time. However, on the other hand, it has already
been proven that smooth solutions to the inviscid
primitive equations, with or without coupling to the temperature equation, can develop singularities in finite time,
see Cao et al.\,\cite{CINT} and Wong \cite{TKW}. Comparing
these two kind results of the two extreme cases, i.e., global
existence for the primitive equations with both
full viscosities and full diffusivity and blowup in finite time
for the inviscid primitive equations, it is natural
for us to consider the intermediate cases, i.e., the primitive
equations with partial viscosities or partial diffusivity, and
to ask of whether the solutions exist globally in time or blow up
in finite time for these intermediate cases.

There has been several works concerning the mathematical studies on
the primitive equations with partial viscosities or partial
diffusivity. It has been proved by Cao--Titi \cite{CAOTITI3}
and Cao--Li--Titi \cite{CAOLITITI1,CAOLITITI2} that the primitive
equations with full viscosities and with either horizontal
or vertical diffusivity have a unique global strong solution.
It turns out that the vertical viscosity is even
not necessary for the global well-posedness of the primitive
equations. In fact, it was proved by Cao--Li--Titi \cite{CAOLITITI3}
that strong solutions are unique and exit globally in time for
the primitive equations with only horizontal viscosity and only horizontal diffusivity for any initial data in $H^2$ (see Cao--Li--Titi \cite{CAOLITITIH1} for some generalization of the result in \cite{CAOLITITI3}).
We would like to point out that there is a notable difference between the arguments for the primitive
equations with full viscosities and those for the case of only horizontal
viscosity: for the primitive equations with full viscosities, the
a priori $L^\infty(L^q)$ estimate on $v$ for some $q\in(3,\infty)$
is sufficient for establishing higher order estimates, but it is
not the case for the primitive equations with only horizontal
viscosity. In fact, as pointed out in \cite{CAOLITITI3},
due to the absence of the vertical viscosity, in order to obtain
higher order energy estimates, one has in some sense
to get the a priori $L^2(L^\infty)$ estimate on $v$. The idea
used in \cite{CAOLITITI3} to overcome this difficulty is to carry
out the precise growth with respect to $q$ of the $L^q$ norms of $v$
for $q\in[4,\infty)$,
and connect the $L^\infty$ norm of $v$ with such precise growth,
by an $N$-dimensional logarithmic Sobolev embedding inequality,
which states that the $L^\infty$ norm can be dominated by some
appropriate growth in $q$ of estimates for the $L^q$ norms, up to some
logarithmic of the higher order norms.

In this paper, we continue to study the primitive equations with
partial viscosities or partial diffusivity. Recall that the case
with horizontal viscosity and horizontal diffusivity has
been investigated in \cite{CAOLITITI3,CAOLITITIH1}, as a counterpart, we consider in the current paper the case with only horizontal viscosity,
but with vertical diffusivity, i.e., system (\ref{1.1})--(\ref{1.4}).
The aim of this paper is to show that system (\ref{1.1})--(\ref{1.4}), subject to appropriate boundary and initial conditions, is global well-posed.

We consider system (\ref{1.1})--(\ref{1.4}) on the domain $\Omega:=M\times(-h,h)$, with $M=(0,1)\times(0,1)$, and complement it with the following boundary and initial conditions
\begin{eqnarray}
  &v, w, p, T \mbox{ are periodic in }x,y,z,\label{1.5}\\
  &v\mbox{ and }p \mbox{ are even in }z, \quad w\mbox{ and }T\mbox{ are odd in }z,\label{1.6}\\
  &(v, T)|_{t=0}=(v_0,T_0). \label{1.7}
\end{eqnarray}
Note that condition (\ref{1.6}) is preserved by system (\ref{1.1})--(\ref{1.4}), as long as it is satisfied initially. Also, we remark that no initial condition is imposed on $w$. This is because there is no dynamical equation for $w$, and in fact, $w$ is uniquely determined by the incompressibility condition (\ref{1.3}). 

Conspicuously, we observe that the periodic and symmetry boundary conditions
(\ref{1.5})--(\ref{1.6}) on the domain $M\times(-h, h)$ are equivalent to
the  physical boundary conditions of no-permeability  and stress-free at the solid physical boundaries $z=-h$ and $z=0$ in the sub-domain
$M\times(-h,0)$, namely:
\begin{eqnarray}
&  v, w, p, T \mbox{ are periodic in }x\mbox{ and }y,\label{PBC1}\\
&(\partial_zv,w)|_{z=-h,0}=0,\quad T|_{z=-h,0}=0.\label{PBC2}
\end{eqnarray}
This equivalence between the two problems can be easily  achieved by suitable reflections and extensions of the solutions.
More precisely, if $(v, w, p, T)$ is a strong solution (see Definition \ref{def1.1}, below, for the definition of strong solutions) to system
(\ref{1.1})--(\ref{1.4}) on the domain $M\times(-h,h)$, subject to
(\ref{1.5})--(\ref{1.7}), then the
restriction of $(v,w,p,T)$ to the sub-domain $M\times(-h, 0)$ is also
a strong solution to the same system but on
the sub-domain, subject to
(\ref{1.7}) and (\ref{PBC1})--(\ref{PBC2}); and, conversely,
if $(v,w,p,T)$ is a
strong solution to system (\ref{1.1})--(\ref{1.4}) on the sub-domain $M\times(-h,0)$, subject to
(\ref{1.7}) and (\ref{PBC1})--(\ref{PBC2}), then by extending
$v, w, p$ and $T$ to the larger domain $M\times(-h,h)$,
respectively, even, odd, even and odd
with respect to $z$, $(v,w,p,T)$ is also a strong solution to the same
system but on the larger domain, subject to (\ref{1.5})--(\ref{1.7}).

Using equation (\ref{1.2}), the pressure $p$ can be represented by
$$
p(x,y,z,t)=p_s(x,y,t)-\int_{-h}^zT(x,y,\xi,t)d\xi,
$$
for unknown ``surface pressure" $p_s$. Using this representation, system (\ref{1.1})--(\ref{1.4}) can be rewritten as
\begin{eqnarray}
&\partial_tv+(v\cdot\nabla_H)v+w\partial_zv-\Delta_Hv
\nonumber\\
&+f_0\overrightarrow{k}\times v+\nabla_H\left(p_s(x,y,t)-\int_{-h}^zT(x,y,\xi,t)d\xi\right)=0,\label{main1}\\
&\nabla_H\cdot v+\partial_zw=0,\label{main2}\\
&\partial_tT+v\cdot\nabla_HT+w\partial_zT-\partial_z^2T=0.\label{main3}
\end{eqnarray}
Concerning the boundary and initial conditions, we can now drop the boundary
conditions for the pressure from (\ref{1.5})--(\ref{1.7}), since it is hidden
in the above formulation, in other words,
the boundary and initial conditions now read as
\begin{eqnarray}
  &v, w, T \mbox{ are periodic in }x,y,z,\label{bc1}\\
  &v\mbox{ is even in }z, \quad w\mbox{ and }T\mbox{ are odd in }z,\label{bc2}\\
  &(v, T)|_{t=0}=(v_0,T_0). \label{ic}
\end{eqnarray}

By the aid of the periodic boundary condition (\ref{bc1}) and the divergence free condition (\ref{main2}), it is obviously that
\begin{equation}
  \label{dvfree}
  \int_{-h}^h\nabla_H\cdot v(x,y,z,t)dz=0,
\end{equation}
for any $(x,y)\in M$. By the periodic and symmetry conditions (\ref{bc1})
and (\ref{bc2}), one has $w|_{z=-h}=w|_{z=h}=-w|_{z=-h}=0$ and, as a result, using (\ref{main2}), $w$ can be represented in $v$ as
\begin{equation}
  w(x,y,z,t)=-\int_{-h}^z\nabla_H\cdot v(x,y,\xi,t)d\xi. \label{w}
\end{equation}
On the other hand side, (\ref{w}) obviously implies (\ref{main2}) and, furthermore, (\ref{dvfree}) and the conditions for $v$ as stated in (\ref{bc1})--(\ref{bc2}) imply those for $w$ as stated in (\ref{bc1})--(\ref{bc2}).

On account of what we stated in the previous paragraph, with the aid of the expression (\ref{w}), one can replace (\ref{main2}) by (\ref{dvfree}) and drop the conditions for $w$ in (\ref{bc1}) and (\ref{bc2}), without changing the system. In other words, system (\ref{main1})--(\ref{main3}), subject to the boundary and initial conditions (\ref{bc1})--(\ref{ic}), is equivalent to the following system
\begin{eqnarray}
&\partial_tv+(v\cdot\nabla_H)v+w\partial_zv-\Delta_Hv\nonumber\\
&+f_0\overrightarrow{k}\times
v+\nabla_H\left(p_s(x,y,t)-\int_{-h}^zT(x,y,\xi,t)d\xi\right)=0,\label{MAIN1}\\
&\int_{-h}^h\nabla_H\cdot v(x,y,z,t)dz=0,\label{MAIN2}\\
&\partial_tT+v\cdot\nabla_HT+w\partial_zT-\partial_z^2T=0,\label{MAIN3}
\end{eqnarray}
with $w$ given by (\ref{w}),
subject to the boundary and initial conditions
\begin{eqnarray}
  &v, T \mbox{ are periodic in }x,y,z,\label{BC1}\\
  &v\mbox{ and }T\mbox{ are even and odd in }z,\mbox{ respectively},\label{BC2}\\
  &(v, T)|_{t=0}=(v_0,T_0). \label{IC}
\end{eqnarray}

Applying the operator $\text{div}_H$ to equation (\ref{MAIN1}) and
integrating the resulting equation with respect to $z$ over $(-h, h)$,
one can see that $p_s(x,y,t)$ satisfies the following (see Appendix A
for the details)
\begin{equation}
  \label{ps}
  \left\{
  \begin{array}{l}
    -\Delta_Hp_s=\frac{1}{2h}\nabla_H\cdot\int_{-h}^h \left(\nabla_H\cdot(v\otimes v)+f_0\overrightarrow{k}\times v-\int_{-h}^z\nabla_HTd\xi\right)dz,\\
    \int_Mp_s(x,y,t)dxdy=0,\quad p_s\mbox{ is periodic in }x\mbox{ and }y.
  \end{array}
  \right.
\end{equation}
Here the condition $\int_Mp_s(x,y,t)dxdy=0$ is imposed to guarantee the uniqueness of such $p_s$.

Before stating our main results, let's introduce some necessary notations
and give the definitions of strong solutions. Throughout this paper,
for $1\leq q\leq\infty$, we use $L^q(\Omega), L^q(M)$ and $W^{m,q}(\Omega), W^{m,q}(M)$ to denote the standard Lebesgue and Sobolev spaces,
respectively. For $q=2$, we use $H^m$ instead of $W^{m,2}$. For simplicity,
we still use the notations $L^p$ and $H^m$ to denote the $N$-product
spaces $(L^p)^N$ and $(H^m)^N$, respectively. We always use $\|u\|_p$
to denote the $L^p(\Omega)$ norm of $u$, while use $\|f\|_{p,M}$ to denote
the $L^p(M)$ norm of $f$. For shortening the
expressions, we sometimes use $\|(f_1,f_2,\cdots,f_n)\|_X$ to denote
the sum $\sum_{i=1}^n\|f_i\|_X$.

We introduce the following functions which will play crucial roles in this paper
\begin{eqnarray}
  u=\partial_zv, \quad\theta=\nabla_H^\perp\cdot v,\quad\eta=\nabla_H\cdot v+\int_{-h}^zTd\xi-\frac{1}{2h}\int_{-h}^h\int_{-h}^zTd\xi dz,
  \label{E1}
\end{eqnarray}
where $\nabla_H^\perp=(-\partial_y, \partial_x)$. As it will be seen later, these functions are introduced to overcome the ``mismatching" of regularities between the horizontal momentum equations and the temperature equation.

\begin{definition}\label{def1.1}
Given a positive time $\mathcal T$. Let $v_0\in H^2(\Omega)$ and
$T_0\in H^1(\Omega)$, with $\int_{-h}^h\nabla_H\cdot v_0(x,y,z) dz=0$
and $\nabla_HT_0\in L^q(\Omega)$, for some $q\in(2,\infty)$, be two
periodic functions, such that they are even and odd in $z$, respectively.
A pair $(v,T)$ is called a strong solution to system
(\ref{MAIN1})--(\ref{IC}) on $\Omega\times(0,\mathcal T)$ if

(i) $v$ and $T$ are periodic in $x,y,z$, and they are even and odd in $z$, respectively;

(ii) $v$ and $T$ have the regularities
\begin{eqnarray*}
&&v\in L^\infty(0,\mathcal T; H^2(\Omega))\cap C([0,\mathcal T];H^1(\Omega)),\quad\partial_tv\in L^2(0,\mathcal T; H^1(\Omega)),\\
&&T\in L^\infty(0,\mathcal T; H^1(\Omega)\cap L^\infty(\Omega)) \cap C([0,\mathcal T]; L^2(\Omega)), \quad\partial_t T\in L^2(0,\mathcal T; L^2(\Omega)), \\
&&(\nabla_H\partial_zv,\partial_zT)\in L^2(0,\mathcal T; H^1(\Omega)),\quad\nabla_HT\in L^\infty(0,\mathcal T; L^q(\Omega)),\\
&&\eta\in L^2(0,\mathcal T; H^2(\Omega)),\quad \theta\in L^2(0,\mathcal T; H^2(\Omega));
\end{eqnarray*}

(iii) $v$ and $T$ satisfy equations (\ref{MAIN1})--(\ref{MAIN3}) a.e.\,in $\Omega\times(0,\mathcal T)$, with $w$ and $p_s$ given by (\ref{w}) and (\ref{ps}), respectively, and satisfy the initial condition (\ref{IC}).
\end{definition}

\begin{remark}
\label{remark1.1}
(i) The regularities in Definition \ref{def1.1} seem a little bit
nonstandard. This is caused by the ``mismatching" of regularities between
the horizontal momentum equation (\ref{MAIN1}) and the temperature equation
(\ref{MAIN3}): a term involving the horizontal derivatives of the temperature
appears in the horizontal momentum equation, but it is only in the vertical
direction that the temperature has dissipation. More precisely, though one
can obtain the regularity that $\nabla_H\partial_zv\in L^2(0,\mathcal T; H^1(\Omega))$, which is included in Definition (\ref{def1.1}), we have no
reason to ask for the regularity that $\nabla_H^2v\in L^2(0,\mathcal T; H^1(\Omega))$, under the assumption on the initial data in Definition
\ref{def1.1}. In fact, recalling the regularity theory for parabolic
system, and checking the  horizontal momentum equation (\ref{MAIN1}), the
regularity that $\nabla_H^2v\in L^2(0,\mathcal T; H^1(\Omega))$ appeals to
somehow $\nabla_H^2T\in L^2(\Omega\times(0,\mathcal T))$; however, this last
requirement need not to be fulfilled, because we only have the smoothing effect
in the vertical direction for the temperature.

(ii) The ``mismatching" of regularities as stated in (i)
does not occur in the system
considered in our previous work \cite{CAOLITITI3}, i.e., the system with only both horizontal viscosities and horizontal diffusivity, because, in this case, the horizontal diffusivity in the temperature equation provides the required regularity that $\nabla_H^2T\in L^2(\Omega\times(0,\mathcal T))$.

(iii) As stated in (i), one can not expect such regularity that $\nabla_H^2v\in L^2(0,\mathcal T; H^1(\Omega))$. However, with the help of $\eta$ and $\theta$ in (\ref{E1}), one can expect that some appropriate combinations of $\nabla_Hv$ and $T$ can indeed have second order spatial derivatives, that is $(\eta,\theta)\in L^2(0,\mathcal T; H^2(\Omega))$, as included in Definition \ref{def1.1}.
\end{remark}

\begin{definition}
A pair $(v,T)$ is called a global strong solution to system (\ref{MAIN1})--(\ref{MAIN3}), subject to the boundary and initial conditions (\ref{BC1})--(\ref{IC}), if it is a strong solution on $\Omega\times(0,\mathcal T)$ for any $\mathcal T\in(0,\infty)$.
\end{definition}

The main result of this paper is the following global well-posedness result.

\begin{theorem}\label{thm1}
Let $v_0\in H^2(\Omega)$ and $T_0\in H^1(\Omega)\cap L^\infty(\Omega)$,
with $\int_{-h}^h\nabla_H\cdot v_0(x,y,z) dz=0$ and $\nabla_HT_0\in L^q(\Omega)$, for some $q\in(2,\infty)$, be two periodic functions, such
that they are even and odd in $z$, respectively.
Then system (\ref{MAIN1})--(\ref{MAIN3}), subject to the boundary and initial conditions (\ref{BC1})--(\ref{IC}), has a unique global strong solution $(v,T)$, which is continuously depending on the initial data.

If we assume, in addition, that $T_0\in H^2(\Omega)$, then $(v,T)$ obeys the following additional regularities
\begin{eqnarray*}
&&T\in L^\infty(0,\mathcal T; H^2(\Omega))\cap C(0,\mathcal T; H^1(\Omega)),\quad\partial_tT\in L^2(0,\mathcal T; H^1(\Omega)),\\
&&\nabla_Hv\in L^2(0,\mathcal T; H^2(\Omega)),\quad \partial_zT\in L^2(0,\mathcal T; H^2(\Omega)),
\end{eqnarray*}
for any time $\mathcal T\in(0,\infty)$.
\end{theorem}

\begin{remark}
Generally, if we imposed more regularities on the initial data, then one can expect more  regularities of the strong solutions, and in particular, the strong solution will belong to $C^\infty(\overline\Omega\times[0,\infty))$, as long as the initial datum lies in $C^\infty(\overline\Omega)$. However, one can not expect that the solutions have as high regularities as desired, if the initial data are not accordingly smooth enough.
\end{remark}

\begin{remark}
Thanks to Theorem \ref{thm1}, and recalling the results in \cite{CAOTITI2,CAOTITI3,CAOLITITI1,CAOLITITI2,CAOLITITI3,CAOLITITIH1}, one can conclude that the primitive equations are globally well-posed, as long as one has the horizontal viscosity and either horizontal or
vertical diffusivity.
\end{remark}


The main difficulties for the mathematical analysis of system (\ref{MAIN1})--(\ref{MAIN3}) come from three aspects: the absence of the vertical viscosity in the horizontal momentum equation, the lack of the horizontal diffusivity in the temperature equation, and the ``mismatching" of regularities between the horizontal momentum equations and the temperature equation caused by the interaction between the lack of the vertical viscosity and the absence of the horizontal diffusivity.
Concerning the difficulties caused by the absence of the
vertical viscosity, \cite{CAOLITITI3,CAOLITITIH1} provide us
with some ideas. As indicated in \cite{CAOLITITI3,CAOLITITIH1},
the absence
of the vertical viscosity forces us to estimate $\|v\|_\infty^2$,
which appears as factors in the
energy inequalities. To obtain this estimate, similar to \cite{CAOLITITI3,CAOLITITIH1},
we estimate the precise growth in $q$
of the $L^q$ norms of $v$ (see Proposition
\ref{prop3.2}), based on which, by applying a logarithmic type Sobolev
inequality (see Lemma \ref{log}), we can control the $L^\infty$ norm
of $v$ by logarithm of high order norms. However, in the current case,
because of the
``mismatching" of regularities between the horizontal momentum
equations and the temperature equation (recall Remark
\ref{remark1.1}
(i)), we are not able to obtain the appropriate estimates in
the same way as in \cite{CAOLITITI3,CAOLITITIH1}. Note that in \cite{CAOLITITI3,CAOLITITIH1} all energy estimates for the
derivatives of the velocity are carried out through multiplying
the corresponding testing functions to the momentum equations
directly; however, for the current case, when working on the
energy estimates for the horizontal derivatives of the velocity,
it is inappropriate to use the momentum equations as the tested
ones. To see this, let's take
the $L^\infty_t(H^1_\textbf{x})$ kind estimate as example: if trying
to use the momentum equation to get the $L^\infty_t(L^2_\textbf{x})$
estimate on $\nabla_Hv$, one may multiply the momentum equation by
$-\Delta_Hv$
and, thus, requires the a priori $L^2_t(L^2_\textbf{x})$ type
estimate
on $\nabla_HT$, which is obviously not guaranteed by the system,
as we only have the vertical diffusivity in the temperature equation.
To overcome this kind of difficulties, we consider
the horizontal curl and some appropriate combination
of the temperature with the horizontal divergence of $v$ or
its derivatives, which prove to have better regularities than the
horizontal derivatives of $v$ or
its derivatives. In other words, the estimates on the horizontal derivatives
of the velocity are achieved indirectly through the corresponding
estimates on the horizontal curls and some appropriate combinations of
the temperature with the horizontal divergences.

For the a priori $L^\infty_t(H_{\textbf x}^1)$ type estimate on $v$,
recalling the ideas mentioned above, it is achieved by carrying out the $L^\infty_t(L^2_{\textbf x})$ type energy estimates on $(u,\eta,\theta)$,
rather than on $(u,\nabla_Hv)$ directly, and using the precise $L^q$
estimates
on $v$ to dominate the main part of $\|v\|_\infty$. These are carried out
in Proposition \ref{propapriH1} and Corollary \ref{apriH1}.
Note that the following fact plays an important role in proving Corollary \ref{apriH1}: inequality
$$
A'(t)+B(t)\leq CA(t)\log B(t)+``\text{other terms}"
$$
guarantees the boundness of $A(t)$ globally in time.
The above inequality is a special case
of the general logarithmic type Gronwall inequality stated in Lemma \ref{LogGron1}.

Based on the a prior $L^\infty_t(H^1_{\textbf x})$ type estimate on $v$,
one can obtain the a priori $L^\infty_t(H^1_{\textbf x})$ type estimate
on $u$. Again, because of the same reason as before, this a priori estimate
is achieved through the $L^\infty_t(L^2_{\textbf x})$ estimate on $(\partial_zu, \varphi, \psi)$, where
\begin{equation}\label{E2}
\varphi=\nabla_H\cdot u+T,\quad \psi=\nabla_H^\perp\cdot u,
\end{equation}
rather than directly on $(\partial_zu, \nabla_Hu)$.

Some higher order a priori estimates, especially those on the derivatives
of the temperature, are still needed to ensure the global well-posedness.
When working on the energy inequalities for the horizontal
derivatives of $T$, caused by the absence of the horizontal diffusivity in
the temperature equation, one has to appeal to somehow $L^\infty$ estimate
on $\nabla_Hv$ to deal with the worst term $\int_\Omega|\nabla_Hv||\nabla_HT|^qdxdydz.$ To deal with this term,
we decompose the velocity into a ``temperature-independent" part and another temperature-dependent part and then deal with them in different ways, by
using the logarithmic Sobolev inequalities of the Br\'ezis-Gallouet-Wainger and Beale-Kato-Majda types, respectively.
The resulting corresponding energy inequalities are of the type
$$
A'(t)+B(t)\leq Cn(t)A(t)\log B(t)+``\text{other terms}",
$$
where $n$ is a locally integrable function on $[0,\infty)$.
Note that this inequality does not necessary guarantee the boundness of
the quantity $A$, in general; however, if it happens that the following
additional relationship holds
$$
n(t)\leq CA^\alpha(t)
$$
for some positive number $\alpha$, then it indeed implies the boundness
of the quantity $A$, see Lemma \ref{LogGron1}. Fortunately, it is the
case in our higher order energy inequality, and therefore, we are able
to obtain the a priori higher order estimates, and furthermore the
global existence of strong solutions. The additional regularities stated
in the theorem follow from the energy inequality for the second order
derivatives of $T$, which is somehow standard.

The rest of this paper is arranged as follows: in the next section, section \ref{sec2}, we collect some preliminaries which will be used throughout the paper. Section \ref{sec3} is the main part of this paper, in which, by using
the ideas obtained above, we establish several a priori estimates
for a regularized system, and the a priori estimates are independent of the regularization parameters. In section \ref{sec4}, based on the a priori
estimates obtained in section \ref{sec3}, we give the proof of Theorem \ref{thm1}.

Throughout this paper, the letter $C$ denotes a general positive constant, which may vary from line to line.

\section{Preliminaries}
\setcounter{tocdepth}{1}
\label{sec2}
In this section, we collect some preliminary results which will be used in the rest of this paper.

\begin{lemma}[see Lemma 2.1 in \cite{CAOLITITI3}] \label{lad}
The following inequality holds:
\begin{align*}
&\int_M\left(\int_{-h}^h|\phi |dz\right)\left(\int_{-h}^h|\varphi \psi |dz\right)dxdy\\
\leq&C\|\phi\|_2\|\varphi\|_2^{\frac{1}{2}}\left(\|\varphi\|_2^{\frac{1}{2}}+\|\nabla_H\varphi\|_{2}^{\frac{1}{2}}
\right)\|\psi\|_2^{\frac{1}{2}}\left(
\|\psi\|_2^{\frac{1}{2}}+\|\nabla_H\psi\|_2^{\frac{1}{2}}\right),
\end{align*}
for every $\phi,\varphi,$ and $\psi$ such that the right hand sides make sense and are finite.
\end{lemma}

\begin{lemma}\label{ladlemma}
We have the following inequalities
\begin{equation}
  \int_M\left(\int_{-h}^h|\phi|dz\right)\left(\int_{-h}^h|\varphi||\psi|dz\right) dxdy
  \leq\left(\int_{-h}^h\|\phi\|_{4,M}dz\right)\left(\int_{-h}^h\|\varphi\|_{4,M}^2dz\right)^{\frac{1}{2}}\|\psi\|_2, \label{ineqlad}
\end{equation}
and
\begin{align}
  &\int_M\left(\int_{-h}^h|\phi||\varphi|dz\right) \left(\int_{-h}^h|\psi|dz\right)dxdy\nonumber\\
  \leq&\left(\int_{-h}^h\|\phi\|_{4,M}^2dz\right)^{\frac{1}{2}}\left(\int_{-h}^h\|\varphi\|_{4,M}^2dz\right)^{\frac{1}{2}} \left(\int_{-h}^h\|\psi\|_{2,M}dz\right),\label{ineqlad1}
\end{align}
for any functions $\phi, \varphi$ and $\psi$, such that the quantities on the right-hand sides make sense and are finite.
\end{lemma}

\begin{proof}
By the H\"older and Minkowski inequalities, we have
\begin{eqnarray*}
  &&\int_M\left(\int_{-h}^h|\phi|dz\right)\left( \int_{-h}^h|\varphi||\psi|dz\right)dxdy\nonumber\\
  &\leq&\int_M\left(\int_{-h}^h|\phi|dz\right)\left(\int_{-h}^h|\varphi|^2dz\right)^{\frac{1}{2}}\left(\int_{-h}^h|\psi|^2dz\right)^{\frac{1}{2}}dxdy
  \nonumber\\
  &\leq&\left[\int_M\left(\int_{-h}^h|\phi|dz\right)^4dxdy\right]^{\frac{1}{4}} \left[\int_M\left(\int_{-h}^h|\varphi|^2dz\right)^2dxdy\right]^{\frac{1}{4}} \|\psi\|_2\nonumber\\
  &\leq&\left(\int_{-h}^h\|\phi\|_{4,M}dz\right) \left(\int_{-h}^h\|\varphi\|_{4,M}^2dz\right)^{\frac{1}{2}}\|\psi\|_2,
\end{eqnarray*}
and
\begin{eqnarray*}
  &&\int_M\left(\int_{-h}^h|\phi||\varphi|dz\right)\left(\int_{-h}^h|\psi|dz \right)dxdy\\
  &\leq&\int_M\left(\int_{-h}^h|\phi|^2dz\right)^{\frac{1}{2}} \left(\int_{-h}^h|\varphi|^2dz\right)^{\frac{1}{2}}\left(\int_{-h}^h|\psi|dz\right)dxdy\\
  &\leq&\left[\int_M\left(\int_{-h}^h|\phi|^2dz\right)^2dxdy\right]^{\frac{1}{4}}\left[\int_M\left(\int_{-h}^h|\varphi|^2dz\right)^2dxdy\right] ^{\frac{1}{4}}\\
  &&\qquad\times\left[\int_M\left(\int_{-h}^h|\psi|dz\right)^2dxdy\right] ^{\frac{1}{2}}\\
  &\leq&\left(\int_{-h}^h\|\phi\|_{4,M}^2dz\right)^{\frac{1}{2}}\left(\int_{-h}^h\|\varphi\|_{4,M}^2dz\right)^{\frac{1}{2}} \left(\int_{-h}^h\|\psi\|_{2,M}dz\right),
\end{eqnarray*}
proving (\ref{ineqlad}) and (\ref{ineqlad1}).
\end{proof}

\begin{lemma}\label{lem2.3}
The following inequalities hold
  \begin{eqnarray*}
    \left(\int_{-h}^h\|f\|_{4,M}^2dz\right)^{\frac12}&\leq&C\left(\|f\|_2^{\frac12}
    \|\nabla_Hf\|_2^{\frac12}+\|f\|_2\right), \\
    \int_{-h}^h\|f\|_{4,M}^2dz&\leq&C\sqrt h\left(\|f\|_2^{\frac12}
    \|\nabla_Hf\|_2^{\frac12}+\|f\|_2\right),
  \end{eqnarray*}
for any function $f$ such that the right-hand sides make sense and are finite. As a consequence, by the Poincar\'e inequality, the following holds
  \begin{eqnarray*}
    \left(\int_{-h}^h\|\nabla_Hf\|_{4,M}^2dz\right)^{\frac12}&\leq&C\|\nabla_Hf\|_2^{\frac12}
    \|\nabla_H^2f\|_2^{\frac12}, \\
    \int_{-h}^h\|\nabla_Hf\|_{4,M}^2dz&\leq&C\sqrt h \|\nabla_Hf\|_2^{\frac12}
    \|\nabla_H^2f\|_2^{\frac12},
  \end{eqnarray*}
if moreover $f$ is periodic in $(x,y)$.
\end{lemma}

\begin{proof}
The conclusion follow easily from the H\"older and Ladyzhenskay inequalities and, thus, the proofs are omitted here.
\end{proof}

The following logarithmic Sobolev inequality, which links the $L^\infty$ norm in terms of the $L^q$ norms up to the logarithm of the high order norms. Some relevant inequalities can be found in \cite{CAOWU,CAOFARHATTITI,DANCHINPAICU}, where the two dimensional case are considered.

\begin{lemma}[Logarithmic Sobolev embedding inequality, see Lemma 2.2 in \cite{CAOLITITI3}]
  \label{log}
Let $F\in W^{1,p}(\Omega)$, with $p>3$, be a periodic function. Then the following inequality holds true
\begin{equation*}
  \|F\|_\infty\leq C_{ {p},\lambda}\max\left\{1,\sup_{r\geq2}\frac{\|F\|_r}{r^\lambda}\right\}
  \log^\lambda
  (\|F\|_{W^{1, {p}}(\Omega)}+e),
\end{equation*}
for any $\lambda>0$.
\end{lemma}

The logarithmic type Gronwall inequality stated and proved
in the following
lemma will be used in establishing the global a priori estimates
with critical nonlinearities. The first logarithmic type Gronwall
inequality in the same spirit as stated here was obtained by Li--Titi \cite{LITITIBOU}, see also Li--Titi \cite{LITITITROPMOIS} for some related inequalities.

\begin{lemma}[Logarithmic Gronwall inequality]\label{LogGron1}
Given $\mathcal T\in(0,\infty)$. Let $A$ and $B$ be two nonnegative measurable functions defined on $(0,\mathcal T)$, with $A$ is absolutely continuous on $(0,\mathcal T)$ and is continuous on $[0,\mathcal T)$, satisfying
$$
\frac{d}{dt}A+B\leq[\ell(t)+m(t)\log(A+e)+n(t)\log(A+B+e)](A+e)+f(t),
$$
where $\ell, m, n$, and $f$ are all nonnegative functions on $(0,\mathcal T)$ belonging to $L^1((0,\mathcal T))$. Assume further that there are two positive constants $K$ and $\alpha$, such that
$$
n(t)\leq K(A(t)+e)^\alpha
$$
for all $t\in(0,\mathcal T)$. Then, we have the following estimate
$$
A(t)+\int_0^tB(s)ds\leq (2Q(t)+1)e^{Q(t)}
$$
for all $t\in(0,\mathcal T)$, where
$$
Q(t)=e^{(\alpha+1)\int_0^t(m(s)+n(s))ds}\left(\log(A(0)+e)+\int_0^t(\ell(s)+f(s)+\log(2K) n(s))ds+t\right).
$$
\end{lemma}

\begin{proof}
Setting $A_1=A+e$ and $B_1=B+A+e$, then
\begin{align*}
  \frac{d}{dt}A_1+B_1=&\frac{d}{dt}A+B+A+e\\
  \leq&[\ell(t)+1+m(t)\log(A+e)+n(t)\log(A +B+e)](A+e)+f(t)\\
  =&(\ell(t)+1+m(t)\log A_1+n(t)\log B_1)A_1+f(t).
\end{align*}
Dividing both sides of the above inequality by $A_1$ yields
\begin{align*}
  \frac{d}{dt}\log A_1+\frac{B_1}{A_1}\leq&\ell(t)+1+\frac{f(t)}{A_1(t)}+m(t)\log A_1+m(t)\log B_1\\
  \leq&\ell(t)+1+f(t)+m(t)\log A_1+n(t)\log B_1.
\end{align*}
Noticing that $\log z\leq\log(z+1)\leq z$ for any $z\in(0,\infty)$, and recalling that $n(t)\leq K(A+e)^\alpha=KA_1^\alpha$, we deduce
\begin{align*}
  n(t)\log B_1=&n(t)\left(\log\frac{B_1}{2KA_1^{\alpha+1}}+(\alpha+1)\log A_1+\log(2K)\right)\\
  \leq&n(t)\left(\frac{B_1}{2KA_1^{\alpha+1}}+(\alpha+1)\log A_1+\log(2K)\right)\\
  \leq&\frac{B_1}{2A_1}+(\alpha+1)n(t)\log A_1+n(t)\log(2K).
\end{align*}
Therefore, one has
\begin{align*}
  \frac{d}{dt}\log A_1+\frac{B_1}{2A_1}\leq(m(t)+(\alpha+1)n(t))\log A_1+\ell(t)+1+f(t)+n(t)\log(2K),
\end{align*}
from which, by denoting $G(t)=\log A_1(t)+\int_0^t\frac{B_1(s)}{2A_1(s)}ds$, one obtains
$$
G'(t)\leq(m(t)+(\alpha+1)n(t))G(t)+\ell(t)+1+f(t)+n(t)\log(2K);
$$
and, thus,
\begin{align*}
G(t)\leq&e^{\int_0^t(m(s)+(\alpha+1)n(s))ds}\left(G(0)+\int_0^t( \ell(s)+f(s)+\log(2K) n(s)+1)ds\right)\\
\leq & e^{(\alpha+1)\int_0^t(m(s)+n(s))ds}\left(\int_0^t(\ell(s)+f(s)+\log(2K) n(s))ds+t\right)\\
&+e^{(\alpha+1)\int_0^t(m(s)+n(s))ds}\log(A(0)+e)=:Q(t).
\end{align*}
Recalling the definition of $G(t)$, it follows from the above estimate that
$$
A_1(t)\leq e^{G(t)}\leq e^{Q(t)},
$$
and further that
\begin{align*}
\int_0^tB_1(s)ds=&2\int_0^tA_1(s)\frac{B_1(s)}{2A_1(s)}ds\leq 2\sup_{0\leq s\leq t}A_1(s)\int_0^t\frac{B_1(s)}{2A_1(s)}ds\\
\leq& 2e^{Q(t)}G(t)\leq 2Q(t)e^{Q(t)}.
\end{align*}
Thanks to the above estimates and recalling the definitions of $A_1$ and $B_1$, the conclusion follows.
\end{proof}

\begin{remark}
(i) A special form of the logarithmic Gronwall inequality in Lemma \ref{LogGron1} reads as
\begin{equation}
\frac{d}{dt}A+B\leq A\log(A+B+e). \label{ADDEQ1}
\end{equation}
Note that this is essentially different from the classic logarithmic
Gronwall inequality like
$\frac{d}{dt}A+B\leq A\log(A+e).$ Noticing that the PDEs
with dissipation, the quantities represented by $B$ in the above
inequality usually have higher order norms than those by $A$ and,
therefore, compared with the using of the usual logarithmic
Gronwall inequality, by using (\ref{ADDEQ1}), one can relax the
regularity assumptions on the initial data and may need only to carry
out some lower order energy estimates, see Li--Titi \cite{LITITIBOU}.

(ii) Another special form of the logarithmic Gronwall inequality in Lemma \ref{LogGron1} is
$$
\frac{d}{dt}A+B\leq n(t)A\log(A+B+e),
$$
where $n\in L^1((0,\mathcal T))$. The above inequality does not necessary
imply the boundedness of $A$ on $(0,\mathcal T)$ in general; however,
as stated in Lemma \ref{LogGron1}, if $n$ satisfies, in addition, that
$n(t)\leq K(A(t)+1)^\alpha$, for some positive constants $K$ and $\alpha$,
then it indeed implies the desired boundedness of $A$ on $(0,\mathcal T)$.

(iii) In the spirit of the system version of the (classic)
Gronwall inequality exploited in Cao--Li--Titi \cite{CAOLITITI3},
one can also exploit the corresponding
system version of the logarithmic Gronwall inequality stated in Lemma \ref{LogGron1}.
\end{remark}

We also need the following Aubin-Lions compactness lemma.

\begin{lemma} [Aubin-Lions Lemma, see Simon \cite{Simon} Corollary 4] \label{AL}Assume that $X, B$ and $Y$ are three Banach spaces, with $X\hookrightarrow\hookrightarrow B\hookrightarrow Y.$ Then it holds that

(i) If $\mathcal{F}$ is a bounded subset of $L^p(0, T; X)$, where $1\leq p<\infty$, and $\frac{\partial \mathcal{F}}{\partial t}=\left\{\frac{\partial f}{\partial t}|f\in \mathcal{F}\right\}$ is bounded in $L^1(0, T; Y)$, then $\mathcal{F}$ is relatively compact in $L^p(0, T; B)$;

(ii) If $\mathcal{F}$ is bounded in $L^\infty(0, T; X)$, and $\frac{\partial \mathcal{F}}{\partial t}$ is bounded in $L^r(0, T; Y)$, where $r>1$, then $\mathcal{F}$ is relatively compact in $C([0, T]; B)$.
\end{lemma}

\section{System with full viscosities and full diffusivity}
\label{sec3}
In this section, we are concerned with energy estimates for the strong solutions to the following regularized system, with both full viscosities and full diffusivity,
\begin{eqnarray}
    &\partial_tv+(v\cdot\nabla_H)v+w\partial_zv-\Delta_Hv-\varepsilon\partial_z^2v  \nonumber\\
    &+f_0\overrightarrow{k}\times v+\nabla_H\left(p_s(x,y,t)-\int_{-h}^zT(x,y,\xi,t)d\xi\right)=0,\label{eq1}\\
    &\int_{-h}^h\nabla_H\cdot v(x,y,z,t)dz=0,\label{eq2}\\
    &\partial_tT+v\cdot\nabla_HT+w
    \partial_zT -\varepsilon\Delta_H T-\partial_z^2T=0, \label{eq3}
\end{eqnarray}
with $w$ given by (\ref{w}),
subject to the boundary and initial conditions (\ref{BC1})--(\ref{IC}).

For any periodic functions $v_0, T_0\in H^2(\Omega)$, which are even and odd in $z$, respectively, there is a unique strong solution to the above system, subject to the boundary and initial conditions (\ref{BC1})--(\ref{IC}), and in fact, we have the following proposition.

\begin{proposition}
  \label{prop3.1}
Suppose that the periodic functions $v_0,T_0\in H^2(\Omega)$ are even and odd in $z$, respectively, with $\int_{-h}^h\nabla_H\cdot v_0(x,y,z)dz=0$. Then for any $\varepsilon>0$, there is a unique global strong solution $(v,T)$ to system (\ref{eq1})--(\ref{eq3}),
subject to the boundary and initial conditions (\ref{BC1})--(\ref{IC}), such that
\begin{align*}
  &(v,T)\in L^\infty_{\text{loc}}([0,\infty);H^2(\Omega))\cap C([0,\infty);H^1(\Omega)),\\
  &(v,T)\in
  L^2_{\text{loc}}([0,\infty);H^3(\Omega)),\quad(\partial_tv,\partial_tT)\in L^2_{\text{loc}}([0,\infty);H^1(\Omega)).
\end{align*}
\end{proposition}

\begin{proof}
The proof can be given in the same way as in \cite{CAOLITITI1} (see Proposition 2.1 there), and thus we omit it here.
\end{proof}

The strong solutions satisfy the following estimates.

\begin{proposition}\label{prop3.2}
For any $0<\mathcal T<\infty$, we have the following:

(i) Basic energy estimate:
$$
\sup_{0\leq t\leq\mathcal T}\|(v,T)\|_2^2(t)+\int_0^{\mathcal T}\|(\nabla_Hv, \partial_zT,
\sqrt\varepsilon \nabla_H T)\|_2^2dt\leq  Ce^\mathcal{T}(\|v_0\|_2^2+\|T_0\|_2^2),
$$
where $C$ is a positive constant depending only on $h$;

(ii) $L^\infty$ estimate on  $T$:
$$
\sup_{0\leq t\leq \mathcal T}\|T\|_\infty(t)\leq   \|T_0\|_\infty;
$$

(iii) $L^q$ estimate on $v$:
$$
\sup_{0\leq t\leq\mathcal T}\|v\|_q(t)\leq C\sqrt q,\quad\mbox{for every }q\in[4,\infty),
$$
for a positive constant $C$ depending only on $h$, $\mathcal T$, and $\|(v_0,T_0)\|_\infty$.
\end{proposition}

\begin{proof}
(i) Multiplying equations (\ref{eq1}) and (\ref{eq3}) by $v$ and $T$, respectively, summing the resulting equations, and integrating over $\Omega$, it follows from integration by parts, using (\ref{eq2}), and the H\"older and Young inequalities that
\begin{align*}
  &\frac{1}{2}\frac{d}{dt}\int_\Omega(|v|^2+|T|^2)dxdydz\\
  &+\int_\Omega\Big(|\nabla_Hv|^2+\varepsilon|\partial
  _zv|^2+\varepsilon|\nabla_HT|^2+|\partial
  _zT|^2\Big)dxdydz\\
  =&-\int_\Omega\left(\int_{-h}^zTd\xi\right)\nabla_H\cdot vdxdydz\\
  \leq&C\|T\|_2\|\nabla_Hv\|_2\leq\frac{1}{2}\|\nabla_Hv\|_2^2+C\|T\|_2^2,
\end{align*}
and thus
$$
\frac{d}{dt}\|(v,T)\|_2^2+\|(\nabla_Hv,\partial_zT,\sqrt\varepsilon\partial_zv,\sqrt\varepsilon\nabla_H T)\|_2^2)
\leq C\|T\|_2^2,
$$
from which, by the Gronwall inequality, one obtains (i).

(ii) Multiplying equation (\ref{eq3}) by $|T|^{q-2}T$, with $q\in[2,\infty)$, and integrating the resultant over $\Omega$, it follows from integration by parts and using (\ref{eq2}) that
\begin{equation*}
  \frac{1}{q}\frac{d}{dt}\|T\|_q^q\leq0,
\end{equation*}
which implies $\sup_{0\leq t\leq \mathcal T}\|T\|_q\leq \|T_0\|_q$. The conclusion follows by taking $q\rightarrow\infty$ and using the fact that $\|T\|_q\rightarrow\|T\|_\infty$, as $q\rightarrow\infty$.

(iii) This has been proven in (iii) of Proposition 3.1 in \cite{CAOLITITI3} (note that the diffusivity plays no role for the proof of (iii) there) and, thus, we omit the details here.
\end{proof}

\subsection{A priori $L^\infty_t(H^1_{\textbf{x}})$ estimate on $v$}

In this subsection, we establish the a priori estimates a priori $L^\infty_t(H^1_{\textbf{x}})$ estimates of $v$. As it will be seen below, we achieve the a priori
$L^\infty_t(H^1_{\textbf{x}})$ on $v$ not through performing directly
the energy inequalities to $v$, but by carrying out the corresponding $L^\infty_t(L^2_{\textbf{x}})$ estimates for $u, \eta$ and $\theta$ defined below:
\begin{equation}\label{etatheta}
  u:=\partial_zv,\quad\eta:=\nabla_H\cdot v+\Phi,\quad\theta:=\nabla_H^\perp\cdot v,
\end{equation}
where $\nabla_H^\perp=(-\partial_y,\partial_x)$ and $\Phi$ is given by
\begin{equation}\label{Phi}
\Phi(x,y,z,t)=\int_{-h}^zT(x,y,\xi,t)d\xi-\frac{1}{2h} \int_{-h}^h\left(\int_{-h}^zT(x,y,\xi,t)d\xi\right) dz.
\end{equation}

Differentiating equation (\ref{eq1}) with respect to $z$ yields
\begin{align}
  \label{4.1u}
\partial_tu&+(v\cdot\nabla_H)u+w\partial_zu-\Delta_Hu-\varepsilon\partial_z^2u\nonumber\\
&+f_0k\times u+(u\cdot\nabla_H)v-(\nabla_H\cdot v)u-\nabla_HT=0.
\end{align}
The functions $\eta$ and $\theta$ satisfy (see Appendix A for the derivation)
\begin{align}
  \partial_t\eta-\Delta_H\eta-\varepsilon\partial_z^2\eta=&-\nabla_H\cdot[(v\cdot\nabla_H)v  +w\partial_zv+f_0k\times v]+(1-\varepsilon)\partial_zT\nonumber\\
  &-wT-\int_{-h}^z(\nabla_H\cdot(vT)-\varepsilon\Delta_HT)d\xi+f(x,y,t) \label{etaeps}
\end{align}
and
\begin{equation}
  \partial_t\theta-\Delta_H\theta-\varepsilon\partial_z^2\theta=-\nabla_H^\perp\cdot[(v\cdot\nabla_H)v  +w\partial_zv+f_0k\times v],\label{thetaeps}
\end{equation}
respectively, with the function $f=f(x,y,t)$ given by
\begin{align}
  f(x,y,t)=&\frac{1}{2h}\int_{-h}^h\left(\int_{-h}^z (\nabla_H\cdot(vT)-\varepsilon\Delta_HT)d\xi+wT\right)
   dz\nonumber\\
   &+\frac{1}{2h}\int_{-h}^h \nabla_H\cdot \big(\nabla_H\cdot(v\otimes v)+f_0\overrightarrow{k}\times v\big)dz.\label{f}
\end{align}

For the convenience, we first prove the following proposition which will be used later:

\begin{proposition}\label{prop3.2-1}
Let $\eta$ and $\theta$ be as in (\ref{etatheta}). The following estimates hold:
\begin{eqnarray*}
\|\eta\|_2^2+\|\theta\|_2^2&\leq& C(\|\nabla_Hv\|_2^2+1), \\
  \|\nabla_Hv\|_2^2&\leq& C(\|\eta\|_2^2+\|\theta\|_2^2+1), \\
    \left(\int_{-h}^h\|\nabla_Hv\|_{4,M}^2dz\right)^{\frac12}
    &\leq&C\left(\|(\eta,\theta)\|_2^{\frac12}\|\nabla_H(\eta,\theta)\|_2^{\frac12} +\|(\eta,\theta)\|_2+1\right),\\
    \int_{-h}^h\|\nabla_Hv\|_{4,M}dz
    &\leq&C\left(\|(\eta,\theta)\|_2^{\frac12}\|\nabla_H(\eta,\theta)\|_2^{\frac12} +\|(\eta,\theta)\|_2+1\right),
  \end{eqnarray*}
  where $C$ is a positive constant depending only on $\|T_0\|_\infty$ and $h$.
\end{proposition}

\begin{proof}
By Proposition \ref{prop3.2}, and recalling the definition of $\Phi$, we have $\|\Phi\|_\infty\leq C\|T\|_\infty\leq C\|T_0\|_\infty$. The first conclusion follows directly from the definitions of $\eta, \theta$, and $\Phi$. By the elliptic estimates, we have
\begin{align*}
  \|\nabla_Hv\|_2^2\leq&C(\|\nabla_H\cdot v\|_2^2+\|\nabla_H^\perp\cdot v\|_2^2)\leq C(\|\eta\|_2^2+\|\theta\|_2^2+\|\Phi\|_2^2)\\
  \leq& C(\|\eta\|_2^2+\|\theta\|_2^2+1),
\end{align*}
proving the second estimate.
For the third estimate, by the elliptic estimates and Lemma \ref{lem2.3}, we have
\begin{eqnarray*}
  \int_{-h}^h\|\nabla_Hv\|_{4,M}^2dz&\leq& C\int_{-h}^h(\|\nabla_H\cdot v\|_{4,M}^2+\|\nabla_H^\perp\cdot v\|_{4,M}^2)dz\\
  &\leq&C\int_{-h}^h(\|(\eta,\theta)\|_{4,M}^2+\|\Phi\|_{4,M}^2)dz \\
  &\leq&C\left(\|(\eta,\theta)\|_2^{\frac12}\|\nabla_H(\eta,\theta)\|_2^{\frac12} +\|(\eta,\theta)\|_2+1\right),
\end{eqnarray*}
while the last inequality follows by applying the H\"older inequality to the third one.
\end{proof}

The energy inequality for $(u,\eta,\theta)$ is contained in the next proposition.

\begin{proposition}
  \label{propapriH1}
Given $\mathcal T\in(0,\infty)$ and assume that $\varepsilon\in(0,1)$.
Let $\eta, \theta$ and $u$ be as in (\ref{etatheta}). We have the following energy inequality:
\begin{align*}
&\frac{d}{dt}\left(\|(\theta,\eta,u)\|_2^2 +\frac{\|u\|_4^4}{2}\right)+ \|\nabla_H(\theta,\eta,u)\|_2^2 +\||u|\nabla_Hu\|_2^2 +\|\sqrt\varepsilon
\partial_z(\theta,\eta,u)\|_2^2\nonumber\\
  \leq&C(\|v\|_\infty^2+\|\nabla_Hv\|_2^2+1)(\|(\theta,\eta)\|_2^2 +\|u\|_4^4+1) +C\|(\partial_zT,\sqrt\varepsilon \nabla_HT)\|_2^2
\end{align*}
for any $t\in(0,\mathcal T)$, where $C$ is a positive constant depending only on $h, \mathcal T, $ and $\|T_0\|_\infty$.
\end{proposition}

\begin{proof}
Multiplying equation (\ref{4.1u}) by $(|u|^2+1)u$ and integrating the resulting equation over $\Omega$, it follows from integration by parts, Proposition \ref{prop3.2}, and using the Young inequality that
\begin{align*}
&\frac{d}{dt}\left(\frac{\|u\|_2^2}{2}+\frac{\|u\|_4^4}{4}\right)+\int_\Omega[|\nabla_Hu|^2+\varepsilon|\partial_zu|^2\\
&+|u|^2(|\nabla_Hu|^2+2|\nabla_H|u||^2+\varepsilon |\partial_zu|^2+2\varepsilon|\partial_z|u||^2)]dxdydz\\
  =&\int_\Omega[\nabla_HT+(\nabla_H\cdot v)u-(u\cdot\nabla_H)v]\cdot(|u|^2+1)udxdydz\\
  \leq&C\int_\Omega[|T|(|u|^2+1)|\nabla_Hu|+|v|(|u|^2+1)|u||\nabla_Hu|]dxdydz\\
  \leq&\frac{1}{2}\int_\Omega(|u|^2+1)|\nabla_Hu|^2dxdydz+C\int_\Omega(|u|^2+1)(|T|^2+|v|^2|u|^2)dxdydz\\
  \leq&\frac{1}{2}\int_\Omega(|u|^2+1)|\nabla_Hu|^2dxdydz +C(1+\|v\|_\infty^2)(\|u\|_4^4+1),
\end{align*}
which gives
\begin{eqnarray}\label{estueps}
  &\displaystyle\frac{d}{dt}\left(\frac{\|u\|_2^2}{2}
  +\frac{\|u\|_4^4}{4}\right)+\|(\nabla_Hu,|u|\nabla_Hu,\sqrt\varepsilon \partial_zu)\|_2^2\nonumber\\
  &\displaystyle\leq C(1+\|v\|_\infty^2)(\|u\|_4^4+1)+\frac12(\|\nabla_Hu\|_2^2+ \||u|\nabla_Hu\|_2^2).
\end{eqnarray}

Multiplying equation (\ref{thetaeps}) by $\theta$ and integrating the resulting equation over $\Omega$, it follows from integration by parts that
\begin{equation}
   \frac{1}{2}\frac{d}{dt}\|\theta\|_2^2+\|(\nabla_H\theta,\sqrt \varepsilon\partial_z\theta) \|_2^2
  = \int_\Omega((v\cdot\nabla_H)v+w\partial_zv+f_0\overrightarrow{k}\times v)\cdot\nabla_H^\perp\theta dxdydz.\label{use0}
\end{equation}
By Propositions \ref{prop3.2} and \ref{prop3.2-1}, it follows from the Young inequality that
\begin{eqnarray}
 \int_\Omega((v\cdot\nabla_H)v+f_0\overrightarrow{k}\times v)\cdot\nabla_H^\perp\theta dxdydz
 \leq\frac18\|\nabla_H\theta\|_2^2+C(\|v\|_\infty^2\|\nabla_Hv\|_2^2+\|v\|_2^2)
 \nonumber\\
  \leq \frac18\|\nabla_H\theta\|_2^2 +C(\|v\|_\infty^2+1)( \|(\eta,\theta)\|_2^2+1).\label{use1}
\end{eqnarray}
By Lemma \ref{ladlemma}, Proposition \ref{prop3.2-1}, (\ref{w}), and the H\"older and Young inequalities, we have
\begin{align}
 &\int_\Omega w\partial_zv \cdot\nabla_H^\perp\theta dxdydz \nonumber\\
  \leq& \int_M\left(\int_{-h}^h|\nabla_Hv| dz\right)\left(\int_{-h}^h|u||\nabla_H\theta|dz\right)dxdy\nonumber\\
  \leq&\left(\int_{-h}^h\|\nabla_Hv\|_{4,M}dz\right)  \left(\int_{-h}^h\|u\|_{4,M}^2dz\right)^{\frac{1}{2}} \|\nabla_H\theta\|_2\nonumber\\
  \leq&C\left(\|(\eta,\theta)\|_2^{\frac12}\|\nabla_H(\eta,\theta) \|_2^{\frac12} +\|(\eta,\theta)\|_2+1\right)\|u\|_4\|\nabla_H\theta\|_2
  \nonumber\\
  \leq&\frac18\|\nabla_H(\eta,\theta)\|_2^2+C[ \|(\eta,\theta)\|_2^2\|u\|_4^4+(\|(\eta,\theta)\|_2^2+1)\|u\|_4^2]\nonumber\\
  \leq&\frac18\|\nabla_H(\eta,\theta)\|_2^2+C
  (\|\nabla_Hv\|_2^2+1)(\|u\|_4^4+1).\label{use1-1}
\end{align}
Substituting (\ref{use1}) and (\ref{use1-1}) into (\ref{use0}) yields
\begin{align}
  \frac12\frac{d}{dt}\|\theta\|_2^2+\|(\nabla_H\theta,\sqrt\varepsilon
  \partial_z\theta) \|_2^2
  \leq&C(\|v\|_\infty^2+\|\nabla_Hv\|_2^2+1)(\|(\eta,\theta) \|_2^2+\|u\|_4^4+1)\nonumber\\
  &+\frac14(\|\nabla_H\eta\|_2^2+\|\nabla_H\theta\|_2^2).\label{use2}
\end{align}

Recalling the definitions of $\eta$ and $\Phi$ and using (\ref{eq2}), one has
\begin{equation*}
  \int_{-h}^h\eta(x,y,z,t) dz=\int_{-h}^h[(\nabla_H\cdot v)+\Phi]dz=0.
\end{equation*}
On account of this, multiplying equation (\ref{etaeps}) by $\eta$, and integrating the resultant over $\Omega$, it follows from integration by parts that
\begin{align}
  &\frac{1}{2}\frac{d}{dt}\|\eta\|_2^2+\|\nabla_H\eta\|_2^2+\varepsilon\|\partial_z\eta\|_2^2
  \nonumber\\
  =&\int_\Omega\left[((1-\varepsilon)\partial_zT-wT)\eta +\left(\int_{-h}^z(vT-\varepsilon\nabla_HT)d\xi\right)\cdot\nabla_H\eta\right]dxdydz\nonumber\\
  &+\int_\Omega[(v\cdot\nabla_H)v+w\partial_zv+f_0\overrightarrow{k}\times v]\cdot\nabla_H\eta dxdydz.
  \label{use3}
\end{align}
Same arguments as for (\ref{use1}) and (\ref{use1-1}) yield
\begin{align*}
 &\int_\Omega[(v\cdot\nabla_H)v+w\partial_zv+f_0\overrightarrow{k}\times v]\cdot\nabla_H\eta dxdydz\nonumber\\
 \leq &\frac18\|\nabla_H(\eta,\theta)\|_2^2+C(\|v\|_\infty^2 +\|\nabla_Hv\|_2^2+1)(\|(\eta,\theta) \|_2^2+\|u\|_4^4+1).
\end{align*}
By Propositions \ref{prop3.2} and \ref{prop3.2-1}, it follows from the H\"older and Young inequalities that
\begin{align*}
  &\int_\Omega\left[((1-\varepsilon)\partial_zT-wT)\eta +\left(\int_{-h}^z(vT-\varepsilon\nabla_HT)d\xi\right)
  \cdot\nabla_H\eta\right]dxdydz\nonumber\\
  \leq&(|1-\varepsilon|\|\partial_zT\|_2+\|T\|_\infty\|w\|_2)
  \|\eta\|_2+C(\|T\|_\infty\|v\|_2+ \varepsilon\|\nabla_HT\|_2)\|\nabla_H\eta\|_2\nonumber\\
  \leq&C(\|\partial_zT\|_2+\|\nabla_Hv\|_2)\|\eta\|_2+C(1+ \varepsilon\|\nabla_HT\|_2)\|\nabla_H\eta\|_2\nonumber\\
  \leq&\frac{1}{8}\|\nabla_H\eta\|_2^2+C(\|\partial_zT\|_2^2+\|\nabla_H v\|_2^2+\varepsilon\|\nabla_HT\|_2^2+\|\eta\|_2^2+1)\nonumber\\
  \leq&\frac{1}{8}\|\nabla_H\eta\|_2^2+C(\|\partial_zT\|_2^2 +\varepsilon \|\nabla_HT\|_2^2+\|\eta\|_2^2+\|\theta\|_2^2+1).
\end{align*}
Substituting the above two inequalities into (\ref{use3}) yields
\begin{align*}
   \frac12\frac{d}{dt}\|\eta\|_2^2+\|(\nabla_H\eta,\sqrt\varepsilon
  \partial_z \eta)\|_2^2
  \leq&C(\|\nabla_Hv\|_2^2+\|v\|_\infty^2+1)(\|(\eta,\theta)\|_2^2 +\|u\|_4^4+1)\nonumber\\
  &+C(\|\partial_zT\|_2^2+\varepsilon \|\nabla_HT\|_2^2)+\frac14\|\nabla_H(\eta,\theta)\|_2^2,
\end{align*}
which, summed with (\ref{estueps}) and (\ref{use2}), yields the conclusion.
\end{proof}

Thanks to Proposition \ref{propapriH1} and using the logarithmic type Gronwall inequality, i.e. Lemma \ref{LogGron1}, we can obtain the a priori
$L^\infty_t(L^2_{\textbf{x}})$ estimate on $(u, \eta,\theta)$. In fact, we have the
following corollary:

\begin{corollary}
  \label{apriH1}
Given $\mathcal T\in(0,\infty)$ and let $\varepsilon\in(0,1)$. Let $\eta,\theta,$ and $u$ be as in (\ref{etatheta}).
The following a priori estimate holds:
\begin{align*}
  \sup_{0\leq t\leq\mathcal T}(\|(\eta,\theta,u)\|_2^2 (t) +\|u\|_4^4(t))+\int_0^{\mathcal T}(&\|\nabla_H(\eta,\theta,u)\|_2^2 +\||u||\nabla_Hu\|_2^2\\
  &+\|\sqrt\varepsilon\partial_z(\eta,\theta,u)\|_2^2
  )dt\leq C
\end{align*}
for a positive constant $C$ depending only on $h, \mathcal T, \|(v_0,T_0)\|_\infty$, and $\|\nabla_Hv_0\|_2+\|\partial_zv_0\|_4$; in particular, $C$ is independent of $\varepsilon\in(0,1)$.
\end{corollary}

\begin{proof}
Denoting
\begin{eqnarray}
  &&A_2 =\|\theta \|_2^2+\|\eta \|_2^2+\|u \|_2^2+\tfrac{\|u \|_4^4}{2}+e,\label{A2}\\
  &&B_2 =\|\nabla_H(\theta,\eta,u) \|_2^2+\|\sqrt\varepsilon\partial_z(\eta,\theta,u)\|_2^2,\label{B2}
\end{eqnarray}
one obtains
\begin{equation}\label{4.2}
  \frac{d}{dt}A_2 +B_2 \leq C(\|v \|_\infty^2+\|\nabla_Hv \|_2^2+1)A_2+ C(\|\partial_zT \|_2^2+\varepsilon\|\nabla_HT\|_2^2)
\end{equation}
for $t\in(0,\mathcal T)$ and for a positive constant $C$ depending only on $h, \mathcal T,$ and $\|T_0\|_\infty$.

By (iii) of Proposition \ref{prop3.2} and applying Lemma \ref{log}, we have
\begin{align}
  \|v\|_\infty\leq&C\max\left\{1,\sup_{q\geq 2}\frac{\|v\|_q}{\sqrt q}\right\}\log^{\frac{1}{2}}(\|v\|_{W^{1,4}(\Omega)}+e)\nonumber\\
  \leq& C\log^{\frac{1}{2}}(\|v\|_{W^{1,4}(\Omega)}+e).\label{4.3}
\end{align}
Recalling the definitions of $\eta$ and $\theta$ and using the elliptic estimate, it follows from the Sobolev embedding inequality that
\begin{align*}
  \|\nabla_Hv\|_4\leq& C(\|\nabla_H\cdot v\|_4+\|\nabla_H^\perp\cdot v\|_4) \leq C(\|\eta\|_4+\|\theta\|_4+\|\Phi\|_4)\\
  \leq&C(\|\eta\|_4+\|\theta\|_4+1)\leq C(\|\eta\|_{H^1(\Omega)}+\|\theta\|_{H^1(\Omega)}+e)\\
  \leq&C(\|\eta\|_2+\|\theta\|_2+\|\nabla_H\eta\|_2+\|\nabla_H\theta\|_2+\|\partial_z\eta\|_2+\|\partial_z \theta\|_2+1)\\
  \leq&C(\|\eta\|_2+\|\theta\|_2+\|\nabla_H\eta\|_2+\|\nabla_H\theta\|_2+\|\nabla_H\cdot u+T\|_2
  +\|\nabla_H^\perp\cdot u\|_2+1)\\
  \leq&C(\|\eta\|_2+\|\theta\|_2+\|\nabla_H\eta\|_2+\|\nabla_H\theta\|_2 +\|\nabla_Hu\|_2+1).
\end{align*}
Therefore, by Propositions \ref{prop3.1} and \ref{prop3.2-1}, one has
\begin{align*}
  \|v\|_{W^{1,4}(\Omega)}\leq&C(\|v\|_4+\|\partial_zv\|_4+\|\nabla_Hv\|_4)\leq C(\|v\|_{H^1(\Omega)}+\|\partial_zv\|_4+\|\nabla_Hv\|_4)\\
  \leq&C(\|v\|_2+\|\nabla_Hv\|_2+\|\partial_zv\|_2+\|\partial_zv\|_4+\|\nabla_Hv\|_4)\\
  \leq&C(\|v\|_2+\|\eta\|_2+\|\theta\|_2+1+\|u\|_2+\|u\|_4+\|\nabla_Hv\|_4)\\
  \leq&C(1+\|\eta\|_2+\|\theta\|_2+\|u\|_2+\|u\|_4+\|\nabla_H\eta\|_2+\|\nabla_H\theta\|_2 +\|\nabla_Hu\|_2)\\
  \leq&C(A_2+B_2).
\end{align*}

With the aid of the above inequality, it follows from (\ref{4.3}) that
\begin{equation}\label{vinf}
\|v\|_\infty^2\leq C\log(A_2+B_2)
\end{equation}
and, consequently, by (\ref{4.2}), we obtain
\begin{align*}
  \frac{d}{dt}A_2+B_2
  \leq C(\|\nabla_Hv \|_2^2+1+\log(A_2+B_2))A_2+ C(\|\partial_zT \|_2^2+\varepsilon\|\nabla_HT\|_2^2).
\end{align*}
Applying Lemma \ref{LogGron1} to the above inequality and using Proposition \ref{prop3.2}, one obtains the conclusion.
\end{proof}

\subsection{A priori $L^\infty_t(H^1_{\textbf{x}})$ estimate on $u=\partial_zv$} In this subsection, we perform the a priori $L^\infty_t(H^1_{\textbf{x}})$ estimate on $u$. As it will be shown below,
the a priori $L^\infty_t(L^2_{\textbf{x}})$ estimate on $\partial_zu$
can be achieved through performing the energy estimates for $u$ directly, while the desired estimate on $\nabla_Hu$ is done by
carrying out the corresponding estimates for $(\varphi, \psi)$ defined,
below, in (\ref{varphipsi}). We first carry out the $L^\infty_t(L^2_{\textbf{x}})$
estimate for $\partial_zu$.

\begin{proposition}\label{aprizu}
Given $\mathcal T\in(0,\infty)$ and assume $\varepsilon\in(0,1)$. Then, the following a priori estimate holds:
\begin{align*}
  \sup_{0\leq t\leq\mathcal T}\|\partial_zu\|_2^2(t)+\int_0^{\mathcal T}\|(\nabla_H\partial_zu,\sqrt\varepsilon\partial_z^2u)\|_2^2 dt\leq C
\end{align*}
for a positive constant $C$ depending only on $h, \mathcal T, \|(v_0,T_0)\|_\infty$, $\|\nabla_Hv_0\|_2+\|\partial_zv_0\|_4+\|\partial_z^2v_0\|_2$; in particular, $C$ is independent of $\varepsilon\in(0,1)$.
\end{proposition}

\begin{proof}
Multiplying equation (\ref{4.1u}) by $-\partial_z^2u$ and integrating the
resultant over $\Omega$, it follows from integration by parts and the H\"older inequality that
\begin{align}
  &\frac{1}{2}\frac{d}{dt}\|\partial_zu\|_2^2+\|\nabla_H\partial_zu\|_2^2+\varepsilon\|\partial_z^2u\|_2^2\nonumber\\
  =&\int_\Omega[(v\cdot\nabla_H)u+w\partial_zu+(u\cdot\nabla_H)v-(\nabla_H\cdot v)u-\nabla_HT]\cdot\partial_z^2udxdydz\nonumber\\
  =&-\int_\Omega[2(u\cdot\nabla_H) u-2(\nabla_H\cdot v)\partial_zu+ \partial_zu\cdot\nabla_Hv-(\nabla_H\cdot u)u]\cdot\partial_zu dxdydz\nonumber\\
  &-\int_\Omega\partial_zT\nabla_H\cdot\partial_zu dxdydz\nonumber\\
  \leq&3\int_\Omega(|u||\nabla_Hu||\partial_zu|+|\nabla_Hv||\partial_zu|^2) dxdydz+\|\partial_zT\|_2^2+\frac14\|\nabla_H\partial_zu\|_2^2. \label{nazueps1}
\end{align}

We need to estimate the terms $\int_\Omega|u||\nabla_Hu||\partial_zu|dxdydz$ and $\int_\Omega|\nabla_Hv||\partial_zu|^2dxdydz$. Noticing that
$$
|\nabla_Hu(x,y,z,t)|\leq\int_{-h}^h|\nabla_H\partial_zu(x,y,z,t)|dz,
$$
it follows from Lemmas \ref{ladlemma} and \ref{lem2.3} and the H\"older and Young inequalities that
\begin{align}
  &3\int_\Omega|u||\nabla_Hu||\partial_zu|dxdydz\nonumber\\
  \leq&C\int_M\left(\int_{-h}^h|\nabla_H\partial_zu|dz\right)
  \left(\int_{-h}^h|u||\partial_zu|dz\right)dxdy\nonumber\\
  \leq&C\left(\int_{-h}^h\|\nabla_H\partial_zu\|_{2,M}dz\right)
  \left(\int_{-h}^h\|u\|_{4,M}^2dz\right)^{\frac{1}{2}}
  \left(\int_{-h}^h\|\partial_zu\|_{4,M}^2dz\right)^{\frac{1}{2}}\nonumber\\
  \leq&C\|\nabla_H\partial_zu\|_2\|u\|_4
  \left(\|\partial_zu\|_2
  +\|\partial_zu\|_2^{\frac{1}{2}}\|\nabla_H\partial_zu\|_2^{\frac{1}{2}} \right)\nonumber\\
  \leq&\frac18\|\nabla_H\partial_zu\|_2^2+C(\|u\|_4^4+1) \|\partial_zu\|_2^2.\label{20171}
\end{align}
Noticing that
$$
|\nabla_Hv(x,y,z,t)|\leq\frac{1}{2h}\int_{-h}^h|\nabla_Hv|dz+\int_{-h}^h |\nabla_Hu|dz,
$$
it follows from Lemmas \ref{ladlemma} and \ref{lem2.3} and the Young inequality that
\begin{align}
  &3\int_\Omega|\nabla_Hv||\partial_zu|^2dxdydz\nonumber\\
  \leq& 3\int_M\left(\int_{-h}^h(|\nabla_Hv|+|\nabla_Hu|)dz\right)\left(
  \int_{-h}^h|\partial_z u|^2 dz\right)dxdy \nonumber\\
  \leq & 3(\|\nabla_Hv\|_2+\|\nabla_H\|_2)\left(
  \int_{-h}^h\|\partial_zu\|_{4,M}^2dz\right) \nonumber\\
  \leq & C(\|\nabla_Hv\|_2+\|\nabla_Hu\|_2)(\|\partial_zu\|_2^2+\|\partial_zu\|_2 \|\nabla_H\partial_zu\|_2)\nonumber\\
  \leq&\frac18\|\nabla_H\partial_zu\|_2^2+C(\|\nabla_Hv\|_2^2+\|\nabla_H u\|_2^2+1)\|\partial_zu\|_2^2. \label{20172}
\end{align}

Substituting (\ref{20171}) and (\ref{20172}) into (\ref{nazueps1}), one obtains
\begin{align*}
  &\frac{d}{dt}\|\partial_zu\|_2^2+\|\nabla_H\partial_zu\|_2^2+\varepsilon\|\partial_z^2u\|_2^2\nonumber\\
  \leq&C(\|\nabla_Hv\|_2^2+\|\nabla_Hu\|_2^2+ \|u\|_4^4+1)\|\partial_zu\|_2^2+C\|\partial_zT\|_2^2.
\end{align*}
Applying the Gronwall inequality to the above inequality and using
Proposition \ref{prop3.1} and Corollary \ref{apriH1}, the conclusion
follows.
\end{proof}

Before proceeding to obtain estimate on $\nabla_Hu$, we define
\begin{equation}
\varphi:=\nabla_H\cdot u+T, \quad \psi:=\nabla_H^\perp\cdot u.\label{varphipsi}
\end{equation}
Equations satisfied by $(\varphi, \psi)$ are derived as follows.
Applying the horizontal divergence operator $\text{div}_H$, or $\nabla_H\cdot$, to equation (\ref{4.1u}) and noticing that
\begin{eqnarray*}
  \nabla_H\cdot((v\cdot\nabla_H)u)&=&v\cdot\nabla_H(\nabla_H\cdot u)+\nabla_Hv:(\nabla_Hu)^T,\\
  \nabla_H\cdot(w\partial_zu)&=&w\partial_z(\nabla_H\cdot u)+\nabla_Hw\cdot\partial_zu,\\
  \nabla_H\cdot(\overrightarrow{k}\times u)&=&\nabla_H\cdot u^\perp=-\nabla_H^\perp\cdot u=-\psi,
\end{eqnarray*}
one has
\begin{eqnarray*}
  \partial_t(\nabla_H\cdot u)+v\cdot\nabla_H(\nabla_H\cdot u)+w\partial_z(\nabla_H\cdot u)-\Delta_H(\nabla_H\cdot u+T)-\varepsilon\partial_z^2\nabla_H\cdot u\\
  =f_0\psi-\nabla_H\cdot((u\cdot\nabla_H)v-(\nabla_H\cdot v)u)-\nabla_H:(\nabla_Hu)^T-\nabla_Hw\cdot\partial_zu.
\end{eqnarray*}
Adding the above equation with (\ref{eq3}) yields
\begin{eqnarray}
  &&\partial_t\varphi+v\cdot\nabla_H\varphi+w\partial_z\varphi-\Delta_H\varphi -\varepsilon\partial_z^2\varphi\nonumber\\ &=&f_0\psi-\nabla_H\cdot((u\cdot\nabla_H)v-(\nabla_H\cdot v)u)+\varepsilon \Delta_HT+(1-\varepsilon)\partial_z^2T\nonumber\\
 && -\nabla_Hv:(\nabla_Hu)^T -\nabla_Hw\cdot\partial_zu.\label{varphi}
\end{eqnarray}
Applying the operator $\nabla_H^\perp\cdot$ to equation (\ref{4.1u}) and noticing that
\begin{eqnarray*}
  \nabla_H^\perp\cdot((v\cdot\nabla_H)u)&=&v\cdot\nabla_H(\nabla_H^\perp\cdot u)+ \nabla_H^\perp v:(\nabla_Hu)^T\\
  &=&v\cdot\nabla_H\psi+\nabla_H^\perp v:(\nabla_Hu)^T,\\
  \nabla_H^\perp\cdot(w\partial_zu)&=&w\partial_z(\nabla_H^\perp\cdot u) +\nabla_H^\perp w\cdot\partial_zu=w\partial_z\psi+\nabla_H^\perp w\cdot\partial_zu,\\
  \nabla_H^\perp\cdot(\overrightarrow{k}\times u)&=&\nabla_H^\perp\cdot u^\perp=\nabla_H\cdot u,
\end{eqnarray*}
where $u^\perp=(-u^2, u^1)$, one obtains
\begin{eqnarray}
  &&\partial_t\psi+v\cdot\nabla_H\psi+w\partial_z\psi-\Delta_H\psi-\varepsilon \partial_z^2\psi\nonumber\\
  &=&-f_0\nabla_H\cdot u-\nabla_H^\perp\cdot((u\cdot\nabla_H)v-(\nabla_H\cdot v)u)\nonumber\\
  &&-\nabla_H^\perp v:(\nabla_Hu)^T-\nabla_H^\perp w\cdot\partial_zu.\label{psi}
\end{eqnarray}

A priori $L^\infty_t(L^2_{\textbf{x}})$ estimate on $(\varphi,\psi)$ is stated in the
next proposition.

\begin{proposition}\label{aprihu}
Given $\mathcal T\in(0,\infty)$ and assume that $\varepsilon\in(0,1)$. Let $\varphi$ and $\psi$ be given in (\ref{varphipsi}). Then, the following a priori estimate holds:
\begin{align*}
  \sup_{0\leq t\leq\mathcal T}\|(\varphi,\psi)\|_2^2(t)+\int_0^{\mathcal T}\|(\nabla_H\varphi,\nabla_H\psi,\sqrt\varepsilon \partial_z\varphi,\sqrt\varepsilon\partial_z\psi)\|_2^2dt\leq C
\end{align*}
for a positive constant $C$ depending only on $h, \mathcal T, \|(v_0,T_0)\|_\infty$, $\|\nabla_Hv_0\|_2+\|\partial_zv_0\|_{H^1}$; in particular, $C$ is independent of $\varepsilon\in(0,1)$.
\end{proposition}

\begin{proof}
Multiplying equation (\ref{varphi}) by $\varphi$ and integrating the resulting equation over $\Omega$, it follows from integration by parts that
\begin{align}
  &\frac12\frac{d}{dt}\|\varphi\|_2^2+\|\nabla_H\varphi\|_2^2 +\varepsilon\|\partial_z\varphi\|_2^2\nonumber\\
  =&\int_\Omega[(f_0\psi+\varepsilon\Delta_HT+(1-\varepsilon)\partial_z^2T) \varphi +((u\cdot\nabla_H)v-\nabla_H\cdot vu)\cdot \nabla_H \varphi ] dxdydz\nonumber\\
  &-\int_\Omega[\nabla_Hv:(\nabla_Hu)^T+\nabla_Hw\cdot\partial_z u]\varphi dxdydz.\label{1700}
\end{align}
Noticing that $\|\partial_z\varphi\|_2\leq\|\nabla_H\partial_zu\|_2+\|\partial_zT\|_2$ and $\|(\varphi,\psi)\|_2^2\leq C(\|\nabla_Hu\|_2^2+1)$, it follows from integrating by parts and the H\"older and Young inequalities that
\begin{align}
  &\int_\Omega[f_0\psi+\varepsilon\Delta_HT+(1-\varepsilon)\partial_z^2T] \varphi dxdydz\nonumber\\
  \leq&C[(\|\nabla_Hu\|_2^2+1)+\varepsilon\|\nabla_HT\|_2\|\nabla_H\varphi\|_2 +\|\partial_zT\|_2(\|\nabla_H\partial_zu\|_2+\|\partial_zT\|_2)] \nonumber\\
  \leq&   \frac18\|\nabla_H\varphi\|_2^2+ +C(\varepsilon\|\nabla_HT\|_2^2+\|\nabla_H \partial_zu\|_2^2+\|\partial_zT\|_2^2+\|\nabla_Hu\|_2^2+1).\label{1701}
\end{align}
Noticing that $|u(x,y,z,t)|\leq\int_{-h}^h|\partial_zu(x,y,z,t)|dz,$
by Lemmas \ref{ladlemma} and \ref{lem2.3}, Proposition \ref{prop3.2-1}, and using the Young inequality, we have
\begin{align}
&\int_\Omega [(u\cdot\nabla_H)v-\nabla_H\cdot vu]\cdot\nabla_H\varphi dxdydz \nonumber\\
  \leq & 2\int_M\left(\int_{-h}^h|\partial_zu|dz\right)\left(\int_{-h}^h|\nabla_Hv| |\nabla_H\varphi|dz\right) dxdy \nonumber\\
  \leq&2\left(\int_{-h}^h\|\partial_zu\|_{4,M}dz\right)\left(\int_{-h}^h\|\nabla_Hv\|_{4,M}^2 \right)^{\frac12}\|\nabla_H\varphi\|_2\nonumber\\
  \leq&C\left(\|\partial_zu\|_2^{\frac12}\|\nabla_H\partial_zu\|_2^{\frac12} +\|\partial_zu\|_2\right)\Big(\|\eta\|_2+\|\theta\|_2 \nonumber\\ &+(\|\eta\|_2+\|\theta\|_2)^{\frac12} (\|\nabla_H\eta\|_2+\|\nabla_H\theta\|_2)^{\frac12}+1\Big) \|\nabla_H\varphi\|_2\nonumber\\
  \leq&\frac18\|\nabla_H\varphi\|_2^2 +C[(\|\partial_zu\|_2^2+\|\eta\|_2^2+\|\theta\|_2^2) (\|\nabla_H\partial_zu\|_2^2\nonumber\\
&+\|\nabla_H\eta\|_2^2+\|\nabla_H \theta\|_2^2)+\|(\partial_zu,\eta,\theta)\|_2^4 +1]\label{1702}
\end{align}
and
\begin{align}
  &-\int_\Omega\nabla_Hv:(\nabla_Hu)^T\varphi dxdydz\nonumber\\
  \leq& \int_M\left(\int_{-h}^h |\nabla_H\partial_zu|dz\right) \left(\int_{-h}^h|\nabla_Hv||\varphi|dz\right) dxdy \nonumber\\
  \leq&\left(\int_{-h}^h\|\nabla_H\partial_zu\|_{2,M}dz\right) \left(\int_{-h}^h\|\nabla_Hv\|_{4,M}^2dz \right)^{\frac12}\left(\int_{-h}^h\|\varphi\|_{4,M}^2dz\right)^{\frac12} \nonumber\\
  \leq& C\|\nabla_H\partial_zu\|_2\left(\|(\eta,\theta)\|_2^{\frac12} \|\nabla_H(\eta,\theta)\|_2^{\frac12}+\|(\eta,\theta)\|_2+1\right) \left(\|\varphi\|_2^{\frac12}\|\nabla_H \varphi\|_2^{\frac12}+\|\varphi\|_2\right)\nonumber\\
  \leq&\frac18\|\nabla_H\varphi\|_2^2+C[\left(\|(\eta,\theta)\|_2^2\| \nabla_H(\eta,\theta)\|_2^2+\|(\eta,\theta)\|_2^4+1\right) \|\varphi\|_2^2+\|\nabla_H\partial_zu\|_2^2]\label{1703}
\end{align}
Applying Lemmas \ref{ladlemma} and \ref{lem2.3} and Proposition \ref{prop3.2-1}, it follows from integrating by parts and the Young inequalities that
\begin{align}
  &-\int_\Omega\nabla_Hw\cdot\partial_zu\varphi dxdydz\nonumber\\
  =&\int_\Omega w(\nabla_H\cdot\partial_zu\varphi+\partial_zu\cdot\nabla_H \varphi)dxdydz \nonumber\\
  \leq&\int_M\left(\int_{-h}^h|\nabla_Hv|dz\right)\left( \int_{-h}^h(|\nabla_H\partial_zu| |\varphi|+|\partial_zu||\nabla_H\varphi|)dz\right)dxdy\nonumber \\
  \leq& \left(\int_{-h}^h\|\nabla_Hv\|_{4,M}dz\right)\left(\int_{-h}^h\|\varphi \|_{4,M}^2dz\right)^{\frac12}\|\nabla_H\partial_zu\|_2\nonumber\\
  &+\left(\int_{-h}^h\|\nabla_Hv\|_{4,M}dz\right)\left(\int_{-h}^h\|\partial_zu \|_{4,M}^2dz\right)^{\frac12}\|\nabla_H\varphi\|_2\nonumber\\
  \leq& C\left(\|(\eta,\theta)\|_2^{\frac12} \|\nabla_H(\eta,\theta)\|_2^{\frac12}+\|(\eta,\theta)\|_2+1\right)
  \Big[\left(\|\varphi\|_2^{\frac12}\|\nabla_H\varphi\|_2^{\frac12} +\|\varphi\|_2\right)\nonumber\\
  &\times \|\nabla_H\partial_zu\|_2+
  \left(\|\partial_zu\|_2^{\frac12}\|\nabla_H\partial_zu\|_2^{\frac12}
  +\|\partial_zu\|_2\right)\|\nabla_H\varphi\|_2\Big]\nonumber\\
  \leq&C\left(\|(\partial_zu,\eta,\theta)\|_2^2\|\nabla_H(\partial_zu, \eta,\theta)\|_2^2+\|(\partial_zu,\eta,\theta)\|_2^4+1\right)(\|\varphi \|_2^2+1)\nonumber\\
  &+\frac18\|\nabla_H\varphi\|_2^2+C\|\nabla_H\partial_zu\|_2^2. \label{1704}
\end{align}
Substituting (\ref{1701})--(\ref{1704}) into (\ref{1700}) yields
\begin{align}
  &\frac{d}{dt}\|\varphi\|_2^2+\|\nabla_H\varphi\|_2^2+\varepsilon\| \partial_z\varphi\|_2^2\nonumber\\
  \leq&C(\|(\partial_zu,\eta,\theta)\|_2^2\|\nabla_H(\partial_zu,\eta, \theta)\|_2^2+\|(\partial_zu,\eta,\theta)\|_2^4
  +1)(\|\varphi\|_2^2+1)\nonumber\\
  &+C(\|(\nabla_Hu,\nabla_H\partial_zu,\partial_zT,\sqrt\varepsilon\nabla_H T)\|_2^2+1).\label{1704-1}
\end{align}

Multiplying equation (\ref{psi}) by $\psi$ and integrating the resultant over $\Omega$, it follows from integration by parts that
\begin{align}
  &\frac12\frac{d}{dt}\|\psi\|_2^2+\|\nabla_H\psi\|_2^2+\varepsilon\| \partial_z\psi\|_2^2\nonumber\\
  =&\int_\Omega[-f_0\nabla_H\cdot u\psi+(u\cdot\nabla_Hv-\nabla_H\cdot v u)\cdot\nabla_H^\perp\psi]dxdydz\nonumber\\
  &-\int_\Omega(\nabla_H^\perp:(\nabla_Hu)^T+\nabla_H^\perp w\cdot\partial_z u)\psi dxdydz.\label{1705}
\end{align}
The same arguments as for (\ref{1702})--(\ref{1704}) yield the estimates
\begin{eqnarray*}
&&\int_\Omega (u\cdot\nabla_Hv-\nabla_H\cdot v u)\cdot\nabla_H^\perp\psi dxdydz\\
&\leq& \frac16\|\nabla_H\psi\|_2^2 +C(\|(\partial_zu,\eta,\theta)\|_2^2\|\nabla_H(\partial_zu,\eta,\theta) \|_2^2+\|(\partial_zu,\eta,\theta)\|_2^4 +1),
\end{eqnarray*}
\begin{eqnarray*}
-\int_\Omega\nabla_H^\perp:(\nabla_Hu)^T\psi dxdydz
&\leq&C \left(\|(\eta,\theta)\|_2^2\| \nabla_H(\eta,\theta)\|_2^2+\|(\eta,\theta)\|_2^4+1\right) \|\psi\|_2^2 \\
&&+\frac16\|\nabla_H\psi\|_2^2+C\|\nabla_H\partial_zu\|_2^2,
\end{eqnarray*}
and
\begin{eqnarray*}
&&-\int_\Omega\nabla_H^\perp w\cdot\partial_zu\psi dxdydz\\
&\leq& C\left(\|(\partial_zu,\eta,\theta)\|_2^2\|\nabla_H(\partial_zu, \eta,\theta)\|_2^2+\|(\partial_zu,\eta,\theta)\|_2^4+1\right) (\|\psi \|_2^2+1)\\
&&+\frac16\|\nabla_H\psi\|_2^2+C\|\nabla_H\partial_zu\|_2^2
\end{eqnarray*}
Thanks to the above estimates, we obtain from (\ref{1705}) that
\begin{align*}
  &\frac{d}{dt}\|\psi\|_2^2+\|\nabla_H\psi\|_2^2+\varepsilon\| \partial_z\psi\|_2^2\nonumber\\
  \leq&C(\|(\partial_zu,\eta,\theta)\|_2^2\|\nabla_H(\partial_zu,\eta, \theta)\|_2^2+\|(\partial_zu,\eta,\theta)\|_2^4
  +1)(\|\psi\|_2^2+1)\nonumber\\
  &+C(\|(\nabla_Hu,\nabla_H\partial_zu)\|_2^2+1).
\end{align*}

Summing the above inequality with (\ref{1704-1}) yields
\begin{align*}
  &\frac{d}{dt}\|(\varphi,\psi)\|_2^2 +\|(\nabla_H\varphi,\nabla_H\psi,\sqrt\varepsilon \partial_z\varphi,\sqrt\varepsilon\partial_z\psi)\|_2^2\nonumber\\
  \leq&C(\|(\partial_zu,\eta,\theta)\|_2^2\|\nabla_H(\partial_zu,\eta, \theta)\|_2^2+\|(\partial_zu,\eta,\theta)\|_2^4
  +1)(\|(\varphi,\psi)\|_2^2+1)\nonumber\\
  &+C(\|(\nabla_Hu,\nabla_H\partial_zu,\partial_zT,\sqrt\varepsilon\nabla_H T)\|_2^2+1),
\end{align*}
from which, by the Gronwall inequality, and using Proposition \ref{prop3.1}, Corollary \ref{apriH1}, and Proposition \ref{aprizu}, one obtains the conclusion.
\end{proof}

\subsection{Energy inequalities for $(\nabla_H\eta, \nabla_H\theta)$}
In this subsection, we are concerned with deriving energy inequalities for
$(\nabla_H\eta, \nabla_H\theta)$, where $\eta$ and $\theta$ are given
in (\ref{etatheta}). It should be
noticed that the energy inequalities for $(\nabla_H\eta, \nabla_H\theta)$ do not yield the a priori estimates
of themselves, without appealing to the
energy inequalities for $\nabla T$.

As a preparation, we prove the following:

\begin{proposition}\label{nu4m}
Let $\eta,\theta,$ and $u$ be given in (\ref{etatheta}), and $\varphi$ and
$\psi$ as in (\ref{varphipsi}). The following estimates hold:
  \begin{eqnarray*}
    \left(\int_{-h}^h\|\nabla_Hu\|_{4,M}^2dz\right)^{\frac12}&\leq& C\left(\|(\varphi,\psi)\|_2^{\frac12}\|\nabla_H(\varphi,\psi) \|_2^{\frac12}+\|(\varphi,\psi)\|_2+1\right), \\
    \int_{-h}^h\|\nabla_Hu\|_{4,M}dz&\leq& C\left(\|(\varphi,\psi)\|_2^{\frac12}\|\nabla_H(\varphi,\psi) \|_2^{\frac12}+\|(\varphi,\psi)\|_2+1\right),\\
    \left(\int_{-h}^h\|u\|_{\infty,M}^2dz\right)^{\frac12}&\leq& C\left(\|u\|_4+ \|(\varphi,\psi)\|_2^{\frac12}\|\nabla_H(\varphi,\psi) \|_2^{\frac12}+\|(\varphi,\psi)\|_2+1\right)
  \end{eqnarray*}
  for a positive constant $C$ depending only on $\|T_0\|_\infty$ and $h$, in
  particular it is independent of $\varepsilon\in(0,1)$.
\end{proposition}

\begin{proof}
By Proposition \ref{prop3.2}, it follows from the elliptic estimate that for any $z\in(-h,h)$
\begin{eqnarray*}
\|\nabla_Hu(\cdot,z)\|_{4,M}&\leq&C(\|\nabla_H\cdot u(\cdot,z)\|_{4,M}+
\|\nabla_H^\perp\cdot u(\cdot,z)\|_{4,M}) \\
&\leq& C(\|\varphi(\cdot,z)\|_{4,M}+\|T(\cdot,z)\|_{4,M}+\|\psi(\cdot,z)\|_{4,M})\\
&\leq&C(\|\varphi(\cdot,z)\|_{4,M}+\|\psi(\cdot,z)\|_{4,M}+1)
\end{eqnarray*}
for a positive constant $C$ depending only on $\|T_0\|_\infty$. Thanks to the above, it follows from Lemma \ref{lem2.3} that
\begin{eqnarray*}
  \left(\int_{-h}^h\|\nabla_Hu\|_{4,M}^2dz\right)^{\frac12}
  \leq C\left[\left(\int_{-h}^h\|(\varphi,\psi)\|_{4,M}^2dz\right)^{\frac12}+1\right]\\
  \leq C\left(\|(\varphi,\psi)\|_2^{\frac12}\|\nabla_H(\varphi,\psi)\|_2^{\frac12} +\|(\varphi,\psi)\|_2+1\right),
\end{eqnarray*}
proving the first inequality, while the second one follows from the first one by applying the H\"older inequality. For the third inequality, by the Sobolev embedding inequality and the H\"older inequality, and using the first conclusion, we have
\begin{align*}
   \left(\int_{-h}^h\|u\|_{\infty,M}^2dz\right)^{\frac12}\leq C\left(\int_{-h}^h(\|u\|_{4,M}^2+\|\nabla_Hu\|_{4,M}^2)dz \right)^{\frac12} \\
   \leq C\left(\|u\|_4+ \|(\varphi,\psi)\|_2^{\frac12}\|\nabla_H(\varphi,\psi) \|_2^{\frac12}+\|(\varphi,\psi)\|_2+1\right),
\end{align*}
proving the third inequality.
\end{proof}

We have the following proposition about the energy inequality for $(\nabla_H\eta,\nabla_H\theta)$:

\begin{proposition}\label{propnhetatheta}
We have the following estimate
\begin{align*}
&\frac{d}{dt}\|\nabla_H(\eta,\theta)\|_2^2 +\|(\Delta_H\eta,\Delta_H\theta,\sqrt\varepsilon \nabla_H\partial_z\eta, \sqrt\varepsilon \nabla_H\partial_z\theta)\|_2^2\\
  \leq&C(\|v\|_\infty^2+\|u\|_4^2 +\|(\varphi,\psi)\|_2 \|\nabla_H(\varphi,\psi)\|_2+\|(\varphi,\psi)\|_2^2+1)
\nonumber\\
   &\times\|\nabla_H(\eta,\theta, T)\|_2^2+C(\|(\eta,\theta,\varphi,\psi)\|_2^2+\|u\|_4^4+1)\nonumber\\
   & \times(\|\nabla_H(v,\eta, \theta,\varphi,\psi)\|_2^2 +\|\partial_zT\|_2^2+1) +C\varepsilon^2\|\Delta_HT\|_2^2,
\end{align*}
where $C$ is a positive constant depending only on $h$ and $\|T_0\|_\infty$; in particular, $C$ is independent of $\varepsilon\in(0,1)$.
\end{proposition}

\begin{proof}
Multiplying equation (\ref{theta}) by $-\Delta_H\theta$ and integrating the resultant over $\Omega$, it follows from integration by parts that
\begin{align}
  &\frac12\frac{d}{dt}\|\nabla_H\theta\|_2^2 +\|\Delta_H\theta\|_2 +\varepsilon\|\nabla_H\partial_z\theta\|_2^2\nonumber\\
  =&\int_\Omega \nabla_H^\perp\cdot(v\cdot\nabla_Hv+w\partial_zv+f_0k\times v)\Delta_H\theta dxdydz\nonumber\\
  \leq&\frac{1}{16}\|\Delta_H\theta\|_2^2+f_0\|\nabla_Hv\|_2^2+
  \int_\Omega(|v||\nabla_H^2v|+|\nabla_Hv|^2\nonumber\\
  &+|w||\nabla_Hu|+|\nabla_H w||u|)|\Delta_H\theta|dxdydz.\label{ntheta0}
\end{align}

We estimate the terms in (\ref{ntheta0}) as follows. First, for  $\int_\Omega|v||\nabla_H^2v||\Delta_H\theta|dxdydz$, by using (\ref{etatheta}) and the Young inequality, we have
\begin{align}
\int_\Omega&|v||\nabla_H^2v||\Delta_H\theta|dxdydz
  \leq \|v\|_\infty\|\Delta_H v\|_2\|\Delta_H\theta\|_2\nonumber\\
  &=\|v\|_\infty\|\nabla_H (\nabla_H\cdot v)-\nabla_H^\perp(\nabla_H^\perp\cdot v)\|_2\|\Delta_H\theta\|_2 \nonumber\\
  &=\|v\|_\infty\|\nabla_H(\eta+\Phi)-\nabla_H^\perp\theta\|_2\|\Delta_H \theta\|_2 \nonumber\\
  &\leq C\|v\|_\infty\|\nabla_H(\eta,\theta,T)\|_2\|\Delta_H \theta\|_2\nonumber\\
  &\leq\frac{1}{16}\|\Delta_H\theta\|_2^2+C\|v\|_\infty^2\|\nabla_H(\eta,\theta, T)\|_2^2.\label{ntheta1}
\end{align}
Then, for the term $\int_\Omega|\nabla_Hv|^2|\Delta_H\theta|dxdydz$, recalling that
$$
|\nabla_Hv(x,y,z,t)|\leq\frac{1}{2h}\int_{-h}^h|\nabla_Hv|dz+\int_{-h}^h |\nabla_Hu|dz,
$$
we deduce by Lemma \ref{ladlemma}, Propositions \ref{prop3.2-1} and \ref{nu4m}, and the Young inequality that
\begin{align}
  &\int_\Omega(|\nabla_Hv|^2|\Delta_H\theta|+|w||\nabla_Hu||\Delta_H\theta|) dxdydz\nonumber\\
  \leq& C\int_{M}\left(\int_{-h}^h(|\nabla_Hv|+|\nabla_Hu|)dz\right)
  \left(\int_{-h}^h |\nabla_Hv||\Delta_H\theta|dz\right)dxdy \nonumber\\
    &+\int_M
    \left(\int_{-h}^h|\nabla_Hv|dz\right)\left(\int_{-h}^h|\nabla_Hu||\Delta_H \theta|dz\right)dxdy\nonumber\\
  \leq& C\left(\int_{-h}^h\|\nabla_H(v,u)\|_{4,M}dz\right)\left(\int_{-h}^h \|\nabla_Hv\|_{4,M}^2dz\right)^{\frac12}\|\Delta_H\theta\|_2 \nonumber\\
  &+\left(\int_{-h}^h\|\nabla_Hv\|_{4,M}dz\right)\left(\int_{-h}^h \|\nabla_Hu\|_{4,M}^2dz\right)^{\frac12}\|\Delta_H\theta\|_2 \nonumber\\
  \leq& C\left(\|(\eta,\theta,\varphi,\psi)\|_2 \|\nabla_H(\eta, \theta,\varphi,\psi)\|_2 +\|(\eta,\theta,\varphi,\psi)\|_2^2 +1\right) \|\Delta_H\theta\|_2 \nonumber\\
  \leq&\frac{1}{16}\|\Delta_H\theta\|_2^2+C(\|(\eta,\theta,\varphi,\psi)\|_2^2+1) (\|\nabla_H(\eta, \theta,\varphi,\psi)\|_2^2+1).\label{ntheta2}
\end{align}
Finally, for the term $\int_\Omega|\nabla_Hw||u||\Delta_H\theta|dxdydz$, thanks to (\ref{w}), (\ref{etatheta})--(\ref{Phi}), we have
\begin{align}
  &\int_\Omega|\nabla_Hw||u||\Delta_H\theta|dxdydz\nonumber\\
  \leq& C\int_M\left(\int_{-h}^h (|\nabla_H\eta|+|\nabla_HT|)dz\right)\left(\int_{-h}^h |u||\Delta_H\theta|dz\right)dxdy. \label{ntheta3'}
\end{align}
For the term $C\int_M\left(\int_{-h}^h (|\nabla_H\eta|+|\nabla_HT|)dz\right)\left(\int_{-h}^h |u||\Delta_H\theta|dz\right)dxdy$, by Lemmas \ref{ladlemma} and \ref{lem2.3} and using the H\"older and Young inequalities, we have
\begin{align}
C\int_M&\left(\int_{-h}^h (|\nabla_H\eta|+|\nabla_HT|)dz\right)\left(\int_{-h}^h |u||\Delta_H\theta|dz\right)dxdy\nonumber\\
  \leq& C\left(\int_{-h}^h\|\nabla_H\eta\|_{4,M}dz\right) \left(\int_{-h}^h\|u\|_{4,M}^2 dz\right)^{\frac12}\|\Delta_H\theta\|_2\nonumber\\
  \leq& C\|\nabla_H\eta\|_2^{\frac12}\|\nabla_H^2\eta\|_2^{\frac12}\|u\|_4 \|\Delta_H\theta\|_2
  \leq \frac{1}{32}\|\Delta_H(\eta,\theta)\|_2^2 +C\|\nabla_H\eta\|_2^2\|u\|_4^4.\label{ntheta3'-1}
\end{align}
For the term $\int_M\left(\int_{-h}^h|\nabla_HT|dz\right)\left(\int_{-h}^h|u||\Delta_H\theta| dz\right)dxdy$, we estimate it as follows. Using the H\"older and Young inequalities and
applying Proposition \ref{nu4m}, we deduce
\begin{align}
  &C\int_M\left(\int_{-h}^h|\nabla_HT|dz\right)\left(\int_{-h}^h|u||\Delta_H\theta| dz\right)dxdy \nonumber\\
  \leq& C\int_M\left(\int_{-h}^h |\nabla_HT|dz\right)\left(\int_{-h}^h|u|^2dz\right)^{\frac12} \left(\int_{-h}^h|\Delta_H\theta|^2dz\right)^{\frac12} dxdy\nonumber\\
  \leq& C\left[\int_M\left(\int_{-h}^h|\nabla_HT|dz
  \right)^2dxdy\right]^{\frac12} \left(\int_{-h}^h\|u\|_{\infty,M}^2dz\right)^{\frac12}\|\Delta_H\theta\|_2 \nonumber\\
  \leq& C\left(\int_{-h}^h\|u\|_{\infty,M}^2dz\right)^{\frac12}\|\Delta_H \theta\|_2\|\nabla_HT\|_2 \nonumber\\
  \leq& C\left(\|u\|_4+ \|(\varphi,\psi)\|_2^{\frac12}\|\nabla_H(\varphi,\psi) \|_2^{\frac12}+\|(\varphi,\psi)\|_2+1\right) \|\Delta_H \theta\|_2\|\nabla_HT\|_2\nonumber\\
  \leq&\frac{1}{32}\|\Delta_H\theta\|_2^2+C[\|u\|_4^2 +\|(\varphi,\psi)\|_2 (\|\nabla_H(\varphi,\psi)\|_2+1)+1]\|\nabla_HT\|_2^2.
  \label{ntheta3'-2}
\end{align}
Thanks to (\ref{ntheta3'-1}) and (\ref{ntheta3'-2}), we obtain from (\ref{ntheta3'}) that
\begin{align}
  &\int_\Omega|\nabla_Hw||u||\Delta_H\theta|dxdydz\nonumber\\
  \leq&\frac{1}{16}\|\Delta_H(\eta,\theta)\|_2^2 +C(\|u\|_4^2 +\|(\varphi,\psi)\|_2 \|\nabla_H(\varphi,\psi)\|_2\nonumber\\
  &+\|(\varphi,\psi)\|_2^2+1)\|\nabla_HT\|_2^2 +C\|\nabla_H\eta\|_2^2\|u\|_4^4.\label{ntheta3}
\end{align}

Combining (\ref{ntheta1}), (\ref{ntheta2}) and (\ref{ntheta3}), we obtain
\begin{align}
  &\int_\Omega(|v||\nabla_H^2v|+|\nabla_Hv|^2+|w||\nabla_Hu|+|\nabla_H w||u|)|\Delta_H\theta|dxdydz\nonumber\\
  \leq&C(\|v\|_\infty^2+\|u\|_4^2 +\|(\varphi,\psi)\|_2 \|\nabla_H(\varphi,\psi)\|_2
   +\|(\varphi,\psi)\|_2^2+1)\|\nabla_H(\eta,\theta, T)\|_2^2\nonumber\\
   &+C(\|(\eta,\theta,\varphi,\psi)\|_2^2+\|u\|_4^4+1) (\|\nabla_H(\eta, \theta,\varphi,\psi)\|_2^2+1)+\frac{3}{16}\|\Delta_H(\eta,\theta)\|_2^2.
   \label{nthetaeta}
\end{align}
Therefore, recalling that $\|\nabla_Hv\|_2^2\leq C(\|\eta\|_2^2+\|\theta\|_2^2+1)$, guaranteed by Proposition \ref{prop3.2-1}, it follows from (\ref{ntheta0}) that
\begin{align}
  &\frac12\frac{d}{dt}\|\nabla_H\theta\|_2^2 +\|\Delta_H\theta\|_2 +\varepsilon\|\nabla_H\partial_z\theta\|_2^2\nonumber\\
  \leq&\frac{1}{4}\|\Delta_H(\eta,\theta)\|_2^2 +C(\|v\|_\infty^2+\|u\|_4^2 +\|(\varphi,\psi)\|_2 \|\nabla_H(\varphi,\psi)\|_2
\nonumber\\
   &+\|(\varphi,\psi)\|_2^2+1)\|\nabla_H(\eta,\theta, T)\|_2^2+C(\|(\eta,\theta,\varphi,\psi)\|_2^2\nonumber\\
   &+\|u\|_4^4+1) (\|\nabla_H(v,\eta, \theta,\varphi,\psi)\|_2^2+1). \label{nhtheta}
\end{align}

Recall that $\int_{-h}^h\eta(x,y,z,t)dy=0,$ which implies
$$
\int_\Omega f(x,y,t)\Delta_H\eta(x,y,z,t)dxdydz=0,
$$
where $f(x,y,t)$ is given by (\ref{f}). Multiplying equation (\ref{etaeps}) by $-\Delta_H\eta$ and integrating the resultant over $\Omega$, it follows from integration by parts, Proposition \ref{prop3.2}, and the H\"older inequality that
\begin{align}
  &\frac{1}{2}\frac{d}{dt}\|\nabla_H\eta\|_2^2+\|\Delta_H\eta\|_2^2 +\varepsilon\|\nabla_H\partial_z\eta\|_2^2\nonumber\\
  =&\int_\Omega\bigg\{\nabla_H\cdot\big[(v\cdot\nabla_H)v +w\partial_zv+f_0\overrightarrow{k}\times v\big]-(1-\varepsilon)\partial_zT\nonumber\\
  &+wT+\left(\int_{-h}^z(\nabla_H\cdot(vT)-\varepsilon\Delta_HT)d\xi\right) +f(x,y,t)\bigg\}\Delta_H\eta dxdydz\nonumber\\
  =&\int_\Omega\bigg\{ \nabla_H\cdot\big[(v\cdot\nabla_H)v+w\partial_zv\big]+f_0\nabla_H\cdot( \overrightarrow{k}\times v)-(1-\varepsilon)\partial_zT+wT \nonumber\\
  &+\left(\int_{-h}^z\big((\nabla_H\cdot v)T+v\cdot\nabla_HT-\varepsilon\Delta_HT\big)d\xi\right)\bigg\}
  \Delta_H\eta dxdydz\nonumber\\
  \leq&C(\|\nabla_Hv\|_2+\|\partial_zT\|_2+\|v\|_\infty\|\nabla_HT\|_2 +\varepsilon\|\Delta_HT\|_2)\|\Delta_H\eta\|_2 \nonumber\\
  &+\int_\Omega \nabla_H\cdot\big[(v\cdot\nabla_H)v+w\partial_zv\big] \Delta_H\eta dxdydz.\label{neta0}
\end{align}
Same arguments as for (\ref{nthetaeta}) yield
\begin{align*}
  &\int_\Omega \nabla_H\cdot\big[(v\cdot\nabla_H)v+w\partial_zv \big] \Delta_H\eta dxdydz\\
  \leq&\int_\Omega(|v||\nabla_H^2v|+|\nabla_Hv|^2+|w||\nabla_Hu|+|\nabla_H w||u|)|\Delta_H\eta|dxdydz \nonumber\\
  \leq&C(\|v\|_\infty^2+\|u\|_4^2 +\|(\varphi,\psi)\|_2 \|\nabla_H(\varphi,\psi)\|_2
   +\|(\varphi,\psi)\|_2^2\nonumber\\
   &+1)\|\nabla_H(\eta,\theta, T)\|_2^2+C(\|(\eta,\theta,\varphi,\psi)\|_2^2+\|u\|_4^4+1) \\
   &\times(\|\nabla_H(\eta, \theta,\varphi,\psi)\|_2^2+1)+\frac{3}{16}\|\Delta_H(\eta,\theta)\|_2^2.
\end{align*}
Thanks to this estimate, it follows from (\ref{neta0}) and the Young inequality that
\begin{align*}
&\frac{1}{2}\frac{d}{dt}\|\nabla_H\eta\|_2^2+\|\Delta_H\eta\|_2^2+ \varepsilon\|\nabla_H\partial_z\eta\|_2^2\nonumber\\
  \leq&\frac{1}{4}\|\Delta_H(\eta,\theta)\|_2^2 +C(\|v\|_\infty^2+\|u\|_4^2 +\|(\varphi,\psi)\|_2 \|\nabla_H(\varphi,\psi)\|_2
\nonumber\\
   &+\|(\varphi,\psi)\|_2^2+1)\|\nabla_H(\eta,\theta, T)\|_2^2+C(\|(\eta,\theta,\varphi,\psi)\|_2^2+\|u\|_4^4\nonumber\\
   &+1) (\|\nabla_H(v,\eta, \theta,\varphi,\psi)\|_2^2 +\|\partial_zT\|_2^2+1)+C\varepsilon^2\|\Delta_HT\|_2^2,
\end{align*}
which, summed with (\ref{nhtheta}), yields the conclusion.
\end{proof}

Note that $\nabla_HT$ is involved in the energy inequality of Proposition \ref{propnhetatheta} and, thus, it does not yield the
a priori estimate for $(\nabla_H\eta,\nabla_H\theta)$. Therefore, we
need to combine the energy inequalities for $(\nabla_H\eta,\nabla_H\theta)$, which have already been stated in Proposition \ref{propnhetatheta}, with those for $\nabla_HT$, which will
be stated in the next subsection.

\subsection{Energy inequality for $\nabla T$} In this subsection, we are concerned with performing the energy inequalities for the first order derivatives of $T$.

Define the function $\varpi(x,y,z,t)$ as follows: for any $z\in(-h,h)$ and $t\in(0,\infty)$, $\varpi(\cdot,z,t)$ is the unique solution to the two-dimensional elliptic system subject to horizontal boundary conditions
\begin{equation}
\left\{
\begin{array}{l}
\nabla_H\cdot\varpi(x,y,z,t)=\Phi(x,y,z,t)-\frac{1}{|M|}\int_M\Phi(x,y,z,t) dxdy,\quad\mbox{in }\Omega,\\
\nabla_H^\perp\cdot\varpi(x,y,z,t)=0,\quad\mbox{in }\Omega,\qquad\int_M\varpi(x,y,z,t) dxdy=0,
\end{array}
\right.\label{beta}
\end{equation}
where $\Phi$ is the function given by (\ref{Phi}).
Define a function $\zeta$ as
\begin{equation}\label{zeta}
  \zeta(x,y,z,t)=v(x,y,z,t)+\varpi(x,y,z,t),
\end{equation}
then, recalling the definitions of $\eta$ and $\theta$, one can easily check that
\begin{equation}\label{zeeithta}
 \nabla_H\cdot\zeta= \eta-\frac{1}{|M|}\int_M\Phi dxdy,\quad\nabla_H^\perp\cdot\zeta=\theta.
\end{equation}

The following proposition will be used later.

\begin{proposition}\label{propzetavarpi}
Let $\eta$ and $\theta$ as in (\ref{etatheta}), $\varpi$ as in (\ref{beta}), and $\zeta$ as in (\ref{zeta}). Then, the following inequalities hold:
  \begin{eqnarray*}
 \int_{-h}^h\| \nabla_H\zeta(\cdot,z,t)\|_{\infty,M}dz&\leq & C(\|\nabla_H\eta\|_2(t)+\|\nabla_H\theta\|_2(t)+1)\\
   &&\times
   \log^{\frac{1}{2}}(e+\|\Delta_H\eta\|_{2}(t)+\|\Delta_H\theta\|_{2}(t)),\\
 \sup_{-h\leq z\leq h}\|\nabla_H\varpi(\cdot,z,t)\|_{\infty,M} &\leq& C\log(e+\|\nabla_HT\|_{q}(t)),\quad q\in(2,\infty),
\end{eqnarray*}
for a positive constant $C$ depending only on $h, q,$ and $\|T_0\|_\infty$.
\end{proposition}

\begin{proof}
 Recall the two-dimensional version of  the Br\'ezis-Gallouet-Wainger inequality (see, e.g., \cite{BG,BW})
$$
\|g\|_{\infty, M}\leq C(1+\|g\|_{H^1(M)})\log^{\frac{1}{2}}(e+\|g\|_{H^2(M)})
$$
for any $g\in H^2(M)$.
By the aid of this, recalling (\ref{zeeithta}), it follows from the two-dimensional elliptic estimates, the Poincar\'e, H\"older and Jensen inequalities that
\begin{align*}
  &\int_{-h}^h\|\nabla_H\zeta\|_{\infty,M}dz\leq C\int_{-h}^h(\|\nabla_H\zeta\|_{H^1(M)}+1)
  \log^{\frac{1}{2}}(e+\|\nabla_H\zeta\|_{H^2(M)})dz\nonumber\\
  \leq&C\int_{-h}^h(\|(\nabla_H^\perp\cdot\zeta,\nabla_H\cdot\zeta) \|_{H^1(M)}+1)  \log^{\frac{1}{2}}(e+\|(\nabla_H^\perp\cdot\zeta,\nabla_H\cdot\zeta)\|_{H^2(M)})dz \nonumber\\
  \leq&C\int_{-h}^h(\|(\nabla_H\nabla_H^\perp\cdot\zeta, \nabla_H\nabla_H\cdot\zeta)\|_{2,M}+1)
  \log^{\frac{1}{2}}(e+\|\Delta_H(\nabla_H^\perp\cdot\zeta ,\nabla_H\cdot\zeta)\|_{2,M})dz \nonumber\\
  \leq&C\int_{-h}^h(\|\nabla_H\eta\|_{2,M}+\|\nabla_H\theta\|_{2,M}+1) \log^{\frac{1}{2}}(e+\|\Delta_H\eta\|_{2,M}+\|\Delta_H \theta\|_{2,M})dz\nonumber\\
  \leq&C\left(\int_{-h}^h\|\nabla_H(\eta,\theta) \|_{2,M}^2)dz+1\right)^{\frac{1}{2}}\left(\int_{-h}^h\log (e+\|\Delta_H(\eta,\theta)\|_{2,M})\frac{dz}{2h} \right)^{\frac{1}{2}}\nonumber\\
   \leq&C(\|\nabla_H\eta\|_2+\|\nabla_H\theta\|_2+1)
   \left[\log \left(\int_{-h}^h(e+\|\Delta_H\eta\|_{2,M}+\|\Delta_H\theta\|_{2,M})\frac{dz}{2h} \right)\right]^{\frac{1}{2}}\nonumber\\
   \leq&C(\|\nabla_H\eta\|_2+\|\nabla_H\theta\|_2+1)
   \log^{\frac{1}{2}}(e+\|\Delta_H\eta\|_{2}+\|\Delta_H\theta\|_{2}),
\end{align*}
proving the first conclusion.

Recall the following logarithmic Sobolev type inequality (see, e.g., \cite{BKM}), for any function $g=(g^1,g^2)\in W^{1,q}(M)$, $q\in(2,\infty)$,
$$
\|\nabla_Hg\|_{\infty, M}\leq C(\|\nabla_H^\perp\cdot g\|_{\infty,M}+\|\nabla_H\cdot g\|_{\infty,M}+1)\log(e+\|g\|_{W^{1,q}(M)}).
$$
By the aid of this, recalling (\ref{beta}) and $\|T\|_\infty\leq\|T_0\|_\infty$, and applying the elliptic estimates, one has
\begin{align*}
  &\sup_{-h\leq z\leq h}\|\nabla_H\varpi(\cdot,z,t)\|_{\infty,M}\\
  \leq&C\sup_{-h\leq z\leq h}(\|\nabla_H^\perp\cdot\varpi(\cdot,z,t)\|_{\infty,M} +\|\nabla_H\cdot\varpi(\cdot,z,t)\|_{\infty,M}+1)\\
  &\times\log
  (e+\|\nabla_H\varpi(\cdot,z,t)\|_{W^{1,q}(M)})\\
  \leq&C\sup_{-h\leq z\leq h}\log(e+\|\nabla_H^\perp\cdot\varpi(\cdot,z,t)\|_{W^{1,q}(M)}
  +\|\nabla_H \cdot\varpi(\cdot,z,t)\|_{W^{1,q}(M)})\\
  =&C\sup_{-h\leq z\leq h}\log(e+\|\nabla_H \cdot\varpi(\cdot,z,t)\|_{W^{1,q}(M)})\\
  \leq&C\sup_{-h\leq z\leq h}\log(e+\|\nabla_H\nabla_H \cdot\varpi(\cdot,z,t)\|_{q,M}) \\
  =&C\sup_{-h\leq z\leq h}\log(e+\|\nabla_H\Phi(\cdot,z,t)\|_{q,M}).
\end{align*}
Recalling the definition of $\Phi$, (\ref{Phi}), one can easily check that
$$
\|\nabla_H\Phi(\cdot,z,t)\|_{q,M}\leq C\|\nabla_HT(\cdot,t)\|_q
$$
and, therefore, we have
\begin{align*}
   &\sup_{-h\leq z\leq h}\|\nabla_H\varpi(\cdot,z,t)\|_{\infty,M} \leq C\log(e+\|\nabla_HT\|_{q}),
\end{align*}
proving the second conclusion.
\end{proof}

Energy inequality for $\nabla T$ is stated in the next proposition.

\begin{proposition}
  \label{propnt}
  Let $\eta$ and $\theta$ be as in (\ref{etatheta}). Then, the following energy inequalities hold:
  \begin{align*}
  &\frac{d}{dt}\left( \frac{\|\nabla T\|_2^2}{2} +\frac{\|\nabla_HT\|_q^q}{q}\right)+\|(\partial_z^2T,\nabla_H\partial_z T,\sqrt\varepsilon \Delta_HT)\|_2^2
  \nonumber\\
  \leq&C_\sigma (\|\nabla_H(\eta,\theta)\|_2^2+1) (\|\nabla_HT\|_q^q+1)\log(e +\|\Delta_H(\eta,\theta)\|_{2}+\|\nabla_HT\|_q) \nonumber\\
&+C_\sigma(\|v\|_\infty^2+1)(\|\nabla T\|_2^2
+\|\eta\|_2^2+1)  +\sigma\left(\|\Delta_H\eta\|_2^2+\|\nabla_H\partial_z T\|_2^2\right),
\end{align*}
for any $\sigma>0$, if $q\in(2,4]$, where $C_\sigma$ is a positive number depending only on $h, q, \|T_0\|_\infty$, and $\sigma$; and
\begin{align*}
  &\frac{d}{dt}\left(\frac{2}{q}\|\nabla_HT\|_q^q +
\|\nabla_HT\|_2^2 \right)+\|\nabla_H\partial_zT\|_2^2+
\left\||\nabla_HT|^{\frac{q}{2}-1}\nabla_H\partial_zT\right\|_2^2
\nonumber\\
  \leq&C (\|\nabla_H(\eta,\theta)\|_2^2+ 1) (\|\nabla _HT\|_q^q+1)\log(e +\|\Delta_H\eta\|_2+\|\Delta_H\theta\|_{2}+\|\nabla_HT\|_q)\\
&+C(\|\Delta_H\eta\|_2^2+\|\nabla_H\eta\|_2^2+1)(\|\nabla _HT\|_q^q+1),
\end{align*}
if $q\in(4,\infty)$, where $C$ is a positive number depending only on $h, q$, and $\|T_0\|_\infty$.
\end{proposition}

\begin{proof}
Integration by parts and using the H\"older inequality yield
\begin{align*}
  \int_\Omega|\partial_zT|^4dxdydz=-\int_\Omega\partial_z(|\partial_z T|^2\partial_z T)Tdxdydz
  \leq3\|T\|_\infty\|\partial_z^2T\|_2\|\partial_z T\|_4^2,
\end{align*}
which implies
\begin{equation}
  \label{zt4-0}\|\partial_zT\|_4^2\leq 3\|T\|_\infty\|\partial_z^2T\|_2.
\end{equation}
Multiplying equation (\ref{eq3}) by $-\partial_z^2T$ and integrating the resultant over $\Omega$, it follows from integration by parts, Proposition \ref{prop3.2}, (\ref{zt4-0}), and the H\"older and Young inequalities that
\begin{align}
  &\frac{1}{2}\frac{d}{dt}\|\partial_zT\|_2^2+\|\partial_z^2T\|_2^2+\varepsilon\|\nabla_H\partial_zT\|_2^2\nonumber\\
  =&\int_\Omega(v\cdot\nabla_HT+w\partial_zT)\partial_z^2Tdxdydz\nonumber\\
  =&\int_\Omega(v\cdot\nabla_HT\partial_z^2T+\frac{1}{2}(\nabla_H\cdot v)|\partial_zT|^2)dxdydz\nonumber\\
  =&\int_\Omega(v\cdot\nabla_HT\partial_z^2T+\frac{1}{2}(\eta-\Phi)|\partial_zT|^2)dxdydz\nonumber\\
  \leq&C(\|v\|_\infty\|\nabla_HT\|_2\|\partial_z^2T\|_2+\|\Phi\|_\infty\|\partial_zT\|_2^2+\|\eta\|_2
  \|\partial_zT\|_4^2)\nonumber\\
  \leq&C(\|v\|_\infty\|\nabla_HT\|_2\|\partial_z^2T\|_2+\|\Phi\|_\infty\|\partial_zT\|_2^2+\|\eta\|_2
  \|T\|_\infty\|\partial_z^2T\|_2)\nonumber\\
  \leq&\sigma \|\partial_z^2T\|_2^2 +C_\sigma(\|v\|_\infty^2\|\nabla_HT\|_2^2 +\|\partial_zT\|_2^2+\|\eta\|_2^2)\nonumber\\
  \leq&\sigma \|\partial_z^2T\|_2^2 +C_\sigma(\|v\|_\infty^2+1)
 (
  \|\nabla_HT\|_2^2 +\|\partial_zT\|_2^2+\|\eta\|_2^2+1),\label{nzt2}
\end{align}
for any $\sigma > 0$ (to be chosen later)  and for some $C_\sigma>0$.

Recalling the definitions of $\eta$ and $\Phi$, (\ref{etatheta}) and (\ref{Phi}), respectively, then by the H\"older inequality, one can easily check
$$
\|\nabla_H\partial_zw\|_2, \|\nabla_Hw\|_2\leq C(\|\nabla_H\eta\|_2+\|\nabla_HT\|_2).
$$
Thanks to the above inequality, multiplying equation (\ref{eq3}) by $-\Delta_HT$, and integrating the resultant over $\Omega$, it follows from integration by parts, Proposition \ref{prop3.2}, and the H\"older and Young inequalities that
\begin{align}
  &\frac{1}{2}\frac{d}{dt}\|\nabla_HT\|_2^2+\|\nabla_H\partial_zT\|_2^2
  +\varepsilon\|\Delta_HT\|_2^2\nonumber\\
  =&\int_\Omega(v\cdot\nabla_HT+w\partial_zT)\Delta_HTdxdydz\nonumber\\
  =&-\int_\Omega(\nabla_HT\cdot\nabla_Hv +\partial_zT\nabla_Hw)\cdot\nabla_HTdxdydz\nonumber\\
  =&-\int_\Omega[\nabla_HT\cdot\nabla_Hv\cdot\nabla_HT -(\nabla_H\partial_zw\cdot\nabla_HT+\nabla_Hw\cdot\partial_z\nabla_HT)T]
  dxdydz\nonumber\\
  \leq&\int_\Omega|\nabla_Hv||\nabla_HT|^2dxdydz+C(\|\nabla_H\partial_zw\|_2\|\nabla_HT\|_2 +\|\nabla_Hw\|_2\|\nabla_H\partial_zT\|_2)\nonumber\\
  \leq&\int_\Omega|\nabla_Hv||\nabla_HT|^2dxdydz+C(\|\nabla_H\eta\|_2+\|\nabla_HT\|_2)(
  \|\nabla_HT\|_2+\|\nabla_H\partial_zT\|_2)\nonumber\\
  \leq&\sigma\|\nabla_H\partial_zT\|_2^2+C_\sigma(\|\nabla_H\eta\|_2^2 +\|\nabla_HT\|_2^2)+\int_\Omega|\nabla_Hv||\nabla_HT|^2dxdydz, \label{nht2}
\end{align}
for any $\sigma > 0$ (to be chosen later) and for some $C_\sigma>0$.

Multiplying equation (\ref{eq3}) by $-\text{div}_H(|\nabla_HT|^{q-2}\nabla_HT)$, for $q\in(2,\infty)$,
integrating the resultant over $\Omega$, and noticing that $\|T\|_\infty\leq\|T_0\|_\infty$, it follows from integration by parts, (\ref{w}), (\ref{etatheta}), and (\ref{Phi}), that
\begin{align*}
  &\frac{1}{q}\frac{d}{dt}\|\nabla_HT\|_q^q+ \int_\Omega|\nabla_HT|^{q-2}(|\nabla_H\partial_zT|^2 +(q-2)|\partial_z|\nabla_HT||^2\nonumber\\
  &\qquad+\varepsilon|\nabla^2_HT|^2 +(q-2)\varepsilon|\nabla_H|\nabla_HT||^2)dxdydz\nonumber\\
  =&-\int_\Omega|\nabla_HT|^{q-2}(\nabla_HT\cdot\nabla_H v+\partial_zT\nabla_Hw) \cdot\nabla_HTdxdydz\nonumber\\
  =&-\int_\Omega[|\nabla_HT|^{q-2}\nabla_HT\cdot\nabla_H v\cdot \nabla_HT-T|\nabla_HT|^{q-2}\nabla_H\partial_zw\cdot\nabla_HT
  \nonumber\\
  &+T\nabla_Hw\cdot\partial_z(|\nabla_HT|^{q-2}\nabla_HT)]dxdydz\nonumber\\
  \leq&C\int_M\left(\int_{-h}^h|\nabla_H(\eta-\Phi)|dz\right)\left(\int_{-h}^h |\nabla_HT|^{q-2} |\nabla_H\partial_zT|dz\right)dxdy \nonumber\\
  &+C\int_\Omega \big(|\nabla_Hv||\nabla_HT|^q+|\nabla_H(\eta-\Phi)||\nabla_HT|^{q-1} \big)dxdydz,
\end{align*}
which, summed with (\ref{nht2}) and using the Young inequality, gives
\begin{align}
&\frac{d}{dt}\left(\frac{\|\nabla_HT\|_q^q}{q}+\frac{\|\nabla_HT\|_2^2}{2} \right)+\left\||\nabla_HT|^{\frac{q}{2}-1}\nabla_H\partial_zT\right\|_2^2+
\|\nabla_H\partial_zT\|_2^2+\varepsilon\|\Delta_HT\|_2^2  \nonumber \\
\leq&\sigma\|\nabla_H\partial_zT\|_2^2+C_\sigma(\|\nabla_H\eta\|_2^2 +\|\nabla_HT\|_2^2) \nonumber\\
&+C\int_\Omega |\nabla_Hv|(|\nabla_HT|^q+1)dxdydz+C\int_\Omega|\nabla_H(\eta-\Phi)||\nabla_HT|^{q-1} dxdydz \nonumber\\
  &+C\int_M\left(\int_{-h}^h|\nabla_H(\eta-\Phi)|dz\right)\left(\int_{-h}^h |\nabla_HT|^{q-2} |\nabla_H\partial_zT|dz\right)dxdy\nonumber\\
=:&\sigma\|\nabla_H\partial_zT\|_2^2+C_\sigma(\|\nabla_H\eta\|_2^2 +\|\nabla_HT\|_2^2)+C(J_1+J_2+J_3).\label{nhtq+2}
\end{align}

We have to estimate the terms $J_1, J_2$, and $J_3$ in (\ref{nhtq+2}).
First, we show that $J_2\leq (q-1)J_3$. In fact, since $T$ is odd and
periodic with respect to $z$, one has
$T|_{z=-h}=T|_{z=h}=-T|_{z=-h}=0$. Therefore, we have
\begin{equation*}
|\nabla_HT(x,y,z,t)|^{q-1}\leq (q-1)\int_{-h}^h|\nabla_HT(x,y,z,t)|^{q-2}|\nabla_H\partial_zT(x,y,z,t)|dz
\end{equation*}
and, thus,
\begin{align*}
  J_2=&\int_\Omega |\nabla_H(\eta-\Phi)||\nabla_HT|^{q-1}dxdydz \nonumber\\
\leq& (q-1)\int_M\int_{-h}^h|\nabla_H(\eta-\Phi)|dz\int_{-h}^h|\nabla_HT|^{q-2}|\nabla_H
\partial_zT|dzdxdy
= (q-1)J_3.
\end{align*}

Next, we estimate $J_3$.
By the H\"older and Minkowski inequalities, we have
\begin{align*}
  J_{3,1}:=&\int_M\left(\int_{-h}^h|\nabla_H\eta|dz\right)
  \left(\int_{-h}^h|\nabla_HT|^{q-2}|\nabla_H
  \partial_zT|dz\right)dxdy \nonumber\\
  \leq& \int_M\left(\int_{-h}^h|\nabla_H\eta|dz\right)
  \left(\int_{-h}^h|\nabla_HT|^{q-2}dz
  \right)^{\frac12}\left(\int_{-h}^h|\nabla_HT|^{q-2}|\nabla_H
  \partial_zT|^2dz\right)^{\frac12}dxdy\nonumber\\
  \leq& \left\|\int_{-h}^h|\nabla_H\eta|dz\right\|_{q,M}\left\| \left(\int_{-h}^h|\nabla_HT|^{q-2}dz
  \right)^{\frac12}\right\|_{\frac{2q}{q-2},M}\left\||\nabla_HT |^{\frac{q}{2}-1}\nabla_H\partial_zT\right\|_2\nonumber\\
  \leq& \left(\int_{-h}^h\|\nabla_H\eta\|_{q,M}dz\right)  \left(\int_{-h}^h\|\nabla_HT\|_{q,M}^{q-2}dz
  \right)^{\frac12}\left\||\nabla_HT |^{\frac{q}{2}-1}\nabla_H\partial_zT\right\|_2\nonumber\\
  \leq& C\left(\int_{-h}^h\|\nabla_H\eta\|_{q,M}dz\right) \|\nabla_HT\|_{q}^{\frac{q-2}{2}}\left\||\nabla_HT |^{\frac{q}{2}-1}\nabla_H\partial_zT\right\|_2
\end{align*}
and, similarly,
\begin{align*}
  J_{3,2}:=&\int_M\left( \int_{-h}^h|\nabla_HT|dz\right)\left( \int_{-h}^h|\nabla_HT|^{q-2}|\nabla_H
  \partial_zT|dz\right)dxdy \nonumber\\
  \leq& C\left(\int_{-h}^h\|\nabla_HT\|_{q,M}dz\right) \|\nabla_HT\|_{q}^{\frac{q-2}{2}}\left\||\nabla_HT |^{\frac{q}{2}-1}\nabla_H\partial_zT\right\|_2.
\end{align*}
Therefore, we have by the H\"older and Young inequalities that
\begin{align}
  J_{3,2}
  \leq&\|\nabla_HT\|_{q}^{\frac{q}{2}}\left\||\nabla_HT |^{\frac{q}{2}-1}\nabla_H\partial_zT\right\|_2\nonumber\\
  \leq&\sigma \left\||\nabla_HT |^{\frac{q}{2}-1}\nabla_H\partial_zT\right\|_2^2+C_\sigma
  \|\nabla_HT\|_q^q\label{nhtq2-1}
\end{align}
for any $\sigma>0$ (to be chosen later) and for some $C_\sigma>0$. Moreover, from the above and by the Gagaliardo-Nirenberg and H\"older inequalities we have
\begin{align*}
  J_{3,1}
  \leq& C\left(\int_{-h}^h\|\nabla_H\eta\|_{q,M}dz\right) \|\nabla_HT\|_{q}^{\frac{q-2}{2}}\left\||\nabla_HT |^{\frac{q}{2}-1}\nabla_H\partial_zT\right\|_2\nonumber\\
  \leq&C\left(\int_{-h}^h\|\nabla_H\eta\|_{2,M}^{\frac2q}\|\Delta_H\eta\|_{2,M}
  ^{1-\frac2q}dz\right)\|\nabla_HT\|_{q}^{\frac{q-2}{2}}\left\||\nabla_HT |^{\frac{q}{2}-1}\nabla_H\partial_zT\right\|_2\nonumber\\
  \leq&C\|\nabla_H\eta\|_{2}^{\frac2q}\|\Delta_H\eta\|_{2}
  ^{1-\frac2q}\|\nabla_HT\|_{q}^{\frac{q-2}{2}}\left\||\nabla_HT |^{\frac{q}{2}-1}\nabla_H\partial_zT\right\|_2.
\end{align*}
We further estimate $J_{3,1}$, in accordance with two different ranges, by the Young inequality as follows:
\begin{align*}
  J_{3,1}
  \leq&\sigma\left(\left\||\nabla_HT |^{\frac{q}{2}-1}\nabla_H\partial_zT\right\|_2^2+\|\Delta_H\eta\|_2^2\right) +C_\sigma\|\nabla_H\eta\|_2^2\|\nabla_HT\|_q^{\frac{q-2}{2}q}\nonumber\\
\leq&\sigma\left(\left\||\nabla_HT |^{\frac{q}{2}-1}\nabla_H\partial_zT\right\|_2^2+\|\Delta_H\eta\|_2^2\right) +C_\sigma\|\nabla_H\eta\|_2^2(\|\nabla_HT\|_q^q+1),
\end{align*}
if $q\in(2,4]$, and
\begin{align*}
  J_{3,1}
  \leq&\sigma\left\||\nabla_HT|^{\frac{q}{2}-1}\nabla_H\partial_zT \right\|_2^2+C_\sigma(\|\Delta_H\eta\|_2^2\|\nabla_HT\|_q^q
+\|\nabla_H\eta\|_2^2),
\end{align*}
if $q\in(4,\infty)$, for any $\sigma>0$ and for some $C_\sigma>0$.
Recalling
(\ref{nhtq2-1}) and noticing that $J_3\leq J_{3,1}+J_{3,2}$, we have
\begin{align}
  J_3
\leq&\sigma\left(\left\||\nabla_HT |^{\frac{q}{2}-1}\nabla_H\partial_zT\right\|_2^2+\|\Delta_H\eta\|_2^2\right) +C_\sigma(\|\nabla_H\eta\|_2^2+1)(\|\nabla_HT\|_q^q+1),\label{J3-}
\end{align}
if $q\in(2,4]$, and
\begin{align}
  J_3
\leq&\sigma\left\||\nabla_HT|^{\frac{q}{2}-1}\nabla_H\partial_zT \right\|_2^2+C_\sigma(\|\Delta_H\eta\|_2^2+\|\nabla_H\eta\|_2^2+1)( \|\nabla_HT\|_q^q+1), \label{J3+}
\end{align}
if $q\in(4,\infty)$, for any $\sigma>0$ and for some $C_\sigma>0$.

Finally, we estimate the term $J_1$. Recalling the decomposition $v=\zeta-\varpi$, we have
\begin{align*}
 J_1=& \int_\Omega|\nabla_Hv|(1+|\nabla_HT|^q)dxdydz \nonumber\\ \leq&\int_\Omega(|\nabla_H\zeta|+|\nabla_H\varpi|)(1+|\nabla_H T|^q)dxdydz\nonumber\\
  \leq&\left(\int_{-h}^h\|\nabla_H\zeta\|_{\infty,M}dz\right)\left(\sup_{-h\leq z\leq h}\|\nabla_HT(\cdot,z,t)\|_{q,M}^q+1\right)\nonumber\\
  &+\left(\sup_{-h\leq z\leq h}\|\nabla_H\varpi(\cdot,z,t)\|_{\infty,M}\right)(\|\nabla_HT\|_q^q+1).
\end{align*}
Recalling that $T|_{z=-h}=0$, it follows from the H\"older inequality that
\begin{align*}
\sup_{-h\leq z\leq h}\|\nabla_HT(\cdot,z,t)\|_{q,M}^q=&\sup_{-h\leq z\leq h}\int_M|\nabla_HT(\cdot,z)|^qdxdy\\
  =&q\sup_{-h\leq z\leq h}\int_M\left(\int_{-h}^z|\nabla_HT|^{q-2} \nabla_HT\cdot\nabla_H\partial_zTd\xi\right)dxdy\\
  \leq&q\int_M\int_{-h}^h|\nabla_HT|^{q-1}|\nabla_H\partial_zT|dzdxdy\nonumber\\
\leq& q\|\nabla_HT\|_q^{\frac{q}{2}}
  \left\||\nabla_HT|^{\frac{q}{2}-1}\nabla_H\partial_zT\right\|_2.
\end{align*}
With the aid of the above inequality, applying Proposition \ref{propzetavarpi}, and using the Young inequality, we obtain
\begin{align*}
  &\left(\int_{-h}^h\|\nabla_H\zeta\|_{\infty,M}dz\right)
  \left(\sup_{-h\leq z\leq h}\|\nabla_HT\|_{q,M}^q+1\right)\nonumber\\
  \leq&C(\|\nabla_H\eta\|_2 +\|\nabla_H\theta\|_2 +1 )
  \log^{\frac{1}{2}}(e+\|\Delta_H\eta\|_2+\|\Delta_H\theta\|_{2})\nonumber\\
  &\times\left(\|\nabla_HT\|_q^{\frac{q}{2}}
  \left\||\nabla_HT|^{\frac{q}{2}-1}\nabla_H\partial_zT\right\|_2
  +1\right)\nonumber\\
  \leq&\sigma\left\||\nabla_HT|^{\frac{q}{2}-1}\nabla_H \partial_zT\right\|_2^2 +C_\sigma(\|\nabla_H\eta\|_2^2+\|\nabla_H\theta\|_2^2+1)
  \nonumber\\
  &\times(\|\nabla_HT\|_q^q+1)\log(e+\|\Delta_H\eta\|_2
+\|\Delta_H\theta\|_{2})
\end{align*}
for any $\sigma>0$ and for some $C_\sigma>0$, and
\begin{align*}
  \left(\sup_{-h\leq z\leq h}\|\nabla_H\varpi(\cdot,z,t)\|_{\infty,M}\right)(\|\nabla_HT\|_q^q+1)
\leq C(\|\nabla_HT\|_q^q+1)\log(e+\|\nabla_HT\|_{q}).
\end{align*}
Therefore, we have
\begin{align}
 J_1
\leq&\sigma\left\||\nabla_HT|^{\frac{q}{2}-1}\nabla_H \partial_zT\right\|_2^2 +C_\sigma(\|\nabla_H\eta\|_2^2+\|\nabla_H\theta\|_2^2+1) (\|\nabla_HT\|_q^q+1)\nonumber\\
  &\times\log(e+\|\Delta_H\eta\|_2
+\|\Delta_H\theta\|_{2}+\|\nabla_HT\|_q) \label{J1}
\end{align}
for any $\sigma>0$ and for some $C_\sigma>0$.

Thanks to the estimates for $J_1$ and $J_3$, i.e. (\ref{J3-})--(\ref{J1}), and recalling that $J_2\leq(q-1)J_3$, it follows from (\ref{nhtq+2}) that
\begin{align}
  &\frac{d}{dt}\left(\frac{\|\nabla_HT\|_q^q}{q} +\frac{\|\nabla_HT\|_2^2}{2}\right)+\|\nabla_H\partial_zT\|_2^2+
\left\||\nabla_HT|^{\frac{q}{2}-1}\nabla_H\partial_zT\right\|_2^2
+\varepsilon\|\Delta_HT\|_2^2\nonumber\\
  \leq&\sigma\left(\|\nabla_H\partial_zT\|_2^2+ \|\Delta_H\eta\|_2^2+\left\||\nabla_HT|^{\frac{q}{2}-1} \nabla_H\partial_zT\right\|_2^2\right)+C_\sigma (\|\nabla_H(\eta,\theta)\|_2^2+ 1)  \nonumber\\
  &\times(\|\nabla _HT\|_q^q+1)\log(e +\|\Delta_H\eta\|_2+\|\Delta_H\theta\|_{2}+\|\nabla_HT\|_q)\label{nhtq+2-}
\end{align}
for any $\sigma>0$ and for some $C_\sigma>0$, if $q\in(2,4]$, and
\begin{align}
  &\frac{d}{dt}\left(\frac{2}{q}\|\nabla_HT\|_q^q +
\|\nabla_HT\|_2^2 \right)+\|\nabla_H\partial_zT\|_2^2+
\left\||\nabla_HT|^{\frac{q}{2}-1}\nabla_H\partial_zT\right\|_2^2
\nonumber\\
  \leq&C (\|\nabla_H(\eta,\theta)\|_2^2+ 1) (\|\nabla _HT\|_q^q+1)\log(e +\|\Delta_H\eta\|_2+\|\Delta_H\theta\|_{2}+\|\nabla_HT\|_q)\nonumber\\
&+C(\|\Delta_H\eta\|_2^2+\|\nabla_H\eta\|_2^2+1)(\|\nabla _HT\|_q^q+1), \label{nhtq+2+}
\end{align}
if $q\in(4,\infty)$.

The first conclusion follows from summing (\ref{nzt2}) and (\ref{nhtq+2-}), and the second one follows from
(\ref{nhtq+2+}).
\end{proof}

\subsection{A priori estimates on $(\nabla_H\eta,\nabla_H\theta,\nabla T)$}
Combining the energy inequalities established in the previous two
subsections and applying the logarithmic type Gronwall inequality, i.e.,
Lemma \ref{LogGron1}, we are able to obtain the required a priori estimates
on $\nabla_H\eta$, $\nabla_H\theta$, and $\nabla T$. In fact, we have the
following proposition.

\begin{proposition}\label{propapriALL-N2T}
Given $\mathcal T\in(0,\infty)$. There is a positive number $\varepsilon_0
\in(0,1)$ depending only on $h$ and $\|T_0\|_\infty$, such that, for any $\varepsilon\in(0,\varepsilon_0)$ and any $q\in(2,\infty)$, we have the
following estimate:
\begin{align*}
\sup_{0\leq t\leq\mathcal T}(\|\nabla_H(\eta,\theta)\|_2^2&+\|\nabla T\|_2^2+\|\nabla_HT\|_q^q)+\int_0^{\mathcal T}(\|\Delta_H(\eta,\theta)\|_2^2\nonumber\\
&+\|(\partial_z^2 T,\nabla_H\partial_zT)\|_2^2+\varepsilon\|(\nabla_H\partial_z\eta, \nabla_H\partial_z\theta,\Delta_HT)dt\leq C,
\end{align*}
where $C$ is a positive constant depending only on $h, \mathcal T,$ and $\|v_0\|_{H^2}+\|T_0\|_{H^1\cap L^\infty}+\|\nabla_HT_0\|_q$; in particular, $C$ is independent of $\varepsilon\in(0,\varepsilon_0)$.
\end{proposition}

\begin{proof}
We first consider the case that $q\in(2,4]$. By Proposition \ref{propnhetatheta} and Proposition \ref{propnt}, where we choose $\sigma =\frac{1}{4}$, we have
\begin{align*}
&\frac{d}{dt}\|\nabla_H(\eta,\theta)\|_2^2 +\|(\Delta_H\eta,\Delta_H\theta,\sqrt\varepsilon \nabla_H\partial_z\eta, \sqrt\varepsilon \nabla_H\partial_z\theta)\|_2^2\\
  \leq&C(\|v\|_\infty^2+\|u\|_4^2 +\|(\varphi,\psi)\|_2 \|\nabla_H(\varphi,\psi)\|_2+\|(\varphi,\psi)\|_2^2+1)
\nonumber\\
   &\times\|\nabla_H(\eta,\theta, T)\|_2^2+C(\|(\eta,\theta,\varphi,\psi)\|_2^2+\|u\|_4^4+1)\nonumber\\
   & \times(\|\nabla_H(v,\eta, \theta,\varphi,\psi)\|_2^2 +\|\partial_zT\|_2^2+1) +C\varepsilon^2\|\Delta_HT\|_2^2,
\end{align*}
where $C$ is a positive constant depending only on $h$ and $\|T_0\|_\infty$, and
  \begin{align*}
  &\frac{d}{dt}\left( \frac{\|\nabla T\|_2^2}{2} +\frac{\|\nabla_HT\|_q^q}{q}\right)+\|(\partial_z^2T,\nabla_H\partial_z T,\sqrt\varepsilon \Delta_HT)\|_2^2
  \nonumber\\
  \leq&C(\|\nabla_H(\eta,\theta)\|_2^2+1) (\|\nabla_HT\|_q^q+1)\log(e +\|\Delta_H(\eta,\theta)\|_{2}+\|\nabla_HT\|_q) \nonumber\\
&+C(\|v\|_\infty^2+1)\|\nabla T\|_2^2
+C(\|v\|_\infty^2+1)(\|\eta\|_2^2+1)+\frac14\left(\|\Delta_H\eta\|_2^2+\|\nabla_H\partial_z T\|_2^2\right)
\end{align*}
provided $q\in(2,4]$, where $C$ is a positive
number depending only on $h, q,$ and $\|T_0\|_\infty$. Choose a
small positive number $\varepsilon_0\in(0,1)$ depending only on $h$ and
$\|T_0\|_\infty$ and let $\varepsilon\in(0,\varepsilon_0)$. Summing the
above two inequalities and denoting
\begin{eqnarray*}
  &A_3=\|\nabla_H(\eta,\theta)\|_2^2+\frac{\|\nabla T\|_2^2}{2}+\frac{\|
\nabla_HT\|_q^q}{q}, \\
&B_3=\frac12\|(\Delta_H\eta,\Delta_H\theta,\sqrt\varepsilon \nabla_H\partial_z\eta, \sqrt\varepsilon \nabla_H\partial_z\theta)\|_2^2+\|(\partial_z^2T,\nabla_H\partial_z T,\sqrt\varepsilon \Delta_HT)\|_2^2,\\
&\ell_3(t)=(\|v\|_\infty^2+\|u\|_4^2 +\|(\varphi,\psi)\|_2 \|\nabla_H(\varphi,\psi)\|_2+\|(\varphi,\psi)\|_2^2)(t)+1,\\
&n_3(t)=\|\nabla_H(\eta,\theta)\|_2^2(t)+1,\quad f_3(t)=(\|v\|_\infty^2(t)+1)(\|\eta\|_2^2(t)+1),
\end{eqnarray*}
one obtains
\begin{equation}\label{ineqA3B3}
A_3'+B_3\leq C(\ell_3(t)+n_3(t)\log(A_3+B_3+e))A_3+Cf_3(t).
\end{equation}
Recalling (\ref{vinf}), i.e., $\|v\|_\infty^2\leq\log (A_2+B_2)$, where
$A_2$ and $B_2$ are given by (\ref{A2}) and (\ref{B2}), respectively, and noticing that $\log z\leq\log(1+z)\leq z$, for $z>0$, we have by Corollary \ref{apriH1} that
\begin{equation}
\int_0^{\mathcal T}\|v\|_\infty^2dt\leq C\int_0^{\mathcal T} \log(A_2+B_2)dt\leq C\int_0^{\mathcal T}(A_2+B_2)dt\leq C.\label{vL2Linfty}
\end{equation}
With the aid of this, and applying Corollary \ref{apriH1} and Proposition \ref{aprihu}, we have
\begin{equation*}
  \int_0^{\mathcal T}(\ell_3(t)+n_3(t)+f_3(t))dt\leq C
\end{equation*}
for a positive constant $C$ depending only on $h, \mathcal T, \|(v_0, T_0)\|_\infty,$ and $\|\nabla_Hv_0\|_2+\|\partial_zv_0\|_{H^1}$. Thanks to the above and noticing that $n_3\leq A_3$, one can apply
Lemma \ref{LogGron1} to (\ref{ineqA3B3}) and obtains
\begin{equation}\label{E4}
\sup_{0\leq t\leq T}A_3(t)+\int_0^{\mathcal T}B_3(t)dt\leq C,
\end{equation}
where $C$ is a positive constant depending only on $h, \mathcal T,$
and $\|v_0\|_{H^2}+\|T_0\|_{H^1\cap L^\infty}+\|\nabla_HT_0\|_q$.
This proves the conclusion for the case $q\in(2,4]$.

We now consider the case when $q\in(4,\infty)$. Thanks to what we
have proven in (\ref{E4}), we have
\begin{align}
&\sup_{0\leq t\leq\mathcal T}(\|\nabla_H(\eta,\theta)\|_2^2(t)+\|\nabla T\|_2^2(t)+\|\nabla_HT\|_4^4(t))\nonumber\\
&+\int_0^{\mathcal T}(\|\Delta_H(\eta,\theta)\|_2^2+\|(\partial_z^2 T,\nabla_H\partial_zT)\|_2^2)dt\leq C, \label{apriALL-N2T'}
\end{align}
where $C$ is a positive constant depending only on $h, \mathcal T,$
and $\|v_0\|_{H^2}+\|T_0\|_{H^1\cap L^\infty}+\|\nabla_HT_0\|_4$.
One still need to show the a priori $L^\infty(0,\mathcal T; L^q)$ estimate
on $\nabla_HT$, for $q\in(4,\infty)$. By Proposition \ref{propnt} and noticing that
\begin{align*}
  &\log(e+\|\Delta_H\eta\|_2+\|\Delta_H\theta\|_2+\|\nabla_HT\|_q)\\
\leq&\log[(e+\|\nabla_HT\|_q)(1+\|\Delta_H\eta\|_2+\|\Delta_H\theta\|_2)]\\
\leq&\log(e+\|\nabla_HT\|_q)+\log(1+\|\Delta_H\eta\|_2+\|\Delta_H\theta\|_2)\\
\leq&\log(e+\|\nabla_HT\|_q)+\|\Delta_H\eta\|_2+\|\Delta_H\theta\|_2,
\end{align*}
we have, for $q\in(4,\infty)$,
\begin{align*}
  &\frac{d}{dt}\left(\frac{2}{q}\|\nabla_HT\|_q^q +
\|\nabla_HT\|_2^2 \right)+\|\nabla_H\partial_zT\|_2^2
\nonumber\\
  \leq&C (\|\nabla_H(\eta,\theta)\|_2^2+ 1) (\|\nabla _HT\|_q^q+1)\log(e +\|\nabla_HT\|_q)\\
&+C(\|\nabla_H(\eta,\theta)\|_2^2+ 1)(\|\Delta_H\eta\|_2+\|\Delta_H\theta\|_{2})(\|\nabla _HT\|_q^q+1)\\
&+C(\|\Delta_H\eta\|_2^2+\|\nabla_H\eta\|_2^2+1)(\|\nabla _HT\|_q^q+1),
\end{align*}
from which, denoting
\begin{eqnarray*}
 & A_4=\frac{2}{q}\|\nabla_HT\|_q^q +
\|\nabla_HT\|_2^2,\quad B_4=\|\nabla_H\partial_zT\|_2^2,\\
&\ell_4(t)=(\|\nabla_H(\eta,\theta)\|_2^2+ 1)(\|\Delta_H\eta\|_2+\|\Delta_H\theta\|_{2}) +\|\Delta_H\eta\|_2^2+\|\nabla_H\eta\|_2^2+1,\\
&m_4(t)=\|\nabla_H(\eta,\theta)\|_2^2+ 1,
\end{eqnarray*}
one obtains
$$
A_4'+B_4\leq C(\ell_4(t)+m_4(t)\log A_4)A_4.
$$
Thanks to (\ref{apriALL-N2T'}) and applying Lemma \ref{LogGron1} to the above inequality, we have
$$
\sup_{0\leq t\leq\mathcal T}\|\nabla_HT\|_q(t)\leq C,
$$
where $C$ is a positive constant depending only on $h, \mathcal T,$
and $\|v_0\|_{H^2}+\|T_0\|_{H^1\cap L^\infty}+\|\nabla_HT_0\|_q$. Thus, this proves the case that $q\in(4,\infty)$.
\end{proof}
\subsection{A priori estimates on $\nabla^2T$} This subsection is devoted
to establishing the a priori estimates on the second order spatial
derivatives of $T$, which is stated in the following proposition:

\begin{proposition}
    \label{prop3.6}
Given a positive time $\mathcal T\in(0,\infty)$, let $\varepsilon_0\in(0,1)$
be the constant given in Proposition \ref{propapriALL-N2T}, and assume that $\varepsilon\in(0,\varepsilon_0)$. The following a priori estimate holds:
\begin{align*}
&\sup_{0\leq t\leq\mathcal T}\|\nabla^2T\|_2^2(t)+\int_0^{\mathcal T}
(\|\partial_z\nabla^2 T\|_2^2+\varepsilon\|\nabla_H\nabla^2T\|_2^2) dt\leq C
\end{align*}
for a positive constant $C$ depending only on $h, \mathcal T,$ and $\|(v_0,T_0)\|_{H^2}$; in particular, $C$ is independent of $\varepsilon\in(0,\varepsilon_0)$.
\end{proposition}

\begin{proof}
By virtue of (\ref{eq2}), one can easily check that
$$
|\nabla_H\cdot v(x,y,z,t)|\leq\int_{-h}^h|\nabla_H\cdot u(x,y,\xi,t)|d\xi.
$$
By the aid of this inequality, differentiating equation (\ref{eq3}) with respect to $z$, multiplying the resulting equation by $-\partial_z^3T$, and integrating over $\Omega$, it follows from integration by parts and using the H\"older inequality that
\begin{align}
  &\frac{1}{2}\frac{d}{dt}\|\partial_z^2T\|_2^2+\|\partial_z ^3T\|_2^2+\varepsilon\|\nabla_H\partial_z^2T\|_2^2\nonumber\\
  =&\int_\Omega\partial_z(v\cdot\nabla_HT+w\partial_zT)\partial_z^3Tdxdydz\nonumber\\
  =&\int_\Omega(v\cdot\nabla_H\partial_zT+\partial_zv\cdot\nabla_HT-(\nabla_H\cdot v)\partial_zT+w\partial_z^2T)\partial_z^3Tdxdydz\nonumber\\
  =&\int_\Omega[(v\cdot\nabla_H\partial_zT+u\cdot\nabla_HT-(\nabla_H\cdot v)\partial_zT)\partial_z^3T+\frac{1}{2}(\nabla_H\cdot v)|\partial_z^2T|^2]dxdydz\nonumber\\
  =&\int_\Omega[(v\cdot\nabla_H\partial_zT+u\cdot\nabla_HT-(\nabla_H\cdot v)\partial_zT)\partial_z^3T-v\cdot\nabla_H\partial_z^2T\partial_z^2T]dxdydz\nonumber\\
  \leq&(\|v\|_\infty\|\nabla_H\partial_zT\|_2+\|u\|_4\|\nabla_HT\|_4)\|\partial_z^3T\|_2 +\|v\|_\infty \|\nabla_H\partial_z^2T\|_2\|\partial_z^2T\|_2\nonumber\\
  &+\int_M\left(\int_{-h}^h|\nabla_H\cdot  u|dz\right)\left(\int_{-h}^h|\partial_zT||\partial_z^3T|dz\right)dxdy.  \label{nazt2eps1}
\end{align}

By Lemma \ref{ladlemma} and Lemma \ref{lem2.3}, and recalling that $\|T\|_\infty\leq\|T_0\|_\infty$, guaranteed by Proposition \ref{prop3.2}, it follows from the H\"older inequality that
\begin{align}
  &\int_M\left(\int_{-h}^h|\nabla_H \cdot u|dz\right)\left(\int_{-h}^h|\partial_zT||\partial_z^3T|dz\right)dxdy\nonumber\\
\leq&\int_M\left(\int_{-h}^h(|\varphi|+|T|)dz\right)\left( \int_{-h}^h|\partial_zT||\partial_z^3T|dz\right)dxdy\nonumber\\
  \leq& \left(\int_{-h}^h\|(\varphi,T)\|_{4,M}dz\right)\left(\int_{-h}^h\|\partial_zT\|_{4,M}^2dz\right)^{\frac{1}{2}} \|\partial_z^3T\|_2\nonumber\\
  \leq&C\left(\|\varphi\|_2^{\frac{1}{2}}\|\nabla_H\varphi\|_2^{\frac{1}{2}}+ \|\varphi\|_2+1\right)\|\partial_zT\|_4\|\partial_z^3T\|_2.\label{add2}
\end{align}
Recalling (\ref{zt4-0}), i.e.,
\begin{equation}
  \label{zt4}\|\partial_zT\|_4^2\leq 3\|T\|_\infty\|\partial_z^2T\|_2.
\end{equation}
Similarly, we have
\begin{equation}
  \label{ht4}
\|\nabla_HT\|_4^2\leq 3\|T\|_\infty\|\Delta_HT\|_2.
\end{equation}

It follows from integration by parts and the H\"older and Cauchy-Schwarz inequalities that
\begin{align*}
\int_\Omega|\nabla_H\partial_zT|^2dxdydz=&\int_\Omega\Delta_HT\partial_z^2Tdxdydz\nonumber\\
  \leq&\|\Delta_H T\|_2\|\partial_z^2T\|_2\leq\frac{1}{2}(\|\Delta_HT\|_2^2 +\|\partial_z^2T\|_2^2)
\end{align*}
and, similarly, $\|\nabla_H\partial_z^2T\|_2^2\leq\frac{1}{2}(\|\partial_z^3T\|_2^2 +\|\Delta_H\partial_zT\|_2^2)$.
On account of these facts and recalling that $\|T\|_\infty\leq\|T_0\|_\infty$, guaranteed by Proposition \ref{prop3.2}, it follows from (\ref{nazt2eps1})--(\ref{ht4}) and using the Young inequality that
\begin{align*}
  &\frac{1}{2}\frac{d}{dt}\|\partial_z^2T\|_2^2+\|\partial_z ^3T\|_2^2+\varepsilon\|\nabla_H\partial_z^2T\|_2^2\nonumber\\
  \leq& (\|v\|_\infty\|\nabla_H\partial_zT\|_2 +\sqrt3\|u\|_4\|T\|_\infty^{\frac12}\|\Delta_HT\|_2^{\frac12}) \|\partial_z^3T\|_2 \nonumber\\
  &+\|v\|_\infty\|\nabla_H\partial_z^2T\|_2 \|\partial_z^2 T\|_2+C\left(\|\varphi\|_2^{\frac{1}{2}} \|\nabla_H\varphi\|_2^{\frac{1}{2}}+ \|\varphi\|_2+1\right) \|\partial_z^2T \|_2^{\frac{1}{2}}\|\partial_z^3T\|_2\nonumber\\
  \leq&\frac14(\|\partial_z^3T\|_2^2+\|\Delta_H\partial_zT\|_2^2) +C[\|\varphi\|_2^2  \|\nabla_H\varphi\|_2^2 + \|\varphi\|_2^4+1\nonumber\\
  &+( \|v\|_\infty^2+\|u\|_4^2)(\|\Delta_HT\|_2^2+\|\partial_z^2T\|_2^2 +1)]
\end{align*}
and, thus,
\begin{align}
  &\frac{1}{2}\frac{d}{dt}\|\partial_z^2T\|_2^2+\frac34\|\partial_z ^3T\|_2^2+\varepsilon\|\nabla_H\partial_z^2T\|_2^2\nonumber\\
  \leq&\frac14 \|\Delta_H\partial_zT\|_2^2 +C[\|\varphi\|_2^2  \|\nabla_H\varphi\|_2^2 + \|\varphi\|_2^4+1\nonumber\\
  &+( \|v\|_\infty^2+\|u\|_4^2)(\|\Delta_HT\|_2^2+\|\partial_z^2T\|_2^2 +1)].\label{nazt2eps}
\end{align}

Applying the horizontal gradient $\nabla_H$ to equation (\ref{eq3}), multiplying the resulting equation by $-\nabla_H\Delta_HT$, and integrating over $\Omega$, it follows from integrating by parts and the H\"older inequality that
\begin{align}
  &\frac{1}{2}\frac{d}{dt}\|\Delta_HT\|_2^2+\|\Delta_H
\partial_zT\|_2^2+\varepsilon\|\nabla_H\Delta_H T\|_2^2\nonumber\\
  =&-\int_\Omega\Delta_H(v\cdot\nabla_HT+w\partial_zT)\Delta_HTdxdydz\nonumber\\
  =&-\int_\Omega(2\nabla_Hv\cdot\nabla_H^2T
  +\Delta_Hv\cdot\nabla_HT +2\nabla_Hw\cdot\nabla_H\partial_zT+\Delta_Hw\partial_zT)\Delta_HTdxyddx\nonumber\\
  =&-\int_\Omega(2\nabla_Hv\cdot\nabla_H^2T+\Delta_Hv\cdot\nabla_HT)\Delta_HTdxdydz\nonumber\\
  &+2\int_\Omega(-\nabla_H(\nabla_H\cdot v)\cdot\nabla_HT\Delta_HT+\nabla_Hw\cdot\nabla_HT\Delta_H\partial_zT)dxdydz\nonumber\\
  &+\int_\Omega(-\Delta_H(\nabla_H\cdot v)T\Delta_HT+\Delta_HwT\Delta_H\partial_zT)dxdydz\nonumber\\
  \leq&2\int_\Omega|\nabla_Hv||\nabla_H^2T|^2dxdydz+3\int_\Omega|\nabla_H^2v||\nabla_HT||\Delta_HT|dx dydz\nonumber\\
  &+2\int_\Omega\left(\int_{-h}^z|\nabla_H(\nabla_H\cdot v)|d\xi\right)|\nabla_HT||\Delta_H\partial_zT|dxdydz\nonumber\\
  &+\|T\|_\infty(\|\Delta_H(\nabla_H\cdot v)\|_2\|\Delta_HT\|_2+\|\Delta_Hw\|_2\|\Delta_H\partial_zT\|_2).\label{nah2teps1}
\end{align}

Next, we are going to estimate the terms on the right-hand side of the above inequality.
Recalling that $T|_{z=-h}=0$, we have
$$
|\Delta_HT(x,y,z,t)|\leq\int_{-h}^h|\Delta_HT(x,y,\xi,t)|d\xi.
$$
Recalling the definitions of $\eta, \theta$ and $\Phi$, (\ref{etatheta})
and (\ref{Phi}), it follows from the two-dimensional horizontal elliptic
estimates, the Ladyzhenskaya and Poincar\'e inequalities that
\begin{align*}
  \|\nabla_H^2v\|_{4,M}^2\leq&C\|\Delta_Hv\|_{4,M}^2=C\|\nabla_H\nabla_H\cdot v-\nabla_H^\perp\nabla_H^\perp\cdot v\|_{4,M}^2\\
  \leq&C(\|\nabla_H(\nabla_H\cdot v)\|_{4,M}^2+\|\nabla_H(\nabla_H^\perp\cdot v)\|_{4,M}^2)\\
  \leq&C(\|\nabla_H\eta\|_{4,M}^2+\|\nabla_H\Phi\|_{4,M}^2+\|\nabla_H\theta\|_{4,M}^2)\\
  \leq&C(\|\nabla_H(\eta,\theta)\|_{2,M}\|\Delta_H(\eta,\theta)\|_{2,M} +\|\nabla_H\Phi\|_{4,M}^2).
\end{align*}
On account of the above inequality, applying Lemma \ref{ladlemma}, recalling that $\|T\|_\infty\leq\|T_0\|_\infty$, and using (\ref{ht4}), it follows from the H\"older and Young inequalities that
\begin{align}
  &3\int_\Omega|\nabla_H^2v||\nabla_HT||\Delta_HT|dxdydz\nonumber\\
  \leq&C\int_M\left(
  \int_{-h}^h|\nabla_H^2v||\nabla_HT|dz\right)\left(
  \int_{-h}^h|\Delta_H\partial_zT|dz\right)dxdy\nonumber\\
  \leq&C\left(\int_{-h}^h\|\nabla_HT\|_{4,M}^2dz\right)^{\frac{1}{2}} \left(\int_{-h}^h\|\nabla_H^2v\|_{4,M}^2dz\right)^{\frac{1}{2}} \left(\int_{-h}^h\|\Delta_H\partial_zT\|_{2,M}dz\right)\nonumber\\
  \leq&C\|\nabla_HT\|_4\bigg[\int_{-h}^h(\|\nabla_H(\eta,\theta)\|_{2,M} \|\Delta_H(\eta,\theta)\|_{2,M}+\|\nabla_H\Phi\|_{4,M}^2)dz\bigg]^{\frac{1}{2}}\|\Delta_H\partial _zT\|_2\nonumber\\
  \leq&C\|T\|_\infty^{\frac12}\|\Delta_HT\|_2^{\frac12} (\|\nabla_H(\eta,\theta)\|_2^{\frac{1}{2}} \|\Delta_H(\eta,\theta)\|_2^{\frac{1}{2}} +\|\nabla_HT\|_4)\|\Delta_H \partial_zT\|_2\nonumber\\
\leq&C\|T\|_\infty^{\frac12}\|\Delta_HT\|_2^{\frac12} (\|\nabla_H(\eta,\theta)\|_2^{\frac{1}{2}} \|\Delta_H(\eta,\theta)\|_2^{\frac{1}{2}} +\|T\|_\infty^{\frac12}\|\Delta_HT\|_2^{\frac12})\|\Delta_H \partial_zT\|_2\nonumber\\
  \leq&\frac{1}{16}\|\Delta_H\partial_zT\|_2^2 +C(\| \nabla_H(\eta,\theta)\|_2^2\|\Delta_H(\eta,\theta)\|_2^2 +\|\Delta_HT\|_2^2). \label{3.50}
\end{align}

By Lemma \ref{ladlemma} and Lemma \ref{lem2.3}, it follows from (\ref{ht4}) and the H\"older and Young inequalities that
\begin{align}
  &2\int_\Omega\left(\int_{-h}^z|\nabla_H(\nabla_H\cdot v)|d\xi\right)|\nabla_HT||\Delta_H\partial_zT|dxdydz\nonumber\\
  \leq&4\int_M\left(\int_{-h}^h(|\nabla_H \eta|+|\nabla_HT|)dz\right)
  \left( \int_{-h}^h|\nabla_HT||\Delta_H\partial_zT|dz\right)dxdy \nonumber\\
  \leq&4\left(\int_{-h}^h(\|\nabla_H\eta\|_{4,M}+\|\nabla_HT\|_{4,M})dz\right)
  \left(\int_{-h}^h\|\nabla_HT\|_{4,M} ^2dz\right)^{\frac{1}{2}}\|\Delta_H\partial_zT\|_2\nonumber\\
  \leq&C(\|\nabla_H\eta\|_2^{\frac{1}{2}}\|\Delta_H\eta\|_2^{\frac{1}{2}}
  +\|\nabla_HT\|_4)\|\nabla_H T\|_4
  \|\Delta_H\partial_zT\|_2\nonumber\\
  \leq&\frac{1}{16}\|\Delta_H\partial_zT\|_2^2 +C (\|\nabla_H T\|_4^4+\|\nabla_H\eta\|_2^2\|\Delta_H\eta\|_2^2)\nonumber\\
\leq&\frac{1}{16}\|\Delta_H\partial_zT\|_2^2 +C (\|\Delta_H T\|_2^2+\|\nabla_H\eta\|_2^2\|\Delta_H\eta\|_2^2).\label{3.51}
\end{align}

Recalling the definitions of $\eta$ and $\Phi$, (\ref{etatheta}) and (\ref{Phi}), respectively, and using the Young inequality, one has
\begin{align}
  &\|T\|_\infty(\|\Delta_H(\nabla_H\cdot v)\|_2\|\Delta_HT\|_2+\|\Delta_Hw\|_2\|\Delta_H\partial_zT\|_2)\nonumber\\
  \leq&C[(\|\Delta_H\eta\|_2+\|\Delta_HT\|_2)\|\Delta_HT\|_2+(\|\Delta_H\eta\|_2+\|\Delta_HT\|)\|\Delta _H\partial_zT\|_2]\nonumber\\
  \leq&\frac{1}{16}\|\Delta_H\partial_zT\|_2 +C (\|\Delta_H\eta\|_2^2+\|\Delta_HT\|_2^2).\label{3.52}
\end{align}

Recalling that $\Delta_HT|_{z=-h}=0$, we have
\begin{align*}
  &\sup_{-h\leq z\leq h}\|\Delta_HT(\cdot, z,t)\|_{2,M}^2
  =\sup_{-h\leq z\leq h}\int_M|\Delta_HT(x,y,z,t)|^2dxdy\nonumber\\
  =&2\sup_{-h\leq z\leq h}\int_{-h}^z\int_M\Delta_HT\Delta_H\partial_zTdxdyd\xi\leq 2\|\Delta_HT\|_2\|\Delta_H\partial_zT\|_2.
\end{align*}
Thanks to the above, recalling the decomposition of $v$, i.e. (\ref{zeta}), it follows from Proposition \ref{propzetavarpi} and the Young inequality that
\begin{align*}
&2\int_\Omega|\nabla_Hv||\nabla_H^2T|^2dxdydz\nonumber\\
  \leq&2\int_{-h}^h(\|\nabla_H\zeta\|_{\infty, M}+\|\nabla_H\varpi\|_{\infty, M})\|\nabla_H^2T\|_{2,M}^2dz\nonumber\\
  \leq& 2\left(\int_{-h}^h\|\nabla_H\zeta\|_{\infty,M}dz\right)\left(\sup_{-h\leq z\leq h}\|\Delta_HT(\cdot,z,t)\|_{2,M}^2\right) \\
  &+2 \left(\sup_{-h\leq z\leq h}\|\nabla_H\varpi(\cdot,z,t)\|_{\infty, M}\right)\|\Delta_HT\|_2^2\nonumber\\
  \leq&C\|\Delta_HT\|_2\|\Delta_H\partial_zT\|_2(\|\nabla_H(\eta,\theta)\|_2+1)\log^{\frac{1}{2}}(e+ \|\Delta_H(\eta, \theta)\|_2)\nonumber\\
  &+C\|\Delta_HT\|_2^2\log(e+\|\nabla_HT\|_q)\nonumber\\
  \leq&\frac{1}{16}\|\Delta_H\partial_zT\|_2^2 +C(1+\|\nabla_H(\eta,\theta)\|_2^2)\|\Delta_HT\|_2^2\nonumber\\
&\times\log
  (e+\|\Delta_H(\eta,\theta)\|_2+\|\nabla_HT\|_4),
\end{align*}
from which, noticing that $\log z\leq\log(1+z)\leq z$, for $z>0$, and using the Young inequality, one obtains
\begin{align}
\int_\Omega|\nabla_Hv||\nabla_H^2T|^2dxdydz
\leq&\frac{1}{16}\|\Delta_H\partial_zT\|_2^2 +C(1+\|\nabla_H(\eta,\theta)\|_2^2)\|\Delta_HT\|_2^2\nonumber\\
&\times
  (1+\|\Delta_H(\eta,\theta)\|_2^2+\|\nabla_HT\|_4).\label{3.55}
\end{align}

Substituting (\ref{3.50})--(\ref{3.55}) into (\ref{nah2teps1}) yields
\begin{align*}
&\frac12\frac{d}{dt}\|\Delta_HT\|_2^2+ \frac34\|\partial_z\Delta _HT\|_2^2+\varepsilon\|\nabla_H\partial_z^2T\|_2^2\\
  \leq&C(1+\|\nabla_H(\eta,\theta)\|_2^2)(1+\|\Delta_H(\eta,\theta)\|_2^2
+\|\nabla_HT\|_4)(\|\Delta_HT\|_2^2+1).
\end{align*}
Summing the above
with (\ref{nazt2eps}) yields
\begin{align*}
  &\frac{d}{dt}\|(\partial_z^2T,\Delta_HT)\|_2^2+\|(\partial_z ^3T,\partial_z\Delta _HT,\sqrt\varepsilon\nabla_H\partial_z^2 T,\sqrt
\varepsilon\nabla_H\Delta_HT )\|_2^2\nonumber\\
\leq&C[(1+\|\nabla_H(\eta,\theta)\|_2^2)(1+\|\Delta_H(\eta,\theta)\|_2
+\|\nabla_HT\|_4) +\|v\|_\infty^2+\|u\|_4^2]\nonumber\\
&\times(\|\Delta_HT\|_2^2+\|\partial_z^2T\|_2^2 +1)+C(\|\varphi\|_2^2\|\nabla_H\varphi\|_2^2+\|\varphi\|_2^4+1),
\end{align*}
from which, by Corollary \ref{apriH1}, Proposition \ref{aprihu}, Proposition \ref{propapriALL-N2T}, recalling (\ref{vL2Linfty}), and using the Gronwall inequality, one obtains
$$
\sup_{0\leq t\leq\mathcal T}\|(\partial_z^2T,\Delta_HT)\|_2^2(t)+\int_0^{\mathcal T}\|(\partial_z ^3T,\partial_z\Delta _HT,\sqrt\varepsilon\nabla_H\partial_z^2 T,\sqrt
\varepsilon\nabla_H\Delta_HT )\|_2^2dt\leq C
$$
for a positive constant $C$ depending only on $h, \mathcal T,$ and $\|(v_0,T_0)\|_{H^2}$. The conclusion follows from the above estimates
by the elliptic estimates.
\end{proof}

\subsection{Uniform a prior estimates} With the aid of the energy inequalities established in the previous subsections, we can obtain the uniform estimates, which are independent of the regularization parameter $\varepsilon$, stated in the following proposition.

\begin{proposition}
  \label{prop3.7}
Given a positive time $\mathcal T\in(0,\infty)$ and let $\varepsilon_0\in(0,1)$ be as  in Proposition \ref{propapriALL-N2T}.
Suppose that $(v_0, T_0)\in H^2(\Omega)$ and $\varepsilon\in(0,\varepsilon_0)$. Let $(v,T)$ be the unique global strong solution to system (\ref{eq1})--(\ref{eq3}), subject to the boundary and initial conditions (\ref{BC1})--(\ref{IC}), and $u, \eta$ and $\theta$ the functions defined by (\ref{etatheta}). Define two quantities $\mathcal Q_1$ and $\mathcal Q_2$ as follows
$$
\mathcal Q_1:=\|v_0\|_{H^2}^2+\|T_0\|_{H^1}^2+\|\nabla_HT_0\|_q^q+\|T_0\|_\infty^2,\quad \mathcal Q_2:=\|v_0\|_{H^2}^2+\|T_0\|_{H^2}^2,
$$
where $q\in(2,\infty)$.

Then, for any $\varepsilon\in(0,\varepsilon_0)$, we have the following a priori
estimate:
\begin{align*}
\sup_{0\leq t\leq\mathcal T}&(\|v\|_{H^2}^2(t)+\|T\|_{H^1}^2(t)+\|\nabla_HT\|_q^q(t)+\|p_s\|_{H^1(M)(t)}^2 +\|T\|_\infty^2(t))\\
&+\int_0^{\mathcal T} (\|\nabla_Hu\|_{H^1}^2+\|\partial_zT\|_{H^1}^2+\|\partial_tv\|_{H^1}^2 +\|\partial_tT\|_2^2\\
&+\|\eta\|_{H^2}^2 +\|\theta\|_{H^2}^2+\|\partial_t\eta\|_2^2+\|\partial_t\theta\|_2^2)dt\leq C_1,
\end{align*}
for a positive constant $C_1$ depending only on $h, \mathcal T,$ and the upper bound of $\mathcal Q_1$, but is independent of $\varepsilon\in(0,\varepsilon)$, and
\begin{align*}
  \sup_{0\leq t\leq\mathcal T}\|T\|_{H^2}^2(t)+\int_0^{\mathcal T}(\|\nabla_Hv\|_{H^2}^2+\|\partial_zT\|_{H^2}^2+\|\partial_tT\|_{H^1}^2) dt\leq C_2,
\end{align*}
for a positive constant $C_2$, depending only on $h, \mathcal T$ and the upper bound of $\mathcal Q_2$, but is independent of $\varepsilon$.
\end{proposition}

\begin{proof}
By Proposition \ref{prop3.2}, Corollary \ref{apriH1}, Proposition \ref{aprizu}, Proposition \ref{aprihu}, and Proposition \ref{propapriALL-N2T}, we have the following
\begin{eqnarray}
&\|v\|_{L^\infty(0,\mathcal T; L^2)}+\|T\|_{L^\infty(0,\mathcal T; L^\infty)}+\|(\nabla_Hv,\partial_zT)\|_{L^2(0,\mathcal T; L^2)}\leq C, \label{APR1}\\
&\|(\eta,\theta)\|_{L^\infty(0,\mathcal T; L^2)}+\|u\|_{L^\infty(0,\mathcal T; L^4)}\nonumber\\
&+\|(\nabla_H\eta,\nabla_H\theta,\nabla_Hu,
\sqrt\varepsilon\partial_zu)\|_{L^2(0,\mathcal T; L^2)}\leq C,\label{APR2}\\
&\|\partial_zu\|_{L^\infty(0,\mathcal T; L^2)}+\|(\nabla_H\partial_zu,\sqrt\varepsilon
\partial_z^2u)\|_{L^2(0,\mathcal T; L^2)}\leq C,\label{APR3}\\
&\|(\varphi,\psi)\|_{L^\infty(0,\mathcal T; L^2)}+\|(\nabla_H\varphi,\nabla_H\psi)\|_{L^2(0,\mathcal T; L^2)}\leq C,\label{APR4}\\
&\|\nabla T\|_{L^\infty(0,\mathcal T; L^2)}+\|\nabla_HT\|_{L^\infty(0,
\mathcal T; L^q)}\nonumber\\
&+\|(\partial_z^2 T,\nabla_H\partial_zT,\sqrt\varepsilon\Delta_HT)\|_{L^2(0,\mathcal T; L^2)}\leq C,\label{APR5}\\
&\|(\nabla_H\eta,\nabla_H\theta)\|_{L^\infty(0,\mathcal T; L^2)}
+\|(\Delta_H\eta,\Delta_H\theta,\sqrt\varepsilon\partial_z^2\eta,
\sqrt\varepsilon\partial_z^2\theta)\|_{L^2(0,\mathcal T; L^2)}\leq C,\label{APR6}
\end{eqnarray}
where the constant $C$ depending only on $h, q, \mathcal T$, and $\mathcal Q_1$, but is independent of $\varepsilon$.

Thanks to (\ref{APR1})--(\ref{APR6}), and recalling the definitions of
$\eta, \theta, \varphi$, and $\psi$, it follows from the two-dimensional
horizontal elliptic estimate that
\begin{align}
  \|v\|_{L^\infty(0,\mathcal T; H^2)}\leq&C(\|v\|_{L^\infty(0,\mathcal T; L^2)}+\|\Delta_Hv\|_{L^\infty(0,\mathcal T; L^2)}+\|\partial_zu\|_{L^\infty(0,\mathcal T; L^2)})\nonumber\\
  \leq&C(1+\|\Delta_Hv\|_{L^\infty(0,\mathcal T; L^2)})\nonumber\\
  =&C(1+\|\nabla_H\nabla_H\cdot v-\nabla_H^\perp\nabla_H^\perp\cdot v\|_{L^\infty(0,\mathcal T; L^2)})\nonumber\\
  \leq&C(1+\|\nabla_H(\eta,\theta,T)\|_{L^\infty(0,\mathcal T; L^2)})\leq C,
  \label{APR7}\\
  \|\nabla_Hu\|^2_{L^2(0,\mathcal T; H^1)}=&\|\nabla_Hu\|_{L^2(0,\mathcal T; L^2)}^2+\|\nabla_H\partial_zu\|_{L^2(0,\mathcal T; L^2)}^2+\|\nabla_H^2u\|_{L^2(0,\mathcal T; L^2)}^2\nonumber\\
  =&\|\nabla_Hu\|_{L^2(0,\mathcal T; L^2)}^2+\|\nabla_H\partial_zu\|_{L^2(0,\mathcal T; L^2)}^2+\|\Delta_Hu\|_{L^2(0,\mathcal T; L^2)}^2\nonumber\\
  =&\|\nabla_H(u, \partial_zu)\|_{L^2(0,\mathcal T; L^2)}^2+
  \|\nabla_H\nabla_H\cdot u-\nabla_H^\perp\nabla_H^\perp\cdot u\|_{L^2(0,\mathcal T; L^2)}^2\nonumber\\
  \leq&\|\nabla_H(u, \partial_zu)\|_{L^2(0,\mathcal T; L^2)}^2
  +C\|\nabla_H(\varphi,\psi,T)\|_{L^2(0,\mathcal T; L^2)}^2\leq C,\label{APR8}
\end{align}
and
\begin{align}
  \|(\eta,\theta)\|_{L^2(0,\mathcal T; H^2)}\leq&C(\|(\eta,\theta)\|_{L^2(0,
  \mathcal T; L^2)}+\|\Delta_H(\eta,\theta)\|_{L^2(0,\mathcal T; L^2)}+ \|\partial_z^2(\eta,\theta)\|_{L^2(0,\mathcal T; L^2)})\nonumber\\
  \leq&C(1+\|\partial _z\nabla_Hu\|_{L^2(0,\mathcal T; L^2)}+\|\partial_zT
  \|_{L^2(0,\mathcal T; L^2)})\leq C,\label{APR9}
\end{align}
where the constant $C$ depending only on $h, q, \mathcal T$, and $\mathcal Q_1$, but is independent of $\varepsilon$.

With the aid of (\ref{w}), (\ref{etatheta}), (\ref{Phi}), (\ref{APR1}), and (\ref{APR2}), it follows from the Sobolev inequality, two-dimensional horizontal elliptic estimates, and the H\"older inequality that
\begin{align}
\|w\|_\infty=&\left\|\int_{_h}^z(\nabla_H\cdot v)d\xi\right\|_\infty\leq\int_{-h}^h\|\nabla_H\cdot v\|_{\infty,M}dz\nonumber\\
  \leq&\int_{-h}^h(\|\eta\|_{\infty,M}+\|\Phi\|_{\infty,M})dz\leq C\int_{-h}^h(\|\eta\|_{H^2(M)}+\|\Phi\|_{\infty,M})dz\nonumber\\
  \leq&C\int_{-h}^h(\|\eta\|_{2,M}+\|\Delta_H\eta\|_{2,M}+\|\Phi\|_{\infty, M})dz\nonumber\\
  \leq&C(\|\eta\|_{2 }+\|\Delta_H\eta\|_{2 }+\|\Phi\|_{\infty })\leq C(1+\|\Delta_H\eta\|_2),
  \label{3.63}
\end{align}where the constant $C$ depending only on $h, q, \mathcal T$, and $\mathcal Q_1$, but is independent of $\varepsilon$.

Thanks to (\ref{APR5}), (\ref{APR7}), and (\ref{3.63}), it follows from equation (\ref{eq3}) and using the Sobolev embedding inequality that
\begin{align*}
  \|\partial_tT\|_2^2\leq&(\|v\|_\infty^2\|\nabla_HT\|_2^2+\|w\|_\infty^2\|\partial_zT\|_2^2+\| \partial_z^2T\|_2^2+\varepsilon^2\|\Delta_HT\|_2^2)\nonumber\\
  \leq&C[\|v\|_{H^2}^2\|\nabla_HT\|_2^2+(1+\|\Delta_H\eta\|_2^2)\|\partial_zT\|_2^2 +\|\partial_z^2T\|_2^2+\varepsilon^2\|\Delta_HT\|_2^2]\nonumber\\
  \leq&C(1+\|\Delta_H\eta\|_2^2+\|\partial_z^2T\|_2^2+\varepsilon^2 \|\Delta_HT\|_2^2)
\end{align*}
and, thus, recalling (\ref{APR5}) and (\ref{APR9}), we have
\begin{equation}
  \label{patt}
  \|\partial_tT\|_{L^2(0,\mathcal T; L^2)}\leq C,
\end{equation}
where the constant $C$ depending only on $h, q, \mathcal T$, and $\mathcal Q_1$, but is independent of $\varepsilon$.
Recalling that $p_s(x,y,t)$ satisfies (see Appendix A (\ref{apxa3}))
\begin{equation*}
\left\{
\begin{array}{l}
-\Delta_Hp_s=\frac{1}{2h}\nabla_H\cdot\int_{-h}^h
\left(\nabla_H\cdot(v\otimes v)+f_0\overrightarrow{k}\times v-\int_{-h}^z\nabla_HTd\xi\right)dz,\\
\int_Mp_s(x,y,t)dxdy=0,\quad p_s\mbox{ is periodic in }x,y.
\end{array}
\right.
\end{equation*}
By elliptic estimates, Poincar\'e inequality, and recalling (\ref{APR5}) and (\ref{APR7}), we have
\begin{align}
  &\|p_s\|_{H^1(M)}^2=(\|p_s\|_{2,M}^2+\|\nabla_Hp_s\|_{2,M}^2)\nonumber\\
  \leq& C\|\nabla_Hp_s\|_{2,M}^2\leq C(\|\nabla_HT\|_2^2+\|\nabla_H\cdot(v\otimes v)\|_2^2+\|v\|_2^2) \nonumber\\
  \leq&C(\|\nabla_HT\|_2^2+\|v\|_\infty^2\|\nabla_Hv\|_2^2+\|v\|_2^2) \leq C(1+\|v\|_{H^2}^2)\leq C,\label{hps}
\end{align}
where the constant $C$ depending only on $h, q, \mathcal T$, and
$\mathcal Q_1$, but is independent of $\varepsilon$. Therefore,
recalling (\ref{APR5}), (\ref{APR7}), (\ref{3.63}), it follows
from (\ref{eq1}) and the Sobolev inequality that
\begin{align*}
  \|\partial_tv\|_2^2\leq&C(\|v\|_\infty^2\|\nabla_Hv\|_2^2+\|w\|_\infty^2\|\partial_zv\|_2^2+\| \Delta_Hv\|_2^2+\varepsilon^2\|\partial_z^2v\|_2^2\nonumber\\
  &+\|v\|_2^2+\|\nabla_Hp_s\|_2^2 +\|\nabla_HT\|_2^2)\nonumber\\
  \leq&C(\|v\|_{H^2}^2+1+\|\Delta_H\eta\|_2)
  \leq  C(1+\|\Delta_H\eta\|_2^2), \label{patv}
\end{align*}
which, recalling (\ref{APR9}), gives
\begin{equation}
  \|\partial_tv\|_{L^2(0,\mathcal T; L^2)}\leq C, \label{patv}
\end{equation}
where the constant $C$ depending only on $h, q, \mathcal T$,
and $\mathcal Q_1$, but is independent of $\varepsilon$.

For simplifying the notations, we introduce $S$ and $R$ as follows
\begin{eqnarray*}
  S=(v\cdot\nabla_H)v+w\partial_zv+f_0\overrightarrow{k}\times v, \quad R=\int_{-h}^z(\nabla_H\cdot(vT)-\varepsilon\Delta_HT)d\xi.
\end{eqnarray*}
By the H\"older and Sobolev embedding inequalities, we have
\begin{align*}
  &\int_\Omega|\nabla_H w|^2|\partial_zv|^2dxdydz\nonumber\\
  \leq& C\int_M\left(\int_{-h}^h (|\nabla_H\eta|+|\nabla_HT|)dz\right)^2\left( \int_{-h}^h|u|^2dz\right)dxdy\nonumber\\
  \leq& C\left\|\int_{-h}^h|u|^2dz\right\|_{\infty,M}(\|\nabla_H\eta\|_2^2+ \|\nabla_HT\|_2^2)\leq C\int_{-h}^h\|u\|_{\infty,M}^2dz\nonumber\\
\leq& C\int_{-h}^h(\|u\|_{2,M}^2+\|\Delta_Hu\|_{2,M}^2)dz\leq C\|(u,\Delta_Hu)\|_2^2\nonumber\\
\leq& C(1+\|\nabla_Hu\|_{H^1}^2).
\end{align*}
Thanks to the above and recalling (\ref{APR5}), (\ref{APR7}), and (\ref{3.63}), it follows from the Sobolev inequality that
\begin{align}
  \|\nabla_HS\|_2^2\leq&\int_\Omega(|v|^2|\nabla_H^2v|^2 +|\nabla_Hv|^4+|w|^2|\nabla_H\partial_zv|^2 \nonumber\\
  &+|\nabla_Hw|^2|\partial_zv|^2+f_0^2|\nabla_Hv|^2)dxdydz\nonumber\\
  \leq&C(\|v\|_\infty^2\|\nabla_H^2v\|_2^2+\|\nabla_Hv\|_4^4 +\|w\|_\infty^2\|\nabla_H\partial_zv\|_2^2 +1+\|\nabla_Hu\|_{H^1}^2)\nonumber\\
  \leq&C(\|v\|_{H^2}^4+\|\Delta_H\eta\|_2^2 +1+\|\nabla_Hu\|_{H^1}^2)\nonumber\\
  \leq&C(1+\|\Delta_H\eta\|_2^2+\|\nabla_Hu\|_{H^1}^2) \label{ns}
\end{align}
and
\begin{align}
\|R\|_2^2\leq & C(\|\nabla_H\cdot(vT)\|_2^2+\varepsilon^2\|\Delta_HT\|_2^2)\nonumber\\
  \leq&C(\|T\|_\infty^2\|\nabla_Hv\|_2^2+\|v\|_\infty^2\|\nabla_HT\|_2^2 +\varepsilon^2\|\Delta_HT\|_2^2 )\nonumber\\
  \leq&C(1+\|v\|_{H^2}^2 +\varepsilon^2\|\Delta_HT\|_2^2 )
  \leq C(1+\|\sqrt\varepsilon\Delta_HT\|_2^2) , \label{r}
\end{align}
where the constant $C$ depending only on $h, q, \mathcal T$, and $\mathcal Q_1$, but is independent of $\varepsilon$. Recalling the expression of $f(x,y,t)$ in (\ref{f}), one can easily check that
$$
f=\frac{1}{2h}\int_{-h}^h(\nabla_H\cdot S+R+wT)dz,
$$
and thus
\begin{align}
  \label{ef}\|f\|_2^2\leq& C(\|\nabla_HS\|_2^2+\|R\|_2^2+\|\nabla_Hv\|_2^2)
  \nonumber\\
  \leq& C(1+\|(\Delta_H\eta,\sqrt\varepsilon\Delta_HT)\|_2^2+\|\nabla_Hu\|_{H^1}^2),
\end{align}
where the constant $C$ depending only on $h, q, \mathcal T$, and $\mathcal Q_1$, but is independent of $\varepsilon$.

Thanks to (\ref{APR3}), (\ref{APR5}), (\ref{APR7}), and (\ref{3.63}), it follows from equation  (\ref{4.1u}) and the Sobolev and H\"older inequalities that
\begin{align}
  \|\partial_t u\|_2^2\leq&C(\|v\|_\infty^2\|\nabla_Hu\|_2^2+\|w\|_\infty^2\|\partial_zu\|_2^2 +\|\Delta_Hu\|_2^2 +\varepsilon^2\|\partial_z^2u\|_2^2\nonumber\\
  &+\|u\|_2^2+\|u\|_4^2\|\nabla_Hv\|_4^2+\|\nabla_HT\|_2^2)\nonumber\\
  \leq&C[\|v\|_{H^2}^4+(1+\|\Delta_H\eta\|_2^2)\|v\|_{H^2}^2+\|\Delta_Hu\|_2^2+\varepsilon^2\| \partial_z^2u\|_2^2+1]\nonumber\\
  \leq&C(1+\|\Delta_H\eta\|_2^2+\|\nabla_Hu\|_{H^1}^2+ \|\sqrt\varepsilon\partial_z^2u\|_2^2), \label{patu}
\end{align}
where the constant $C$ depending only on $h, q, \mathcal T$,
and $\mathcal Q_1$, but is independent of $\varepsilon$.
Using (\ref{ns})--(\ref{ef}), and recalling (\ref{APR5}) and (\ref{APR7}),
it follows from equations (\ref{etaeps})--(\ref{thetaeps}) and the Sobolev
and H\"older inequalities that
\begin{align}
  \|\partial_t\theta\|_2^2\leq & \|\Delta_H\theta\|_2^2+\varepsilon^2\|\partial_z^2\theta\|_2^2+\| \nabla_HS\|_2^2\nonumber\\
  \leq &C( \|(\Delta_H\theta,\sqrt\varepsilon\partial_z^2\theta)\|_2^2
  +1+\|\Delta_H\eta\|_2^2+\|\nabla_Hu\|_{H^1}^2) \label{patheta}
\end{align}
and
\begin{align}
  \|\partial_t\eta\|_2^2\leq&\|\Delta_H\eta\|_2^2+\varepsilon^2\|\partial_z^2\eta\|_2^2
  +\|\nabla_HS\|_2^2+(1-\varepsilon)^2\|\partial_zT\|_2^2+\|wT\|_2^2+\|R\|_2^2+\|f\|_2^2\nonumber\\
  \leq &C( \|(\Delta_H\eta,\sqrt\varepsilon\partial_z^2\eta,\sqrt\varepsilon\Delta_HT)
  \|_2^2
  +1+\|\nabla_Hu\|_{H^1}^2)\label{pateta}
\end{align}
for a positive constant $C$ depending only on $h, q, \mathcal T$, and $\mathcal Q_1$, but is independent of $\varepsilon$.

Thanks to (\ref{APR3}), (\ref{APR5}), (\ref{APR6}), (\ref{APR8}), (\ref{patt}), (\ref{patv}), and using the elliptic estimates, it follows from (\ref{patu})--(\ref{pateta}) that
\begin{align*}
\|\partial_tT\|_{L^2(0,\mathcal T; L^2)}^2& +\|\partial_tv\|_{L^2(0,\mathcal T; H^1)}^2=\|(\partial_tT,\partial_tv,\partial_t u,\partial_t\nabla_Hv)\|_{L^2(0,\mathcal T; L^2)}^2\nonumber\\
  \leq&\|(\partial_tT,\partial_tv,\partial_t u)\|_{L^2(0,\mathcal T; L^2)}^2+C(\|\nabla_H\cdot\partial_tv\|_2^2+\|\nabla_H^\perp\cdot\partial_tv\|_2^2)\nonumber\\
  \leq&\|(\partial_tT,\partial_tv,\partial_t u)\|_{L^2(0,\mathcal T; L^2)}+C\|(\partial_t\eta,\partial_tT,\partial_t\theta)\|_{L^2(0,\mathcal T; L^2)}^2\leq C
\end{align*}
for a positive constant $C$ depending only on $h, q, \mathcal T$, and $\mathcal Q_1$, but is independent of $\varepsilon$.
The first conclusion follows from the above inequality,
(\ref{APR1}), (\ref{APR5}), (\ref{APR7})--(\ref{APR9}), and (\ref{hps}).

We now prove the second conclusion.
By Proposition \ref{prop3.6}, one has
\begin{align}
&\sup_{0\leq t\leq\mathcal T}\|\nabla^2T\|_2^2(t)+\int_0^{\mathcal T}
(\|\partial_z\nabla^2 T\|_2^2+\varepsilon\|\nabla_H\nabla^2T\|_2^2) dt\leq C \label{estsec3}
\end{align}
for a positive constant $C$ depending only on $h, \mathcal T,$ and $\|(v_0,T_0)\|_{H^2}$, but is independent of $\varepsilon$.
Recalling the expressions of $\eta$ and $\theta$, it follows from the elliptic estimates that
\begin{equation}
  \|\nabla_Hv\|_{H^2}^2\leq C(\|\nabla_H\cdot v\|_{H^2}^2+\|\nabla_H^\perp\cdot v\|_{H^2}^2)
  \leq C(\|\theta\|_{H^2}^2+\|\eta\|_{H^2}^2+\|T\|_{H^2}^2) \label{JIA}
\end{equation}
for a positive constant $C$ depending only on $h$, but is independent of $\varepsilon$.
Recalling (\ref{3.63}) and the first conclusion, it follows from equation (\ref{eq3}), the estimate (\ref{estsec3}), and the H\"older and Sobolev inequalities that
\begin{align}
  \|\nabla\partial_tT\|_2^2\leq&\int_\Omega(|v|^2|\nabla_H^2T|^2+|\nabla_Hv|^2||\nabla_HT|^2+|w|^2| |\nabla_H\partial_zT|^2\nonumber\\
  &+|\nabla_Hw|^2|\partial_zT|^2+|\nabla\partial_z^2T|^2+\varepsilon^2|\nabla\Delta_HT|^2) dxdydz\nonumber\\
  \leq&C(\|v\|_\infty^2\|\nabla_H^2T\|_2^2+\|\nabla_Hv\|_4^2\|\nabla_HT\|_4^2+\|w\|_\infty^2\| \nabla_H\partial_zT\|_2^2\nonumber\\
  &+\|\nabla_Hw\|_4^2\|\partial_zT\|_4^2+\|\nabla\partial_z^2T\|_2^2+\varepsilon^2\|\nabla\Delta_H T\|_2^2)\nonumber\\
  \leq&C[\|v\|_{H^2}^2\|T\|_{H^2}^2+(1+\|\Delta_H\eta\|_2^2)\|T\|_{H^2}^2\nonumber\\
  &+\|\nabla_Hv\|_{H^2}^2\|T\|_{H^2}^2+\|\partial_zT\|_{H^2}^2+\varepsilon^2\|\nabla_H\Delta_HT \|_2^2]\nonumber\\
  \leq&C(1+\|\Delta_H\eta\|_2^2+\|\nabla_Hv\|_{H^2}^2+\|\partial_zT\|_{H^2}^2 +\varepsilon\|\nabla_H\Delta_HT\|_2^2)  \label{patnabt}
\end{align}
for a positive constant $C$ depending only on $h, \mathcal T,$ and $\|(v_0,T_0)\|_{H^2}$, but is independent of $\varepsilon$.
Combining this inequality with (\ref{estsec3})--(\ref{JIA}), as well as the first conclusion, yields the second conclusion.
\end{proof}
\section{Proof of Theorem \ref{thm1}}
\label{sec4}
In this section, we consider the system with only horizontal viscosities
and only vertical diffusivity, i.e., system (\ref{MAIN1})--(\ref{MAIN3}), subject to the boundary and initial conditions (\ref{BC1})--(\ref{IC}), and establish the global well-posedness
of strong solutions. In other words, we prove our main result, Theorem \ref{thm1}.

\begin{proof}[Proof of Theorem \ref{thm1}]
\textbf{Global existence. } As in Proposition \ref{prop3.7}, we set
$$
\mathcal Q_1=\|v_0\|_{H^2}^2+\|T_0\|_{H^1}^2+\|\nabla_HT_0\|_q^q+\|T_0\|_\infty^2,
$$
with $q\in(2,\infty)$. Thanks to the regularities and spatial symmetries
of $v_0$ and $T_0$, one can choose periodic functions $v_{0\varepsilon}$
and $T_{0\varepsilon}$, which are even and odd in $z$, respectively, such
that $v_{0\varepsilon}\in H^2(\Omega)$, $T_{0\varepsilon}\in H^2(\Omega)$,
$$
\int_{-h}^h\nabla_H\cdot v_{0\varepsilon}(x,y,z)dz=0, \quad\|T_{0\varepsilon}\|_\infty\leq\|T_0\|_\infty,
$$
and
$$
v_{0\varepsilon}\rightarrow v_0\mbox{ in }H^2(\Omega),\quad T_{0\varepsilon}\rightarrow T_0\mbox{ in }H^1(\Omega),\quad\nabla_HT_{0\varepsilon}\rightarrow\nabla_HT_0\mbox{ in }L^q(\Omega).
$$
Note that such $v_{0\varepsilon}$ and $T_{0\varepsilon}$ can be chosen as the standard mollification of $v_0$ and $T_0$, respectively. Set
$$
\mathcal Q_{1\varepsilon}=\|v_{0\varepsilon}\|_{H^2}^2+\|T_{0\varepsilon}\|_{H^1}^2 +\|\nabla_HT_{0\varepsilon}\|_q^q+\|T_{0\varepsilon}\|_\infty^2
$$
then $\mathcal Q_{1\varepsilon}\leq 2\mathcal Q_1$, for sufficiently small $\varepsilon$. By Proposition \ref{prop3.1}, there is a unique global strong
solution $(v_\varepsilon, T_\varepsilon)$ to system (\ref{eq1})--(\ref{eq3}),
subject to the boundary conditions (\ref{BC1})--(\ref{BC2}) and the
initial condition
$$
(v_\varepsilon, T_\varepsilon)|_{t=0}=(v_{0\varepsilon}, T_{0\varepsilon}).
$$

By Proposition \ref{prop3.7}, the following uniform estimate
\begin{align}
\sup_{0\leq t\leq \mathcal T}&(\|v_\varepsilon\|_{H^2}^2+\|T_\varepsilon\|_{H^1}^2+\|\nabla_H T_\varepsilon\|_q^q+\|\nabla_Hp_{\varepsilon}\|_2^2+\|T_\varepsilon\|_\infty^2)\nonumber\\
  &+\int_0^{\mathcal T}(\|\nabla_Hu_\varepsilon\|_{H^1}^2+\|\theta_\varepsilon\|_{H^2}^2+ \|\eta_\varepsilon\|_{H^2}^2+\|\partial_zT_\varepsilon\|_{H^1}^2 +\|\partial_t\eta_\varepsilon\|_2^2\nonumber\\
  &+\|\partial_t\theta_\varepsilon\|_2^2+\|\partial_tv_\varepsilon\|_{H^1}^2 +\|\partial_tT_\varepsilon\|_{H^1}^2)\leq C,\label{apri1}
\end{align}
for a positive constant $C$ depending only on $h, \mathcal T,$ and
$\mathcal Q_1$ and, thus, is independent of $\varepsilon$,
here $u_\varepsilon, \eta_\varepsilon,
\theta_\varepsilon$ are the associated functions defined by
(\ref{etatheta}) and $p_\varepsilon=p_\varepsilon(x,y,t)$ is
the associated pressure function determined by (\ref{ps}).

On account of the above a priori estimates, by the Aubin-Lions lemma, i.e. Lemma \ref{AL}, there is a subsequence, still denoted by $(v_\varepsilon, T_\varepsilon)$, and $(v,T)$, such that
\begin{eqnarray*}
  &&v_\varepsilon\rightarrow v\mbox{ in }C([0,\mathcal T];H^1(\Omega)),\quad T_\varepsilon\rightarrow T \mbox{ in }C([0,\mathcal T];L^2(\Omega)),\\
  &&v_\varepsilon\overset{*}{\rightharpoonup}v\mbox{ in }L^\infty(0,\mathcal T; H^2(\Omega)),\quad\partial_t v_\varepsilon\rightharpoonup\partial_tv\mbox{ in }L^2(0,\mathcal T; H^1(\Omega)),\\
  &&T_\varepsilon\overset{*}{\rightharpoonup}T\mbox{ in }L^\infty(0,\mathcal T; H^1(\Omega)),\quad\partial_tT_\varepsilon\rightharpoonup\partial_tT\mbox{ in }L^2(0,\mathcal T;L^2(\Omega)),\\
  &&\nabla_Hu_\varepsilon\rightharpoonup\nabla_Hu\mbox{ in }L^2(0,\mathcal T;H^1(\Omega)),\quad \partial_zT_\varepsilon\rightharpoonup\partial_zT\mbox{ in }L^2(0,\mathcal T;H^1(\Omega)),\\
  &&\nabla_HT_\varepsilon\overset{*}{\rightharpoonup}\nabla_HT\mbox{ in }L^\infty(0,\mathcal T;L^q(\Omega)),\quad p_\varepsilon\rightharpoonup p_s\mbox{ in }L^2(0,\mathcal T;H^1(M)),\\
  &&\theta_\varepsilon\rightharpoonup\theta\mbox{ in }L^2(0,\mathcal T;H^2(\Omega)),\quad\partial_t\theta_\varepsilon\rightharpoonup
  \partial_t\theta \mbox{ in }L^2(0,\mathcal T;L^2(\Omega)),\\
  &&\eta_\varepsilon\rightharpoonup\eta\mbox{ in }L^2(0,\mathcal T;H^2(\Omega)),\quad\partial_t\eta_\varepsilon\rightharpoonup
  \partial_t\eta \mbox{ in }L^2(0,\mathcal T;L^2(\Omega)),
\end{eqnarray*}
where $\rightharpoonup$ and $\overset{*}{\rightharpoonup}$ denote the weak and weak-* convergences, respectively. Due to these convergences, one can take the limit $\varepsilon\rightarrow0$ in systems (\ref{eq1})--(\ref{eq3}) and (\ref{4.1u})--(\ref{thetaeps}),
to show that $(v,T)$ satisfies system (\ref{MAIN1})--(\ref{MAIN3}), and $(u,\eta,\theta)$, defined by (\ref{etatheta}), satisfies
\begin{align}
\partial_tu+(v\cdot\nabla_H)u&+w\partial_zu-\Delta_Hu+f_0k\times u\nonumber\\
&+(u\cdot\nabla_H)v-(\nabla_H\cdot v)u-\nabla_HT=0,\label{uu}\\
\partial_t\eta-\Delta_H\eta=&-\nabla_H\cdot[(v\cdot\nabla_H)v  +w\partial_zv+f_0k\times v]\nonumber\\
  &+\partial_zT-wT-\int_{-h}^z\nabla_H\cdot(vT)d\xi+f(x,y,t),\label{eta}\\
\partial_t\theta-\Delta_H\theta=&-\nabla_H^\perp\cdot[(v\cdot\nabla_H)v  +w\partial_zv+f_0\overrightarrow{k}\times v],\label{theta}
\end{align}
in the sense of distribution, where the function $f=f(x,y,t)$ is now given by
\begin{align}
  f=&\frac{1}{2h}\int_{-h}^h\left(\int_{-h}^z\nabla_H\cdot(vT)d\xi+wT+\nabla_H\cdot \big(\nabla_H\cdot(v\otimes v)+f_0\overrightarrow{k}\times v\big)\right)\label{ff}
   dz.
\end{align}
Moreover, by the weakly lower semi-continuity of the norms and recalling
that $(v_\varepsilon, T_\varepsilon)$ satisfies the a priori estimate (\ref{apri1}), we can see that $(v,T)$ still satisfies the same a priori
estimate as (\ref{apri1}). This implies the regularity properties stated
in Definition \ref{def1.1} and, as a result, systems
(\ref{MAIN1})--(\ref{MAIN3}) and (\ref{uu})--(\ref{theta}) are satisfied
a.e.\,in $\Omega\times(0,\mathcal T)$. Furthermore, recalling the first
line of the previous convergences, one can easily show that $(v,T)$
satisfies the initial condition (\ref{IC}) and, therefore, $(v,T)$ is a
strong solution to system (\ref{MAIN1})--(\ref{MAIN3}), subject to the
boundary and initial conditions (\ref{BC1})--(\ref{IC}).

Now, if we assume, in addition, that $T_0\in H^2(\Omega)$, then the
mollification $T_{0\varepsilon}$ converges strongly to $T_0$ in
$H^2(\Omega)$. As a result, the quantity
$\mathcal Q_{2\varepsilon}:=\|v_{0\varepsilon}\|_{H^2}^2+\|T_{0\varepsilon}\|_{H^2}^2$
is bounded by $2\mathcal Q_2=2(\|v_{0}\|_{H^2}^2+\|T_{0}\|_{H^2}^2)$
for small $\varepsilon$. By Proposition \ref{prop3.7}, for small
$\varepsilon$,
we have the following uniform estimate
\begin{align*}
  \sup_{0\leq t\leq\mathcal T}\|T_\varepsilon\|_{H^2}^2(t)+\int_0^{\mathcal T}(\|\nabla_Hv_\varepsilon\|_{H^2}^2+\|\partial_zT_\varepsilon\|_{H^2}^2 +\|\partial_tT_\varepsilon\|_{H^1}^2) dt\leq C
\end{align*}
for a positive constant $C$ depending only on $h, \mathcal T,$ and
$\mathcal Q_2$, but is independent of $\varepsilon$. This a priori estimate,
by the weakly lower semi-continuity of the norms and the Aubin-Lions lemma,
implies the additional regularities as stated in the theorem. This completes
the proof of the existence part of the theorem.

\textbf{Continuous dependence on the initial data. }Let $(v_1,T_1)$ and
$(v_2,T_2)$ be two strong solutions to the same system with initial data
$(v_{01}, T_{01})$ and $(v_{02}, T_{02})$, respectively. Denote
$v=v_1-v_2$,
$w=w_1-w_2$, and $T=T_1-T_2$. Then, $(v,T)$ satisfies
\begin{align}
  &\partial_t v+(v_1\cdot\nabla_H)v+w_1\partial_zv-\Delta_Hv+f_0\overrightarrow{k}\times v+\nabla_Hp_s(x,y,t)\nonumber\\
  &\qquad =\int_{-h}^z\nabla_H T(x,y,\xi,t)d\xi-(v\cdot\nabla_H)v_2-w\partial_zv_2,\label{dv}\\
  &\partial_tT+v_1\cdot\nabla_HT+w_1\partial_zT-\partial_z^2T=-v\cdot\nabla_HT_2-w\partial_zT_2, \label{dt}
\end{align}
with initial data $(v_0,T_0)=(v_{01}-v_{02}, T_{01}-T_{02})$. Recalling the
regularities of strong solutions stated in Definition \ref{def1.1}, it is
clear that all terms the above equations are well defined pointwisely.

Note that
$$
|\nabla_Hv_2(x,y,z,t)|\leq\frac{1}{2h}\int_{-h}^h|\nabla_Hv_2(x,y,z,t)|dz+\int_{-h}^h|\nabla_H \partial_zv_2(x,y,z,t)|dz.
$$
Recalling the regularities of $(v,T)$, multiplying equation (\ref{dv}) by $v$ and integrating the resultant over $\Omega$, it follows from integration by parts, the H\"older inequality, and Lemma \ref{lad} that
\begin{align*}
  &\frac{1}{2}\frac{d}{dt}\|v\|_2^2+\|\nabla_Hv\|_2^2\\
  =&\int_\Omega\left[\nabla_H\left(\int_{-h}^zTd\xi\right)-(v\cdot\nabla_H)v_2-w\partial_zv_2\right] \cdot v dxdydz\\
  =&-\int_\Omega\left[\left(\int_{-h}^zTd\xi\right)(\nabla_H\cdot v)+((v\cdot\nabla_H)v_2+w\partial_zv_2)\cdot v\right]dxdydz\\
  \leq&C\|T\|_2\|\nabla_Hv\|_2+C\int_\Omega\left[|v|^2\left(\int_{-h}^h (|\nabla_H\partial_zv_2|+|\nabla_Hv_2|) dz\right)\right.\\
  &\left.+\left(\int_{-h}^h|\nabla_H\cdot v|dz\right)|\partial_zv_2||v|\right]dxdydz\\
  \leq&C\|T\|_2\|\nabla_Hv\|_2+C\int_M\left(\int_{-h}^h|v|^2dz\right)
  \left(\int_{-h}^h(|\nabla_Hv_2|+|\nabla_H\partial_z v_2|)dz\right)dxdy\\
  &+C\int_M\left(\int_{-h}^h|\nabla_Hv|dz\right)\left(
  \int_{-h}^h|\partial_zv_2||v|dz\right)dxdy\\
  \leq&C\|T\|_2\|\nabla_Hv\|_2+C\|v\|_2(\|v\|_2+\|\nabla_Hv\|_2)(\|\nabla_Hv_2\|_2+\|\nabla_H\partial_z v_2\|_2)\\
  &+C\|\nabla_Hv\|_2\|\partial_zv_2\|_2^{\frac{1}{2}}(\|\partial_zv_2\|_2^{\frac{1}{2}}+\|\nabla_H \partial_zv_2\|_2^{\frac{1}{2}})\|v\|_2^{\frac{1}{2}}(\|v\|_2^{\frac{1}{2}}+\|\nabla_Hv\|_2^{ \frac{1}{2}})\\
  \leq&\frac{1}{2}\|\nabla_Hv\|_2^2+C[\|T\|_2^2+(1+\|\nabla_Hv_2\|_2^2+\|\nabla_H\partial_zv_2\|_2^2) \|v\|_2^2\\
  &+\|\partial_zv_2\|_2^2(\|\partial_zv_2\|_2^2+\|\nabla_H\partial_zv_2\|_2^2)\|v\|_2^2]\\
  \leq&\frac{1}{2}\|\nabla_Hv\|_2^2+C(1+\|\nabla v_2\|_2^2)^2(1+\|\nabla_H\partial_zv_2\|_2^2)(\|T\|_2^2+\|v\|_2^2)
\end{align*}
and, thus,
\begin{equation}
  \frac{d}{dt}\|v\|_2^2+\|\nabla_Hv\|_2^2\leq C(1+\|\nabla v_2\|_2^2)^2(1+\|\nabla_H\partial_zv_2\|_2^2)(\|T\|_2^2+\|v\|_2^2)
  \label{uniq1}
\end{equation}
for $t\in(0,\mathcal T)$.

Recalling the regularities of $(v,T)$, multiplying equation (\ref{MAIN3}) by $T$, integrating the resulting
equation over $\Omega$, and noticing that
$\|T_2\|_\infty\leq\|T_{02}\|_\infty$, it follows from integration by
parts, (\ref{ineqlad}), and the Ladyzhenskaya inequality that
\begin{equation}\label{ADD0}
\frac{1}{2}\frac{d}{dt}\|T\|_2^2 +\|\partial_zT\|_2^2
  =-\int_\Omega(v\cdot\nabla_H T_2+w\partial_zT_2)Tdxdydz
\end{equation}
Noticing that $\|T_2\|_\infty\leq\|T_{02}\|_\infty$, it follows from integration by parts and the Young inequalities that
\begin{align}
 &\int_\Omega w\partial_zT_2Tdxdydz
=-\int_\Omega T_2(\partial_zwT+w\partial_zT) dxdydz\nonumber\\
=&\int_\Omega((\nabla_H\cdot v)T_2T-w\partial_zTT_2)dxdydz
\leq \|T_2\|_\infty(\|\nabla_Hv\|_2\|T\|_2+ \|\partial_zT\|_2\|w\|_2) \nonumber\\
\leq& C\|T_2\|_\infty\|\nabla_Hv\|_2(\|T\|_2+\|\partial_zT\|_2)
\leq \frac14\|\partial_zT\|_2^2+C(\|\nabla_Hv\|_2^2+\|T\|_2^2). \label{ADD1}
\end{align}
Note that $T|_{z=-h}=0$, we have $|T|\leq\int_{-h}^h|\partial_zT|dz.$ With
the aid of this, by the H\"older, Minkowski, Gagaliardo-Nirenberg, and Young inequalities, we deduce
\begin{align}
-\int_\Omega v\cdot&\nabla_HT_2Tdxdydz\leq\int_M\left(\int_{-h}^h |v||\nabla_Hv_2| dz\right)\left(\int_{-h}^h|\partial_zT|dz\right)dxdy\nonumber\\
\leq&\int_M\left(\int_{-h}^h|v|^2dz\right)^{\frac12}\left(\int_{-h}^h |\nabla_HT_2|^2dz\right)^{\frac12}\left( \int_{-h}^h|\partial_zT|dz\right)dxdy\nonumber \\
\leq&\left\|\left(\int_{-h}^h|v|^2dz\right)^{\frac12} \right\|_{\frac{2q}{q-2},M}
\left\|\left(\int_{-h}^h|\nabla_HT_2|^2dz\right)^{\frac12}\right\|_{q,M} \left\|\int_{-h}^h|\partial_zT|dz\right\|_{2,M}\nonumber\\
\leq&\left(\int_{-h}^h\|v\|_{\frac{2q}{q-2},M}^2dz\right)^{\frac12} \left(\int_{-h}^h\|\nabla_HT_2\|_{q,M}^2dz\right)^{\frac12} \left(\int_{-h}^h\|\partial_zT\|_{2,M}dz\right) \nonumber\\
\leq& C\left[\int_{-h}^h\left(\|v\|_{2,M}^2+\|v\|_{2,M}^{\frac{2}{q}(q-2)} \|\nabla_Hv\|_{2,M}^{\frac4q}\right)dz\right]^{\frac12}\|\nabla_HT_2\|_q \|\partial_zT\|_2\nonumber\\
\leq&C\left(\|v\|_2+\|v\|_2^{\frac{q-2}{q}}\|\nabla_Hv\|_2^{\frac2q} \right)\|\nabla_HT_2\|_q \|\partial_zT\|_2\nonumber\\
\leq&\frac14(\|\partial_zT\|_2^2+\|\nabla_Hv\|_2^2)+C\left(\|\nabla_HT_2\|_q^2+ \|\nabla_HT_2\|_q^{\frac{2q}{q-2}}\right)\|v\|_2^2.\label{ADD2}
\end{align}
Substituting (\ref{ADD1}) and (\ref{ADD2}) into (\ref{ADD0}) yields
\begin{equation}
  \frac{d}{dt}\|T\|_2^2+\|\partial_zT\|_2^2\leq C\|\nabla_Hv\|_2^2+C\left(1+\|\nabla_HT_2\|_q^2+ \|\nabla_HT_2\|_q^{\frac{2q}{q-2}}\right)(\|v\|_2^2+\|T\|_2^2), \label{uniq2}
\end{equation}
for $t\in(0,\mathcal T)$.

Multiplying (\ref{uniq1}) by a sufficiently large positive constant $A$, and summing the resulting inequality with (\ref{uniq2}) up yiedls
\begin{align*}
  &\frac{d}{dt}(A\|v\|_2^2+\|T\|_2^2)+\frac{1}{2}(A\|\nabla_Hv\|_2^2+\|T\|_2^2)\\
  \leq&C\left[(1+\|\nabla v_2\|_2^2)^2(1+\|\nabla_H\partial_zv_2\|_2^2)+ \left(\|\nabla_HT_2\|_q^2+ \|\nabla_HT_2\|_q^{\frac{2q}{q-2}}\right)\right](\|T\|_2^2+\|v\|_2^2),
\end{align*}
from which, by the Gronwall inequality, one obtains
\begin{align*}
  &\sup_{0\leq s\leq t}(\|v\|_2^2(s)+\|T\|_2^2(s))+\int_0^t(\|\nabla_Hv\|_2^2+\|\partial_zT\|_2^2)ds\\
  \leq&Ce^{\int_0^t\left[(1+\|\nabla v_2\|_2^2)^2(1+\|\nabla_H\partial_zv_2\|_2^2)+ \left(\|\nabla_HT_2\|_q^2+ \|\nabla_HT_2\|_q^{\frac{2q}{q-2}}\right)\right]ds} (\|v_0\|_2^2+\|T_0\|_2^2)
\end{align*}
for any $t\in(0,\mathcal T)$. This proves the continuous dependence of the strong solutions on the initial data, in particular the uniqueness. This completes the proof of the theorem.
\end{proof}

\section{Appendix A: Equations for $\eta$ and $\theta$}
\label{appendixa}
In this appendix, we present the details of the derivation of the equations for $\eta$ and $\theta$, where $\eta$ and $\theta$ are the same functions as defined by (\ref{etatheta}), i.e.,
\begin{eqnarray*}
  &&\eta=\nabla_H\cdot v+\int_{-h}^zT(x,y,\xi,t)d\xi-\frac{1}{2h} \int_{-h}^h\left(\int_{-h}^zT(x,y,\xi,t)d\xi\right) dz,\\
  &&\theta=\nabla_H^\perp\cdot v,\qquad \nabla_H^\perp=(-\partial_y, \partial_x),
\end{eqnarray*}
with $(v,T)$ being a strong solution to system (\ref{eq1})--(\ref{eq3}), subject to the boundary and initial conditions (\ref{BC1})--(\ref{IC}).

Applying the operator $\nabla_H^\perp\cdot$ to equation (\ref{eq1}), and noticing that $\nabla_H^\perp\cdot\nabla_Hp_s=0$, one obtains
\begin{equation}
  \label{apxa1}
  \partial_t\theta-\Delta_H\theta-\varepsilon\partial_z^2\theta=-\nabla_H^\perp\cdot[(v\cdot\nabla_H)v+w\partial_zv+f_0k\times v],
\end{equation}
obtaining the equation for $\theta$.

Applying the operator $\nabla_H\cdot$, i.e., $\text{div}_H$, to equation (\ref{eq1}), one gets
\begin{align}
\partial_t(\nabla_H\cdot v)-\Delta_H&\bigg(\nabla_H\cdot v+\int_{-h}^zT(x,y, \xi, t)d\xi-p_s(x,y,t)\bigg)-\varepsilon\partial_z^2(\nabla_H\cdot v)\nonumber\\
 =&-\nabla_H\cdot [(v\cdot\nabla_H)v+w\partial_zv+f_0k\times v].\label{apxa2}
\end{align}
Integrating the above equation with respect to $z$ over the interval $(-h, h)$, and noticing
\begin{equation*}
  \int_{-h}^h[(v\cdot\nabla_H)v+w\partial_zv]dz=\int_{-h}^h[(v\cdot\nabla_H)v+(\nabla_H\cdot v)v]dz=\int_{-h}^h\nabla_H\cdot(v\otimes v)dz,
\end{equation*}
and (recalling $\int_{-h}^h\nabla_H\cdot v dz=0$)
$$
\int_{-h}^h[\partial_t(\nabla_H\cdot v)-\Delta_H(\nabla_H\cdot v)-\partial_z^2(\nabla_H\cdot v)]dz=0,
$$
we obtain
\begin{equation}
-\Delta_Hp_s=\frac{1}{2h}\nabla_H\cdot\int_{-h}^h\left(\nabla_H\cdot(v\otimes v)+f_0k\times v-\int_{-h}^z\nabla_HTd\xi\right)dz. \label{apxa3}
\end{equation}
Substituting (\ref{apxa3}) into (\ref{apxa2}), one has
\begin{align*}
  &\partial_t(\nabla_H\cdot v)-\Delta_H\bigg(\nabla_H\cdot v+\int_{-h}^zTd\xi-\frac{1}{2h}\int_{-h}^h\int_{-h}^zTd\xi dz\bigg)-\varepsilon\partial_z^2(\nabla_H\cdot v)\nonumber\\
  =&-\nabla_H\cdot [(v\cdot\nabla_H)v+w\partial_zv+f_0k\times v]+\frac{1}{2h}\int_{-h}^h\nabla_H\cdot(\nabla_H\cdot(v\otimes v)+f_0k\times v)dz,
\end{align*}
from which, recalling the definition of $\eta$, one arrives at
\begin{eqnarray}
  &\partial_t\eta-\Delta_H\eta-\varepsilon\partial_z^2\eta
  = -\nabla_H\cdot[(v\cdot\nabla_H)v+w\partial_zv+f_0k\times v]-\varepsilon\partial_zT \nonumber\\
  &+\frac{1}{2h}\int_{-h}^h\nabla_H\cdot(\nabla_H\cdot(v\otimes v)+f_0k\times v)+\partial_t\left(\int_{-h}^zTd\xi-\frac{1}{2h}\int_{-h}^h\int_{-h}^zTd\xi dz\right).\label{apxa4}
\end{eqnarray}
We compute the last term on the right-hand side of (\ref{apxa4}) as follows. On account of equation (\ref{eq3}), we have
\begin{align*}
  \int_{-h}^z\partial_t Td\xi=&-\int_{-h}^z(\nabla_H\cdot (v T)+\partial_z(wT)-\varepsilon\Delta_HT-\partial_z^2T)d\xi\\
=&-\int_{-h}^z(\nabla_H\cdot(vT)-\varepsilon\Delta_HT)d\xi-wT+\partial_zT+(wT-\partial_zT)|_{z=-h},
\end{align*}
and thus
\begin{align*}
  &\int_{-h}^z\partial_tTd\xi-\frac{1}{2h}\int_{-h}^h\int_{-h}^z\partial_t Td\xi dz\\
  =&-\int_{-h}^z(\nabla_H\cdot (vT)-\varepsilon\Delta_HT)d\xi-wT+\partial_zT\\
  &+\frac{1}{2h}\int_{-h}^h\left(\int_{-h}^z(\nabla_H\cdot(vT) -\varepsilon\Delta_HT)d\xi\right) dz+\frac{1}{2h}\int_{-h}^hwTdz.
\end{align*}
Substituting the above equality into (\ref{apxa4}) yields
\begin{align}
  \partial_t\eta-\Delta_H\eta-\varepsilon\partial_z^2\eta=&-\nabla_H\cdot[(v\cdot\nabla_H)v  +w\partial_zv+f_0k\times v]+(1-\varepsilon)\partial_zT-wT\nonumber\\
  &-\int_{-h}^z(\nabla_H\cdot(vT)-\varepsilon\Delta_HT)d\xi+f(x,y,t),\label{apxa5}
\end{align}
with function $f=f(x,y,t)$ given by
\begin{align*}
  f=&\frac{1}{2h}\int_{-h}^h\left(\int_{-h}^z(\nabla_H\cdot(vT)-\varepsilon\Delta_HT)d\xi+wT\right)
   dz\\
   &+\frac{1}{2h}\int_{-h}^h \nabla_H\cdot \big(\nabla_H\cdot(v\otimes v)+f_0k\times v\big)dz.
\end{align*}

\section*{Acknowledgments}
{J.L.\ and E.S.T.\ would like to thank the ICERM, Brown University, for the warm and kind hospitality where this work was completed. Part of this work was done when J.L.\ was a postdoctoral fellow at the Weizmann Institute of Science. The work of J.L.\ was supported in part by the Direct Grant for Research 2016/2017 (Project Code: 4053216) from The Chinese University of Hong Kong. The work of E.S.T.\ was supported in part by the ONR grant N00014-15-1-2333 and the NSF grants DMS-1109640 and DMS-1109645.}
\par


\begin{thebibliography}{50}
\bibitem{AZGU}
Az\'erad,~P.; Guill\'en,~F.: \emph{Mathematical justification of the hydrostatic approximation in the primitive equations of geophysical fluid dynamics}, SIAM J. Math. Anal., \bf33 \rm(2001), 847--859.

\bibitem{BLNNT}
Bardos,~C.; Lopes Filho,~M.~C.; Niu,~Dongjuan; Nussenzveig Lopes,~H.~J.; Titi,~E.~S.: \emph{Stability of two-dimensional viscous incompressible flows under three-dimensional perturbations and inviscid symmetry breaking}, SIAM J. Math. Anal., \bf45~\rm(2013), 1871--1885.

\bibitem{BKM}
Beale,~J.~T., Kato,~T., Majda,~A.: \emph{Remarks on the breakdown of smooth solutions for the 3-D Euler equations}, Commun. Math. Phys., \bf94~\rm(1984), 61--66.

\bibitem{BGMR03}
Bresch,~D., Guill\'en-Gonz\'alez,~F., Masmoudi,~N., Rodr\'iguez-Bellido,~M.~A.: \emph{On the
uniqueness of weak solutions of the two-dimensional primitive equations}, Differential Integral Equations, \bf16~\rm(2003), 77--94.

\bibitem{BW}
Brezis,~H., Wainger,~S.: \emph{A note on limiting cases of Sobolev embeddings},
Comm. Partial Differential Equations, \bf5~\rm(1980), 773--789.

\bibitem{BG}
Brezis,~H., Gallouet,~T.: \emph{Nonlinear Schr\"odinger evolution equations}, Nonlinear
Anal., \bf4~\rm(1980), 677--681.


\bibitem{BKL04}
Bresch,~D., Kazhikhov, A., Lemoine,~J.: \emph{On the two-dimensional hydrostatic Navier-Stokes equations}, SIAM J.~Math. Anal., \bf36~\rm(2004), 796--814.

\bibitem{CINT}
Cao,~C., Ibrahim,~S., Nakanishi,~K., Titi,~E.~S.: \emph{Finite-time blowup for the inviscid primitive equations of oceanic and atmospheric dynamics}, Comm. Math. Phys.,  \bf337~\rm(2015), 473--482.

\bibitem{CAOFARHATTITI}
Cao,~C., Farhat, A., Titi, E.~S.: \emph{Global well-posedness of an inviscid three-dimensional pseudo-Hasegawa-Mima model,} Comm. Math. Phys., \bf319~\rm(2013),  195--229.

\bibitem{CAOLITITI1}
Cao,~C., Li, J., Titi, E.~S.: \emph{Local and global well-posedness of strong solutions to the 3D primitive equations with vertical eddy diffusivity,} Arch. Rational Mech. Anal., \bf214~\rm(2014), 35--76.

\bibitem{CAOLITITI2}
Cao,~C., Li, J., Titi, E.~S.: \emph{Global well-posedness of strong solutions to the 3D primitive equations with horizontal eddy diffusivity}, J. Differential Equations, \bf257~\rm(2014), 4108--4132.

\bibitem{CAOLITITI3}
Cao,~C., Li, J., Titi, E.~S.: \emph{Global well-posedness of the 3D primitive equations with only horizontal viscosity and diffusivity}, Comm. Pure Appl. Math., \bf69~\rm(2016), 1492--1531.

\bibitem{CAOLITITIH1}
Cao,~C., Li, J., Titi, E.~S.: \emph{Strong solutions to the 3D primitive equations with only horizontal dissipation: Near $H^1$ initial data}, J. Funct. Anal. (2017), http://dx.doi.org/10.1016/j.jfa.2017.01.018.


\bibitem{CAOWU}
Cao,~C., Wu,~Jia.: \emph{Global regularity for the two-dimensional anisotropic Boussinesq equations with vertical dissipation}, Arch. Ration. Mech. Anal., \bf208~\rm(2013), 985--1004.

\bibitem{CAOTITI1}
Cao,~C., Titi,~E.~S.: \emph{Global well-posedness and finite-dimensional global attractor for a 3-D planetary geostrophic viscous model}, Comm. Pure Appl. Math., \bf56~\rm(2003), 198--233.

\bibitem{CAOTITI2}
Cao,~C., Titi,~E.~S.: \emph{Global well-posedness of the three-dimensional viscous primitive equations of large scale ocean and atmosphere dynamics}, Ann. of Math., \bf166~\rm(2007), 245--267.

\bibitem{CAOTITI3}
Cao,~C., Titi,~E.~S.: \emph{Global well-posedness of the 3D primitive equations with partial vertical turbulence mixing heat diffusivity}, Comm. Math. Phys., \bf310~\rm(2012), 537--568.

\bibitem{CONFOINSBOOK}
Constantin,~P.; Foias,~C.: \emph{Navier-Stokes equations}, Chicago Lectures in Mathematics, University of Chicago Press, Chicago, IL, 1988.


\bibitem{DANCHINPAICU}
Danchin,~R., Paicu,~M.: \emph{Global existence results for the anisotropic Boussinesq system
in dimension two,} Math. Models Methods Appl. Sci., \bf21~\rm(2011), 421--457.

\bibitem{GMR01}
Guill\'en-Gonz\'alez,~F., Masmoudi,~N., Rodr\'iguez-Bellido,~M.~A.: \emph{Anisotropic estimates and strong solutions of the primitive equations,} Differ. Integral Equ., \bf14~\rm(2001), 1381--1408.

%

\bibitem{HIEHUSKAS}
Hieber,~M., Hussein,~A., Kashiwabara,~T.: \emph{Global strong $L^p$ well-posedness of the 3D primitive equations with heat and salinity diffusion.}, J. Differential Equations, \bf261~\rm(2016),
6950--6981.

\bibitem{HIEKAS}
Hieber,~M., Kashiwabara,~T: \emph{Global strong well-posedness of the three dimensional primitive equations in $L^p$-spaces.}, Arch. Ration. Mech. Anal., \bf221~\rm(2016), 1077--1115.

\bibitem{KOB06}
Kobelkov,~G.~M.: \emph{Existence of a solution in the large for the 3D large-scale ocean dynamics equaitons,} C. R. Math. Acad. Sci. Paris, \bf343~\rm(2006), 283--286.

\bibitem{KPRZ}
Kukavica,~I.; Pei,~Y.; Rusin,~W.; Ziane,~M.: \emph{Primitive equations with continuous initial data}, Nonlinearity, \bf27~\rm(2014), 1135--1155.


\bibitem{KZ07B}
Kukavica,~I., Ziane,~M.: \emph{On the regularity of the primitive equations of the ocean,} Nonlinearity, \bf20~\rm(2007), 2739--2753.

\bibitem{LADYZHENSKAYA}
Ladyzhenskaya,~O.~A.: \emph{The mathematical theory of viscous incompressible flow}, Second English edition, revised and enlarged. Translated from the Russian by Richard A. Silverman and John Chu. Mathematics and its Applications, Vol. 2 Gordon and Breach, Science Publishers, New York-London-Paris 1969.

\bibitem{LEWAN}
Lewandowski~R.: \emph{Analyse Math\'ematique et Oc\'eanographie}, Masson,
Paris, 1997.

\bibitem{LITITIBOU}
Li,~J.; Titi,~E.~S.: \emph{Global well-posedness of the 2D Boussinesq equations with vertical dissipation}, Arch. Ration. Mech. Anal., \bf220~\rm(2016), 983--1001.

\bibitem{LITITITROPMOIS}
Li,~J.; Titi,~E.~S.: \emph{A tropical atmosphere model with moisture: global well-posedness and relaxation limit}, Nonlinearity, \bf29~\rm(2016),  2674--2714.

\bibitem{LITITIUNIQ}
Li,~J.; Titi,~E.~S.: \emph{Existence and uniqueness of weak solutions to viscous primitive equations for a certain class of discontinuous initial data.}, SIAM J. Math. Anal., \bf49~\rm(2017), 1--28.

\bibitem{LITITIHYDRO}
Li,~J.; Titi,~E.~S.: \emph{Small aspect ratio limit from Navier-Stokes equations to primitive equations: mathematical justification of hydrostatic approximation}, preprint.

\bibitem{LTW92A}
Lions,~J.~L., Temam,~R., Wang,~S.: \emph{New formulations of the primitive equations of the atmosphere and appliations,} Nonlinearity, \bf5~\rm(1992), 237--288.

\bibitem{LTW92B}
Lions,~J.~L., Temam,~R., Wang,~S.: \emph{On the equations of the large-scale ocean}, Nonlinearity, \bf5~\rm(1992), 1007--1053.

\bibitem{LTW95}
Lions,~J.~L., Temam,~R., Wang,~S.: \emph{Mathematical study of the coupled models of atmosphere and ocean (CAO III)}, J. Math. Pures Appl., \bf74~\rm(1995), 105--163.

\bibitem{MAJDA}
Majda,~A.: \emph{Introduction to PDEs and Waves for the Atmosphere and Ocean,}
New York University, Courant Institute of Mathematical Sciences, New York; American Mathematical Society, Providence, RI, 2003.

\bibitem{PED}
Pedlosky,~J.: \emph{Geophysical Fluid Dynamics, 2nd edition}, Springer, New York, 1987.

\bibitem{PTZ09}
Petcu,~M., Temam,~R., Ziane,~M.: \emph{Some mathematical problems in geophysical fluid dynamics,}
Elsevier: Handbook of Numarical Analysis, \bf14~\rm(2009), 577--750.

\bibitem{Simon}
Simon,~J.: \emph{Compact sets in the space $L^p(0, T; B)$}, Ann. Mat. Pure Appl., \bf146~\rm(1987), 65--96.

\bibitem{TACHIM}
Tachim-Medjo,~T.: \emph{On the uniqueness of $z$-weak solutions of the three-dimensional primitive equations of the ocean}, Nonlinear Anal. Real World Appl., \bf11~\rm(2010), 1413--1421.

\bibitem{TEMNSBOOK}
Temam,~R: \emph{Navier-Stokes Equations. Theory and Numerical Analysis}, Revised edition, Studies in Mathematics and its Applications, 2., North-Holland Publishing Co., Amsterdam-New York, 1979.

\bibitem{TZ04}
Temam,~R., Ziane,~M.: \emph{Some mathematical problems in geohpysical fluid dynamics},
Elsevier: Handbook of Mathematical Fluid Dynamics, \bf3~\rm(2004), 535--657.

\bibitem{VALLIS}
Vallis,~G.~K.: \emph{Atmospheric and Oceanic Fluid Dynamics,} Cambridge Univ. Press, 2006.

\bibitem{WP}
Washington,~W.~M., Parkinson,~C.~L.: \emph{An Introduction to Three Dimensional Climate Modeling}, Oxford University Press, Oxford, 1986.

\bibitem{TKW}
Wong,~T.~K.: \emph{Blowup of solutions of the hydrostatic Euler equations.}, Proc. Amer. Math. Soc., \bf143~\rm(2015), 1119--1125.

\bibitem{ZENG}
Zeng,~Q.~C.: \emph{Mathematical and Physical Foundations of Numerical Weather Prediction,} Science Press, Beijing, 1979.
\end{thebibliography}
\end{document}